\documentclass[11pt,reqno]{amsart}

\usepackage{latexsym,amsmath,amsthm,amsxtra,xypic}

\usepackage[T1]{fontenc}
\usepackage{lmodern}
\usepackage{microtype}
\usepackage{mathtools}
\usepackage{amssymb}
\usepackage{mathrsfs}
\usepackage{booktabs}
\usepackage{enumitem}
\usepackage{xcolor}
\usepackage[colorlinks=true,linkcolor=blue!55!black,
  citecolor=green!45!black,urlcolor=blue!60!black]{hyperref}
\usepackage[nameinlink,capitalise,noabbrev]{cleveref}
\usepackage[margin=0.9in]{geometry}

\hypersetup{
  pdftitle={Extending Structures for n-Lie Algebras},
  pdfauthor={Tao Zhang},
  pdfsubject={Action-form and unrestricted unified products, non-abelian
    extensions, matched pairs, complements, and explicit examples for n-Lie algebras},
  pdfkeywords={n-Lie algebra, unified product, extending structure,
    action form, crossed product, matched pair, deformation map, explicit example}
}

\allowdisplaybreaks
\numberwithin{equation}{section}

\newtheorem{theorem}{Theorem}[section]
\newtheorem{proposition}[theorem]{Proposition}
\newtheorem{lemma}[theorem]{Lemma}
\newtheorem{corollary}[theorem]{Corollary}
\newtheorem{definition}[theorem]{Definition}
\newtheorem{example}[theorem]{Example}
\newtheorem{remark}[theorem]{Remark}

\renewcommand{\L}{\mathcal{L}}
\newcommand{\End}{\mathrm{End}}
\newcommand{\lam}{\lambda}
\newcommand{\E}{\mathrm{E}}
\newcommand{\Ext}{\mathrm{Ext}}
\newcommand{\si}{\sigma}
\newcommand{\hg}{\hat{\mathfrak g}}

\newcommand{\g}{\mathfrak{g}}
\newcommand{\h}{\mathfrak{h}}
\newcommand{\Extd}{\operatorname{Extd}}
\newcommand{\Hom}{\operatorname{Hom}}
\newcommand{\GL}{\operatorname{GL}}
\newcommand{\id}{\operatorname{id}}

\newcommand{\Ker}{\operatorname{Ker}}
\newcommand{\pr}{\operatorname{pr}}
\newcommand{\ad}{\operatorname{ad}}
\newcommand{\natprod}{\mathbin{\natural}}
\newcommand{\bowprod}{\mathbin{\bowtie}}

\title[Extending Structures for $n$-Lie Algebras]
{Extending Structures for $n$-Lie Algebras}
\author{Tao Zhang}
\address{School of Mathematics and Statistics (School of Cryptology), Henan Normal University,
Xinxiang 453007, P. R. China}
\email{zhangtao@htu.edu.cn}
\subjclass[2010]{Primary 17A42; Secondary 17B56, 17A36}
\keywords{$n$-Lie algebra, extending structure, unified product, non-abelian extension, matched pair, deformation}

\begin{document}

\begin{abstract}
  Motivated by the Casas-Loday-Pirashvili maps concerning representations of $n$-Leibniz algebras, we give an intrinsic characterization of the cohomology of $n$-Lie algebras with coefficient in a representation space $V$.
  We investigated the cohomology theory of $n$-Lie algebras in three different ways:
  cohomology theory of Leibniz algebras, infinitesimal deformation and abelian extension.
  In the main part of this paper, we solve the extending problem for $n$-Lie algebras.
  A necessary and sufficient unified-product criterion is obtained.
  The case of crossed products, sparse non-abelian extensions, matched pairs are studied as special cases.
  We also investigate  the  factorizations, deformation maps, and complements problem for $n$-Lie algebras.
  An appendix removes the general non-reduced unified product with all intermediate mixed components.
\end{abstract}

\maketitle

\section{Introduction}

The passage from binary Lie brackets to higher skew brackets originates
in Nambu's generalization of Hamiltonian mechanics~\cite{Nambu} and in
Filippov's algebraic formulation of \(n\)-Lie algebras~\cite{Filippov}.
The subject now links higher-order dynamics, Nambu--Poisson geometry,
multilinear algebra, and higher gauge models; see
\cite{Takhtajan,deAzcarragaIzquierdo} for broad accounts of these
connections.  The structural theory was developed through the study of
representations, ideals, simple algebras, and cohomology
\cite{Kasymov,Ling,DaletskiiTakhtajan,Rotkiewicz}.  These developments
make extension problems indispensable: extensions construct new
algebras from known ones, while their equivalence classes measure the
failure of a prescribed subalgebra or quotient to determine the ambient
algebra uniquely.  They also provide the natural setting for cocycles,
factorizations, and deformation problems; compare
\cite{AD,AfiBasdouri,Arfa}.
Matched-pair and bicrossed-product in factorization theory explain why the same framework also controls complementary subalgebras, see \cite{MajidMatched,MajidPhysics}.

On the other hand, the algebraic theory of $n$-Lie algebras have been studied by many authors, see \cite{BLW,DaletskiiTakhtajan,Ling,Gau,Kasymov,Takhtajan}.
Specially, the (co)homology theory for $n$-Lie algebras was introduced by L. Takhtajan in \cite{Takhtajan,DaletskiiTakhtajan} and by P. Gautheron in \cite{Gau},
which is adapted to the study of formal differential geometry and of formal deformations of Nambu structures.
Abelian extensions and non-abelian extensions of $n$-Lie algebras were studied in \cite{BLW} and \cite{AfiBasdouri,SongMakhloufTang}.
Deformations and Nijenhuis operators on $n$-Lie algebras were investigated in \cite{LiuSheng}.
The calculation of generating indices of \(n\)-Lie algebras was given in \cite{BaiHanBai}.

A central feature of the cohomological approach is the fundamental Leibniz algebra \(\mathcal L=\bigwedge^{n-1}\mathfrak g\) associated with an \(n\)-Lie algebra \(\mathfrak g\). The relevance of Leibniz algebras to \(n\)-ary cohomology is already present in \cite{DaletskiiTakhtajan} and in the representation-theoretic and cohomological framework of Casas-Loday-Pirashvili \cite{CasasLodayPirashvili}.
The intrinsic connection between the reprensentation of $n$-Lie algebras and the reprensentation of Leibniz algebras for the case of 3-Lie algebras were done in \cite{Zhang,ZhangExtending}.

Building on these foundations, in the first part of this paper, we give a detailed characterization of the Leibniz bimodule structure on \(\Hom(\mathfrak g,V)\) induced by a linear map \(\rho:\bigwedge^{n-1}\mathfrak g\to\End(V)\). We prove that the prescribed left and right actions satisfy the Leibniz bimodule identities precisely when \(\rho\) satisfies the two \(n\)-Lie representation identities.
We then describe the associated cohomology through the fundamental-object cochain complex. Under the natural identification
$$
\Hom\!\left(\mathcal L^{\otimes p}\otimes\mathfrak g,V\right)
\cong
\Hom\!\left(\mathcal L^{\otimes p},\Hom(\mathfrak g,V)\right),
$$
we identify its differential with the Leibniz coboundary determined by these two actions. The contribution is thus an explicit representation-level criterion and a cochain-level realization of the established relationship, rather than the introduction of a new cohomology theory.
In fact, we also study the cohomology of $n$-Lie algebras in three different ways:
by cohomology theory of the derived Leibniz algebras $\L=\wedge^{n-1} \g$, by infinitesimal deformation theory of $n$-Lie algebras,
and by abelian extension of $n$-Lie algebras.
It is proved that a 1-parameter infinitesimal deformation of a $n$-Lie algebra $\mathfrak{g}$  corresponds to a 1-cocycle of
  $\mathfrak{g}$ with the coefficients in the adjoint representation.
  We show that abelian extensions can be classified by the first cohomology group.


Let \(E\) be a vector space containing a fixed \(n\)-Lie algebra
\(\g\).  The extending-structures problem asks for all \(n\)-Lie
brackets on \(E\) for which \(\g\) remains a subalgebra, together with
an effective equivalence relation on the resulting brackets.  For Lie algebras this problem is encoded by the unified product of Agore and
Militaru~\cite{AgoreMilitaru}.
They also introduced and developed for various other varieties of algebras in ~\cite{AgoreMilitaru1,AgoreMilitaru2,AgoreMilitaru4,AgoreMilitaru5}.
The extending problem for \(3\)-Lie algebras was established  in~\cite{ZhangExtending}.
The extending-structures problem generalized the abelian extensions and non-abelian extensions problems since in the exact sequence of vector spaces
\begin{equation}
0\longrightarrow\g\xrightarrow{\,i\,}E
 \xrightarrow{\,p\,}V\longrightarrow0,
\label{eq:extproblem}
\end{equation}
we only need $\g$ to be a subalgebra but $V$ is not, and thus this is not necessarily an exact sequence of \(n\)-Lie algebras.

The rest of this paper is denoted to study extending structures for $n$-Lie algebras.
But it is a hard problem for the \(n\)-ary case since it is essentially different from the case of Lie algebras and most complicated when $n>3$.
The general \(n\)-ary problem has a combinatorial feature absent in the binary and ternary cases.
We resolve this problem on the basis of the following observation.
After choosing \(E=\g\oplus V\), we see that an arbitrary
bracket has two direct-sum components on every space
\begin{equation}
 \bigwedge\nolimits^kV\otimes
 \bigwedge\nolimits^{n-k}\g,
 \qquad 1\leq k\leq n.
\label{eq:intro-mixed-degrees}
\end{equation}
Thus, for \(n>3\), the degrees \(2,\ldots,n-2\) cannot be recovered
from the two extreme mixed degrees.  This observation leads to two
complementary objectives.  First, one needs a tractable action-form
theory for the important class in which the intermediate mixed brackets
vanish.  Second, one needs a complete construction retaining all mixed
degrees, both to delimit the action-form theory and to solve the
unrestricted extending problem.

The main body of this paper addresses the first objective.  Its datum is written as
\[
 \rho_{\g}:\bigwedge^{n-1}\g\to\operatorname{End}(V),
 \qquad
 \rho_V:\bigwedge^{n-1}V\to\operatorname{End}(\g),
\]
together with
\[
 \vartheta_V:V\to\Hom(\bigwedge^{n-1}\g,\g),
 \qquad
 \vartheta_{\g}:\g\to\Hom(\bigwedge^{n-1}V,V),
\]
a twisting map \(\omega:\bigwedge^nV\to\g\), and an alternating
\(V\)-valued \(n\)-operation.  This notation is representation-aligned:
\(\rho_{\g}\) and \(\rho_V\) become the canonical representations of
the two factors whenever the corresponding factor is a subalgebra,
whereas the \(\vartheta\)-maps retain the other extreme components.

The novelty of this approach is not a formal replacement of ternary
variables by \(n\)-tuples.  It separates two mathematically different
levels of the problem: a representation-oriented action form, for which
the Filippov signs remain computable by hand, and an unrestricted form,
which retains the middle mixed degrees that appear only when \(n>3\).
This separation makes it possible to state a concise criterion in the
main text without losing the realization and classification results for
arbitrary extending brackets.

The principal contributions of the paper are the following.
\begin{enumerate}[label=\textup{(\roman*)},leftmargin=2.5em]
\item We find  the relationship between  the Leibniz bimodule structure on \(\Hom(\mathfrak g,V)\) and a linear map \(\rho:\bigwedge^{n-1}\mathfrak g\to\End(V)\) satisfies the two \(n\)-Lie representation identities. An intrinsic characterization of the cohomology of $n$-Lie algebras with coefficient in a representation space $V$ is obtained by using this Leibniz bimodule.

\item We derive a necessary and sufficient unified-product criterion by
substituting arbitrary elements directly into the defining Filippov
identity.  Both direct-sum components and every positional sign are
displayed explicitly.

\item We identify exactly how representations and cocycles arise from
the action datum.  In particular, the abelian specialization recovers
the standard semidirect product and the ordinary one-cocycle identity with respect to a representation.
Also every extreme matched pair gives representations of \(\g\) and \(V\) on each other.

\item We develop the theory of crossed products, sparse non-abelian extensions, matched pairs, extreme factorizations, deformation maps, and the classification of complements within the framework of unified-product.

\item In \Cref{app:full}, we solve the unrestricted problem.  Two components are introduced at every degree in
\eqref{eq:intro-mixed-degrees}; the resulting full unified product is characterized by a direct Filippov calculation and is shown to realize
and classify every extending structure of \(\g\) through \(V\).
\end{enumerate}

The combination of the two theories is useful for both construction and classification of $n$-Lie algebras.  The action form keeps concrete calculations readable
and exposes the representation-theoretic content.  The full datum in the
appendix prevents this simplification from being mistaken for a complete
parametrization when \(n>3\).  In particular, it supplies a precise test
for whether an arbitrary extension admits an extreme-degree complement,
and it makes the ternary case transparent: when \(n=3\), there are no
intermediate degrees, so the action form is already the full six-map unified product as in \cite{ZhangExtending}, but generalized all the results in that paper.

The paper is organized as follows.  \Cref{sec:prelim} recalls the
defining Filippov identity and representations. We explain the cohomology theory of $n$-Lie algebras in three different ways:
  cohomology theory of Leibniz algebras, infinitesimal deformation and abelian extension.
\Cref{sec:unified} defines the action-form unified product, proves its full
criterion directly, and relates its maps to representations and
cocycles.  \Cref{sec:classification} gives the corresponding sparse
realization and equivalence theory.  \Cref{sec:crossed} treats
crossed products and sparse non-abelian extensions, including a direct computation.  \Cref{sec:matched} develops
matched pairs and extreme factorizations.  Finally, \Cref{sec:complements} classifies complements by an explicit four-action deformation equation.
\Cref{app:full} gives the general non-extreme extending datum, unified-product criterion, realization theorem, and classification theorem.

Throughout,  all vector spaces are over a field \( K \) of characteristic zero.  Every map is \( K \)-linear.  Empty
lists are omitted and a sum with no indices is zero. We assume $n\geq 2$ for $n$-Lie algebras.

\section{Notes on $n$-Lie algebras and cohomology theory}\label{sec:prelim}

In this section, we recall some facts about classical cohomology theory for $n$-Lie algebras.
Some concepts and notations such as representations and cocylcles will be used in the following sections.

\begin{definition}\label{def:nlie}
An \emph{\(n\)-Lie algebra} is a vector space \({\g}\) with an alternating
map
\[
 [-,\ldots,-]_{\g}:\bigwedge\nolimits^n{\g}\longrightarrow {\g}
\]
such that, for all \(x_1,\ldots,x_{n-1},y_1,\ldots,y_n\in {\g}\),
\begin{align}
 &[x_1,\ldots,x_{n-1},[y_1,\ldots,y_n]_{\g}]_{\g}
 \notag\\
 &\quad =
 \sum_{i=1}^{n}
 [y_1,\ldots,y_{i-1},
   [x_1,\ldots,x_{n-1},y_i]_{\g},
   y_{i+1},\ldots,y_n]_{\g}.
\label{eq:FI}
\end{align}
\end{definition}

For \(X=(x_1,\ldots,x_{n-1})\in\wedge^{n-1} \g, y_i\in {\g}\), set
\[
 \ad_X(y_i)=[x_1,\ldots,x_{n-1},y_i]_{\g}.
\]
Thus \eqref{eq:FI} says that \(\ad_X\) is a derivation:
\[
 \ad_X([y_1,\ldots,y_n]_{\g}) =
 \sum_{i=1}^{n} [y_1,\ldots,y_{i-1},\ad_X(y_i),  y_{i+1},\ldots,y_n]_{\g}.
\]

Put $\L:=\wedge^{n-1} \g$, which is called fundamental set.
The elements in $\L$ are called fundamental objects.
Define the fundamental product on fundamental object by
\begin{eqnarray}\label{eq:fundamental}
X\circ Y=\sum_{i=1}^{n-1}(y_1,\cdots,[x_1,\cdots,x_{n-1},y_i],\cdots,y_{n-1}).
\end{eqnarray}
In \cite{DaletskiiTakhtajan}, the authors  proved that $\L$ is a Leibniz algebra  satisfying the following Leibniz rule
$$X\circ (Y\circ Z)= (X\circ Y)\circ Z+Y\circ (X\circ Z),$$
and
$$\ad_X\ad_Y (w)-\ad_Y\ad_X (w)=\ad(X\circ Y) (w),$$ 
for all $X,Y,Z\in\L, w\in \g$, i.e. $\ad: \L \to \End(\g)$ is a homomorphism of Leibniz algebras.

Recall that for a Leibniz algebra $\L$, a representation is a vector space $M$
together with two bilinear maps
$$[\cdot,\cdot]_L:\L\times {M}\to {M} \,\,\, \text{and}\,\,\,  [\cdot,\cdot]_R: {M}\times \L\to {M},$$
satisfying the following three axioms
\begin{itemize}
\item[$\bullet$] {\rm(LLM)}\quad  $[X\circ Y, m]_L=[X, [Y, m]_L]_L-[Y,[X, m]_L]_L$,
\item[$\bullet$] {\rm(LML)}\quad  $[m, X\circ Y]_R=[[m, X]_R, Y]_R+[X, [m, Y]_R]_L$,
\item[$\bullet$] {\rm(MLL)}\quad  $[m, X\circ Y]_R=[X,[m, Y]_R]_L-[[X,m]_L, Y]_R,$
\end{itemize}
for all  $X,Y\in\L, m\in {M}$.

By (LML) and (MLL) we also have
\begin{itemize}
\item[$\bullet$] {\rm(MMM)}\quad  $[[m, X]_R, Y]_R+[[X,m]_L, Y]_R=0.$
\end{itemize}
In fact, assume (LLM), one of (LML),(MLL),(MMM) can be derived from the other two.

Given an $n$-Lie algebra $\g$ and a vector space $ {V}$, we define the left action and right action of $\L$ on $\Hom(\g, {V})$ by the following maps
$$[\cdot,\cdot]_L:\L \otimes \Hom(\g, {V})\to \Hom(\g, {V})$$
$$[\cdot,\cdot]_R: \Hom(\g, {V})\otimes \L \to \Hom(\g, {V})$$
by
\begin{eqnarray}
\label{eq:leibniz01}{[(x_1,\cdots,x_{n-1}),\phi]_L}(x_n)&=&\rho(x_1,\cdots,x_{n-1})\phi(x_n)-\phi([x_1,\cdots,x_{n-1},x_n]),\\
\label{eq:leibniz02}{[\phi,(x_1,\cdots,x_{n-1})]_R}(x_n)&=&\phi([x_1,\cdots,x_n])-\sum_{i=1}^n(-1)^{n-i}\rho(x_1,\cdots, \hat{x_i},\cdots, x_n)\phi(x_i),
\end{eqnarray}
for all $\phi\in \Hom(\g, {V}), x_i\in \g$, where $\rho$ is a map from $\L$ to $\End(V)$.
The authors in \cite{CasasLodayPirashvili} give similar maps when $\g$ is an $n$-Leibniz algebra, our maps are inspired by it.

\begin{theorem}\label{thm:rep}
Let $\g$ be an $n$-Lie algebra. Then $\Hom(\g, {V})$ equipped with the above two maps
$[\cdot,\cdot]_L$ and $[\cdot,\cdot]_R$ is a representation of Leibniz algebra $\L$
if and only if the following two conditions are satisfied for $\rho$, $\forall x_i, y_i \in \g$,
\begin{itemize}
\item[$\bullet$]{\rm(R1)}\quad $[\rho(x_1,\cdots,x_{n-1}),\rho(y_1,\cdots,y_{n-1})]=\rho((x_1,\cdots,x_{n-1})\circ (y_1,\cdots,y_{n-1}))$,
\item[$\bullet$]{\rm(R2)}\quad $\rho(x_1,\cdots,x_{n-2},[y_1,\cdots,y_{n}])=$
            ${\sum_{i=1}^n}\, (-1)^{n-i}\rho(y_1,\cdots,\hat{y_i}\cdots, y_{n}) \rho(x_1,\cdots,x_{n-2},y_i)$.
\end{itemize}
\end{theorem}

\begin{proof} For $X=(x_1,\cdots,x_{n-1}),\, Y=(y_1,\cdots,y_{n-1})\in\L, \, y_n\in \g$, first we compute the equality
$$[X\circ Y, \phi]_L(y_n)=[X, [Y, \phi]_L]_L(y_n)-[Y,[X, \phi]_L]_L(y_n).$$
By definition \eqref{eq:leibniz01}, the left hand side is equal to
\begin{eqnarray*}
[X\circ Y, \phi]_L(y_n)=\rho(X\circ Y)\phi(y_n)-\phi(\ad_{X\circ Y}(y_n)),
\end{eqnarray*}
and the right hand side is equal to
\begin{eqnarray*}
&&[X, [Y, \phi]_L]_L(y_n)-[Y,[X, \phi]_L]_L(y_n)\\
&=&\rho(X)[Y, \phi]_L(y_n)-[Y, \phi]_L(\ad_X(y_n))-\rho(Y)[X, \phi]_L(y_n)+[X, \phi]_L(\ad_Y(y_n))\\
&=&\rho(X)\rho(Y)\phi(y_n)-\rho(X)\phi(\ad_Y(y_n))-\rho(Y)\phi(\ad_X(y_n))+\phi(\ad_Y\ad_X(y_n))\\
&&-\rho(Y)\rho(X)\phi(y_n)+\rho(Y)\phi(\ad_X(y_n))+\rho(X)\phi(\ad_Y(y_n))-\phi(\ad_X\ad_Y(y_n))\\
&=&\rho(X)\rho(Y)\phi(y_n)+\phi(\ad_Y\ad_X(y_n))-\rho(Y)\rho(X)\phi(y_n)-\phi(\ad_X\ad_Y(y_n))\\
&=&[\rho(X),\rho(Y)]\phi(y_n)-\phi([\ad_X,\ad_Y](y_n)).
\end{eqnarray*}
Since $\ad: \L \to \End(\g)$ is a homomorphism of Leibniz algebras,
thus (LLM) is valid for $[\cdot,\cdot]_L$ if and only if (R1) is valid for $\rho$.

Next we compute the equality
$$[[\phi, X]_R, Y]_R(y_n)+[[X,\phi]_L, Y]_R(y_n)=0.$$
By \eqref{eq:leibniz01} and \eqref{eq:leibniz02} we have
$${[(x_1,\cdots,x_{n-1}),\phi]_L(w)+[\phi,(x_1,\cdots,x_{n-1})]_R}(w)=-{\sum_{i=1}^{n-1}}\,  (-1)^{n-i}\rho(x_1,\cdots,\hat{x_i},\cdots,x_{n-1},w)\phi(x_i),$$
thus
$${[(x_1,\cdots,x_{n-1}),\phi]_L+[\phi,(x_1,\cdots,x_{n-1})]_R}=-{\sum_{i=1}^{n-1}}\,  (-1)^{n-i}\rho(x_1,\cdots,\hat{x_i},\cdots,x_{n-1},\cdot)\phi(x_i),$$
where we denote $\rho(x_1,\cdots,\hat{x_i},\cdots,x_{n-1},\cdot)\phi(x_i):\g\to  {V}$ by
$$\rho(x_1,\cdots,\hat{x_i},\cdots,x_{n-1},\cdot)\phi(x_i)(w)=\rho(x_1,\cdots,\hat{x_i},\cdots,x_{n-1},w)\phi(x_i).$$
Now replace $x_i$ by $y_i$, $\phi$ by $-{\sum_{i=1}^n}\,  (-1)^{n-i}\rho(x_1,\cdots,\hat{x_i},\cdots,x_{n-1},\cdot)\phi(x_i)$ in \eqref{eq:leibniz02}, then we have
\begin{eqnarray*}
&&[[\phi, X]_R+[X, \phi]_L, Y]_R(y_n)\\
&=&-{\sum_{i=1}^{n-1}}\,  (-1)^{n-i}\rho(x_1,\cdots,\hat{x_i},\cdots,x_{n-1},\cdot)\phi(x_i)([y_1,\cdots,y_n])\\
&&+{\sum_{i=1}^{n-1}}\, (-1)^{n-j}\rho(y_1,\cdots,\hat{y_j},\cdots,y_n)\left({\sum_{i=1}^n}\,  (-1)^{n-i}\rho(x_1,\cdots,\hat{x_i},\cdots,x_{n-1},\cdot)\phi(x_i)(y_j)\right)\\
&=&-{\sum_{i=1}^{n-1}} (-1)^{n-i}\Big(\rho(x_1,\cdots,\hat{x_i},\cdots,x_{n-1},[y_1,\cdots,y_n])\\
&&-{\sum_{i=1}^n}\, (-1)^{n-j}\rho(y_1,\cdots,\hat{y_j},\cdots,y_n)\rho(x_1,\cdots,\hat{x_i},\cdots,x_{n-1},y_j)\Big)\phi(x_i),
\end{eqnarray*}
thus (MMM) is valid for $[\cdot,\cdot]_L$ and $[\cdot,\cdot]_R$ if and only if (R2) is valid for $\rho$.

Let us verify (MLL). The first term on its right-hand side of (MLL) is
\begin{align}
 &[(x_1,\ldots,x_{n-1}),
 [\phi,(y_1,\ldots,y_{n-1})]_R]_L(y_n) \notag\\
 &=\rho(x_1,\ldots,x_{n-1})\phi([y_1,\ldots,y_n]) \notag\\
 &\quad-\rho(x_1,\ldots,x_{n-1})
 \rho(y_1,\ldots,y_{n-1})\phi(y_n) \notag\\
 &\quad-\sum_{j=1}^{n-1}(-1)^{n-j}
 \rho(x_1,\ldots,x_{n-1})
 \rho(y_1,\ldots,\widehat{y_j},\ldots,y_{n-1},y_n)
 \phi(y_j) \notag\\
 &\quad-\phi\bigl([y_1,\ldots,y_{n-1},
 [x_1,\ldots,x_{n-1},y_n]]\bigr) \notag\\
 &\quad+\rho(y_1,\ldots,y_{n-1})
 \phi([x_1,\ldots,x_{n-1},y_n]) \notag\\
 &\quad+\sum_{j=1}^{n-1}(-1)^{n-j}
 \rho(y_1,\ldots,\widehat{y_j},\ldots,y_{n-1},
 [x_1,\ldots,x_{n-1},y_n])\phi(y_j).              \label{eq:MLL-first}
\end{align}
The second term is
\begin{align}
 &[[ (x_1,\ldots,x_{n-1}),\phi]_L,
 (y_1,\ldots,y_{n-1})]_R(y_n) \notag\\
 &=\rho(x_1,\ldots,x_{n-1})\phi([y_1,\ldots,y_n])
 -\phi\bigl([x_1,\ldots,x_{n-1},[y_1,\ldots,y_n]]\bigr) \notag\\
 &\quad-\rho(y_1,\ldots,y_{n-1})
 \rho(x_1,\ldots,x_{n-1})\phi(y_n) \notag\\
 &\quad+\rho(y_1,\ldots,y_{n-1})
 \phi([x_1,\ldots,x_{n-1},y_n]) \notag\\
 &\quad-\sum_{j=1}^{n-1}(-1)^{n-j}
 \rho(y_1,\ldots,\widehat{y_j},\ldots,y_{n-1},y_n)
 \rho(x_1,\ldots,x_{n-1})\phi(y_j) \notag\\
 &\quad+\sum_{j=1}^{n-1}(-1)^{n-j}
 \rho(y_1,\ldots,\widehat{y_j},\ldots,y_{n-1},y_n)
 \phi([x_1,\ldots,x_{n-1},y_j]).                 \label{eq:MLL-second}
\end{align}
Subtracting \eqref{eq:MLL-second} from
\eqref{eq:MLL-first}, we obtain
\begin{align}
 &[(x_1,\ldots,x_{n-1}),
 [\phi,(y_1,\ldots,y_{n-1})]_R]_L(y_n) \notag\\
 &\quad-[[ (x_1,\ldots,x_{n-1}),\phi]_L,
 (y_1,\ldots,y_{n-1})]_R(y_n) \notag\\
 &=-\bigl[\rho(x_1,\ldots,x_{n-1}),
 \rho(y_1,\ldots,y_{n-1})\bigr]\phi(y_n) \notag\\
 &\quad+\phi\bigl([x_1,\ldots,x_{n-1},[y_1,\ldots,y_n]]
 -[y_1,\ldots,y_{n-1},[x_1,\ldots,x_{n-1},y_n]]\bigr) \notag\\
 &\quad-\sum_{j=1}^{n-1}(-1)^{n-j}
 \Bigl\{
 [\rho(x_1,\ldots,x_{n-1}),
 \rho(y_1,\ldots,\widehat{y_j},\ldots,y_{n-1},y_n)] \notag\\
 &\hspace{43mm}
 -\rho(y_1,\ldots,\widehat{y_j},\ldots,y_{n-1},
 [x_1,\ldots,x_{n-1},y_n])
 \Bigr\}\phi(y_j) \notag\\
 &\quad-\sum_{j=1}^{n-1}(-1)^{n-j}
 \rho(y_1,\ldots,\widehat{y_j},\ldots,y_{n-1},y_n)
 \phi([x_1,\ldots,x_{n-1},y_j]).                 \label{eq:MLL-difference}
\end{align}
Now condition (R1) gives
\begin{align}
 &\bigl[\rho(x_1,\ldots,x_{n-1}),
 \rho(y_1,\ldots,\widehat{y_j},\ldots,y_{n-1},y_n)\bigr] \notag\\
 &=\sum_{\substack{1\leq i\leq n-1\\i\neq j}}
 \rho(y_1,\ldots,\widehat{y_j},\ldots,
 [x_1,\ldots,x_{n-1},y_i],\ldots,y_{n-1},y_n) \notag\\
 &\quad+\rho(y_1,\ldots,\widehat{y_j},\ldots,y_{n-1},
 [x_1,\ldots,x_{n-1},y_n]).                       \label{eq:R1-second-use}
\end{align}
Using (R1) and \eqref{eq:R1-second-use} in
\eqref{eq:MLL-difference}, we get
\begin{align}
 &[(x_1,\ldots,x_{n-1}),
 [\phi,(y_1,\ldots,y_{n-1})]_R]_L(y_n) \notag\\
 &\quad-[[ (x_1,\ldots,x_{n-1}),\phi]_L,
 (y_1,\ldots,y_{n-1})]_R(y_n) \notag\\
 &=\sum_{i=1}^{n-1}
 \phi\bigl([y_1,\ldots,y_{i-1},
 [x_1,\ldots,x_{n-1},y_i],y_{i+1},\ldots,y_n]\bigr) \notag\\
 &\quad-\sum_{i=1}^{n-1}
 \rho(y_1,\ldots,y_{i-1},[x_1,\ldots,x_{n-1},y_i],
 y_{i+1},\ldots,y_{n-1})\phi(y_n) \notag\\
 &\quad-\sum_{\substack{1\leq i,j\leq n-1\\i\neq j}}
 (-1)^{n-j}
 \rho(y_1,\ldots,\widehat{y_j},\ldots,
 [x_1,\ldots,x_{n-1},y_i],\ldots,y_{n-1},y_n)
 \phi(y_j) \notag\\  \label{eq:MLL-reduced}
 &\quad-\sum_{i=1}^{n-1}(-1)^{n-i}
 \rho(y_1,\ldots,\widehat{y_i},\ldots,y_{n-1},y_n)
 \phi([x_1,\ldots,x_{n-1},y_i])\\
 &=[\phi,(x_1,\ldots,x_{n-1})\circ
 (y_1,\ldots,y_{n-1})]_R(y_n).     \notag
\end{align}
This proves (MLL).

At last, one can check that (LML) is valid for $[\cdot,\cdot]_L$ and $[\cdot,\cdot]_R$ if and only if (R1) and (R2) are valid for $\rho$.
\end{proof}

\begin{definition}\label{def:representation}
Let \(\g\) be an \(n\)-Lie algebra.  A \emph{representation} of
\(\g\) on a vector space \(M\) is a linear map
\[
 \rho:\bigwedge\nolimits^{n-1}\g\longrightarrow\operatorname{End}(M)
\]
satisfying
\begin{align}
 &[\rho(x_1,\ldots,x_{n-1}),
   \rho(y_1,\ldots,y_{n-1})]
 \notag\\
 &\quad =
 \sum_{i=1}^{n-1}
 \rho(y_1,\ldots,y_{i-1},
 [x_1,\ldots,x_{n-1},y_i]_{\g},
 y_{i+1},\ldots,y_{n-1}),
\label{eq:rep1}\\
 &\rho(x_1,\ldots,x_{n-2},[y_1,\ldots,y_n]_{\g})
 \notag\\
 &\quad =
 \sum_{i=1}^{n}(-1)^{n-i}
 \rho(y_1,\ldots,\widehat{y_i},\ldots,y_n)
 \rho(x_1,\ldots,x_{n-2},y_i).
\label{eq:rep2}
\end{align}
We denote it by $(V, \rho_\g)$ or simply $(V, \rho)$ and  $V$ is called an $\g$-module.
\end{definition}

For example, given an $n$-Lie algebra $\g$, there is a natural {\bf adjoint representation} on itself. The corresponding representation
$\ad(x_1,\cdots,x_{n-1})=\ad_X$ is given by
\begin{eqnarray*}
\ad(x_1,\cdots,x_{n-1})(x_n)=[x_1,\cdots,x_{n-1},x_n].
\end{eqnarray*}

\begin{proposition}
Given a representation $\rho$ of the $n$-Lie algebra $\g$ on the vector space $V$. Define a bracket $\g\oplus V$ by
\begin{eqnarray}
[x_1 + u_1, \cdots, x_n + u_n]=[x_1,\cdots,x_{n}] +{\sum_{i=1}^n}\, (-1)^{n-i} \rho (x_1,\cdots, \hat{x_i},\cdots, x_n)(u_i),
\end{eqnarray}
Then $\g\oplus V$ is an $n$-Lie algebra, which is called the semidirect product of $n$-Lie algebra $\g$ and $V$.
\end{proposition}

Now we define the generalized cochain complex for an $n$-Lie algebra $\g$ with coefficients in $ {V}$ by
$$C^p(\g, {V}):=\Hom\left(\left(\wedge{}^{(n-1)p}\g\right)\otimes\g, {V}\right)
=\Hom\left(\L{}^{p},\Hom(\g, {V})\right)$$
and the coboundary operator
$$d^{p-1}:C^{p-1}(\g, {V})\to C^{p}(\g, {V})$$
by the Loday-Pirashvili cohomology of Leibniz algebras of $\L$ with coefficient in $\Hom(\g,V)$:
\begin{eqnarray*}
&&d^{{p}-1}\omega(X_1,X_2,\cdots,X_{p},w)\\
&=&d^{{p}-1}\omega(X_1,X_2,\cdots,X_{p})(w)\\
&=&\sum_{i=1}^{{p}-1}(-1)^{i+1}[X_i,\omega(X_1,\cdots,\hat{X_i},\cdots,X_{p})]_L(w)+(-1)^{p}[\omega(X_1,\cdots,X_{{p}-1}),X_{p}]_R(w)\\
&&+\sum_{1 \leq i<j \leq {p}}(-1)^i \omega\left({X}_1, \cdots, \widehat{{X}}_i,
\cdots, {X}_{j-1}, {X}_i \circ {X}_j, {X}_{j+1}, \cdots, {X}_{p}\right) (w)\\
\end{eqnarray*}
for all $X_i\in \L=\wedge{}^{n-1}\g,\ w\in \g$. It is proved that $d\circ d=0$.
For more details on cohomology of  Leibniz algebras, see \cite{Loday}.

\begin{theorem}\label{thm:cohomology}
There is a well-defined cohomology of an $n$-Lie algebra $\g$ with coefficients in ${V}$ seen as the cohomology of Leibniz algebra $\L$ with coefficients in $\Hom(\g, {V})$. The cohomology group is denoted by $\mathbf{H}^p(\g,{V}):=\mathbf{H}^p(\L,\Hom(\g, {V}))$.
\end{theorem}
\begin{proof} According to \eqref{eq:leibniz01} and \eqref{eq:leibniz02}, for all $X_i, {X}_{p}=(x_{p}^1, \cdots, x_{p}^{n-1})\in \L=\wedge{}^{n-1}\g,\ w\in \g$, we have
\begin{eqnarray*}
&&d^{{p}-1}\omega(X_1,X_2,\cdots,X_{p},w)\\
&=&\sum_{i=1}^{{p}-1}(-1)^{i+1}[X_i,\omega(X_1,\cdots,\hat{X_i},\cdots,X_{p})]_L(w)+(-1)^{p}[\omega(X_1,\cdots,X_{{p}-1}),X_{p}]_R(w)\\
&&+\sum_{1 \leq i<j \leq {p}}(-1)^i \omega\left({X}_1, \cdots, \widehat{{X}}_i,\cdots, {X}_{j-1}, {X}_i \circ {X}_j, {X}_{j+1}, \cdots, {X}_{p}\right) (w)\\
&=&\sum_{i=1}^{{p}-1}(-1)^{i+1}\left[\rho(X_i)\omega(X_1,\cdots,\hat{X_i},\cdots,X_{p})(w)-\omega(X_1,\cdots,\hat{X_i},\cdots,X_{p})([X_i,w])\right]\\
&&+(-1)^{p}\left[\omega(X_1,\cdots,X_{{p}-1})([X_{p},w])-\sum_{i=1}^n(-1)^{n-i}\rho(x_{p}^1,\cdots, \widehat{x_{p}^i},\cdots,x_{p}^{n-1}, w)\omega(X_1,\cdots,X_{{p}-1})(x_{p}^i)\right]\\
&&+\sum_{1 \leq i<j \leq {p}}(-1)^i \omega\left({X}_1, \cdots, \widehat{{X}}_i,\cdots, {X}_{j-1}, {X}_i \circ {X}_j, {X}_{j+1}, \cdots, {X}_{p}\right) (w)\\
&=&\sum_{i=1}^{{p}-1}(-1)^{i+1}\rho(X_i)\omega(X_1,\cdots,\hat{X_i},\cdots,X_{p},w)\\
&&+\sum_{i=1}^{{p}-1}(-1)^{i}\omega(X_1,\cdots,\hat{X_i},\cdots,X_{p},[X_i,w])
+(-1)^{p}\omega(X_1,\cdots,X_{{p}-1},[X_{p},w])\\
&&+(-1)^{{p}+1}\sum_{i=1}^n(-1)^{n-i}\rho(x_{p}^1,\cdots, \widehat{x_{p}^i},\cdots, x_{p}^{n-1}, w)\omega(X_1,\cdots,X_{{p}-1},x_{p}^i)\\
&&+\sum_{1 \leq i<j \leq {p}}(-1)^i \omega\left({X}_1, \cdots, \widehat{{X}}_i,\cdots, {X}_{j-1}, {X}_i \circ {X}_j, {X}_{j+1}, \cdots, {X}_{p},w\right)\\
&=&\sum_{i=1}^{p}(-1)^{i+1} \rho\left({X}_i\right) \omega\left({X}_1, \cdots,{\widehat{X_i}}, \cdots, {X}_{p}, w\right)+\sum_{i=1}^{p}(-1)^i  \omega\left({X}_1, \cdots,{\widehat{{X}_i}}, \cdots, {X}_{p},\left[{X}_i, w\right]\right) \\
&& +(-1)^{{p}+1}\sum_{i=1}^{n-1}(-1)^{n-i} \rho\left(x_{p}^1, \cdots, \widehat{x_{p}^i}, \cdots, x_{p}^{n-1}, w\right) \omega\left({X}_1, \cdots, {X}_{{p}-1}, x_{p}^i\right)\\
&& +\sum_{1 \leq i<j \leq {p}}(-1)^i \omega\left({X}_1, \cdots, \widehat{{X}}_i,
\cdots, {X}_{j-1}, {X}_i \circ {X}_j, {X}_{j+1}, \cdots, {X}_{p}, w\right) \\
\end{eqnarray*}
The terms on the right-hand side of the last equality are precisely the well-known cohomology of $n$-Lie algebras.
See \cite{Gau,Takhtajan} for similar formulas of the cohomology  of $n$-Lie algebras with adjoint representation of $\g$ on itself but not on $V$.
\end{proof}
\begin{remark}
While no new cohomology appears to arise--both the cohomology of Leibniz algebras and that of
$n$-Lie algebras being already known in the literature--we establish a connection between them, showing that the cohomology of
$n$-Lie algebras is essentially the cohomology of Leibniz algebras with the fundamental set $\L$ acting on the specific representation space $\Hom(\g,V)$.
\end{remark}

According to the above Theorem \ref{thm:cohomology}, a 0-cochain is a map $\nu\in\Hom(\g, {V})$,
a 1-cochain is a map
$\omega\in \Hom\left(\wedge{}^{n-1}\g,\Hom(\g, {V})\right)=\Hom\left(\wedge{}^n\g, {V}\right)$,
and the coboundary operator is give by
\begin{eqnarray}
\label{eq:cobound01}d^{0}\nu(X_1,w)&=&d^{0}\nu(X_1)(w)=-[\nu,X_1]_R(w),\\
\label{eq:cobound02}d^{1}\omega(X_1,X_2,w)&=&[X_1,\omega(X_2)]_L(w)+[\omega(X_1),X_2]_R(w)-\omega(X_1\circ X_2)(w).
\end{eqnarray}

Put $X_1=(x_1,\cdots,x_{n-1})\in \L,\ w=x_n\in\g$ in the equality \eqref{eq:cobound01}, then by \eqref{eq:leibniz02} we have
\begin{eqnarray}\label{eq:1coc}
d^0\nu(x_1,\cdots,x_n)&=&{\sum_{i=1}^n}\, (-1)^{n-i}\rho(x_1,\cdots,\hat{x_i},\cdots,x_n)\nu(x_i)-\nu([x_1, \cdots, x_n]).
\end{eqnarray}
\begin{definition}
Let $\g$ be an $n$-Lie algebra and $(V, \rho)$ be an $\g$-module. Then a map $\nu\in\Hom(\g, {V})$
is called 0-cocycle if and only if
\begin{eqnarray}\label{eq:0coc}
{\sum_{i=1}^n}\, (-1)^{n-i}\rho(x_1,\cdots,\hat{x_i},\cdots,x_n)\nu(x_i)-\nu([x_1, \cdots, x_n])=0,
\end{eqnarray}
and a map $\omega: \wedge^n\g\to V$ is called a 1-coboudary if there exists a map $\nu\in\Hom(\g, {V})$ such that $\omega=d^0\nu$.
\end{definition}
Put $X_1=(x_1,\cdots,x_{n-1})\in \L, X_2=(y_1,\cdots,y_{n-1})\in \L$, $w=y_n\in\g$ in the equality \eqref{eq:cobound02}, then we have
\begin{eqnarray*}
&&d^{1}\omega(x_1,\cdots,x_{n-1},y_1,\cdots,y_{n-1},y_n)\\
&=&[x_1,\cdots,x_{n-1},\omega(y_1,\cdots,y_{n-1})]_L(y_n)+[\omega(x_1,\cdots,x_{n-1}),y_1,\cdots,y_{n-1}]_R(y_n)\\
&&-\omega((x_1,\cdots,x_{n-1})\circ (y_1,\cdots,y_{n-1}))(y_n),
\end{eqnarray*}
where
\begin{eqnarray*}
&&[x_1,\cdots,x_{n-1},\omega(y_1,\cdots,y_{n-1})]_L(y_n)\\
&=&\rho(x_1,\cdots,x_{n-1})\omega(y_1,\cdots,y_{n-1})(y_n)-\omega(y_1,\cdots,y_{n-1})([x_1,\cdots,x_{n-1},y_n])\\
&=&\rho(x_1,\cdots,x_{n-1})\omega(y_1,\cdots,y_{n-1},y_n)-\omega(y_1,\cdots,y_{n-1},[x_1,\cdots,x_{n-1},y_n]),
\end{eqnarray*}
\begin{eqnarray*}
&&[\omega(x_1,\cdots,x_{n-1}),y_1,\cdots,y_{n-1}]_R(y_n)\\
&=&\omega(x_1,\cdots,x_{n-1})([y_1,\cdots,y_{n-1},y_n])-{\sum_{i=1}^n}\, (-1)^{n-i}\rho(y_1,\cdots,\hat{y_i},\cdots, y_n)\omega(x_1,\cdots,x_{n-1})(y_i)\\
&=&\omega(x_1,\cdots,x_{n-1},[y_1,\cdots,y_n])-{\sum_{i=1}^n}\, (-1)^{n-i}\rho(y_1,\cdots,\hat{y_i},\cdots, y_n)\omega(x_1,\cdots,x_{n-1},y_i),
\end{eqnarray*}
and
\begin{eqnarray*}
\omega((x_1,\cdots,x_{n-1})\circ (y_1,\cdots,y_{n-1}))(y_n)&=&\omega(\sum_{i=1}^{n-1}(y_1,\cdots,[x_1,\cdots,x_{n-1},y_i],\cdots,y_{n-1}])(y_n)\\
&=&\sum_{i=1}^{n-1}\omega(y_1,\cdots,[x_1,\cdots,x_{n-1},y_i],\cdots,y_{n-1},y_n).
\end{eqnarray*}
\begin{definition}\label{def:1coc}
Let $\g$ be an $n$-Lie algebra and $(V, \rho)$ be an $\g$-module. Then a map
$\omega: \wedge^n\g\to V$ is called a 1-cocycle if  $\forall x_i,y_i\in \g$,
\begin{eqnarray}\label{eq:one-cocycle}
\nonumber&& \omega(x_1,\cdots,x_{n-1},[y_1, \cdots, y_n]_\g)+\rho(x_1,\cdots,x_{n-1})\omega(y_1,\cdots, y_n)\\
\nonumber&=&{\sum_{i=1}^n}\, \omega(y_1,\cdots,[x_1,\cdots,x_{n-1}, y_i]_\g,\cdots, y_n)\\
&&+{\sum_{i=1}^n}\, (-1)^{n-i}\rho(y_1,\cdots,\hat{y_i},\cdots, y_n)\omega(x_1,\cdots,x_{n-1},y_i).
\end{eqnarray}
\end{definition}

\begin{remark}
Readers may notice that the 1-cocycle for $n$-Lie algebras is in fact the 2-cocycle in the Chevalley-Eilenberg cohomology of Lie algebras when $n=2$.
In fact, that is the key difference between the Chevalley-Eilenberg cohomology of Lie algebras and  Loday-Pirashvili cohomology of Leibniz algebras.
The Loday-Pirashvili cohomology of Lie algebras is very different to the Chevalley-Eilenberg cohomology when Lie algebras were seen as Leibniz algebras.
For more about the concise relationship between the cohomology of  Lie algebras and Leibniz algebras, see \cite{FW}.
\end{remark}
\subsection{Infinitesimal deformations of $n$-Lie algebras}

In this section, we study infinitesimal deformations of $n$-Lie algebras.
First we prove that a 1-parameter infinitesimal deformation of a $n$-Lie algebra $\g$
corresponds to a 1-cocycle of $\g$ with the coefficients in the adjoint representation.
Then we consider when two infinitesimal deformations are equivalent. Finally,
we introduce the notion of Nijenhuis operators to describe trivial deformations.

Let $\g$ be an $n$-Lie algebra, and $\omega:\wedge^{n}\g\to\g$ be a linear map. Consider a $\lambda$-parametrized family of linear operations:
\begin{eqnarray*}
[x_1,\cdots,x_{n}]_{\lam,\omega}&\triangleq& [x_1,\cdots,x_{n}]+ \lambda\omega(x_1,\cdots,x_{n})+o({\lambda^2}).
 \end{eqnarray*}

If all $[\cdot,\cdot,\cdot]_{\lam,\omega}$ endow $\g$ with $n$-Lie algebra structures which is denoted by $\g_{\lam,\omega}$, then we say that $\omega$ generates a
$\lambda$-parameter infinitesimal deformation of the $n$-Lie algebra $\g$.

\begin{theorem}\label{thm:deformation}
$\omega$ generates a $\lambda$-parameter infinitesimal deformation of the $n$-Lie algebra $\g$ is equivalent to\\
(i) $\omega$ itself defines an $n$-Lie algebras structure on $\g$;\\
(ii) $\omega$ is a 1-cocycle of $\g$ with the coefficients in the adjoint representation.
\end{theorem}

\begin{proof}
For the equality
\begin{eqnarray*}
[x_1,\cdots,x_{n-1}, [y_1,\cdots,y_{n}]_{\lam,\omega}]_{\lam,\omega}
&=&{\sum_{i=1}^n}\, [y_1, \cdots,[x_1,\cdots,x_{n-1}, y_i]_{\lam,\omega},\cdots, y_n]_{\lam,\omega},
\end{eqnarray*}
the left hand side is equal to
\begin{eqnarray*}
&&[x_1,\cdots,x_{n-1}, [y_1,\cdots,y_{n}]+\lam\omega(y_1,\cdots,y_{n})]_{\lam,\omega}\\
&=&[x_1,\cdots,x_{n-1}, [y_1,\cdots,y_{n}]]+\lam\omega(x_1,\cdots,x_{n-1}, [y_1,\cdots,y_{n}])\\
&&+[x_1,\cdots,x_{n-1},\lam\omega(y_1,\cdots,y_{n})]+\lam\omega(x_1,\cdots,x_{n-1},\lam\omega(y_1,\cdots,y_{n}))\\
&=&[x_1,\cdots,x_{n-1}, [y_1,\cdots,y_{n}]]\\
&&+\lam\{\omega(x_1,\cdots,x_{n-1}, [y_1,\cdots,y_{n}])+[x_1,\cdots,x_{n-1},\omega(y_1,\cdots,y_{n})]\}\\
&&+\lam^2\omega(x_1,\cdots,x_{n-1},\omega(y_1,\cdots,y_{n})),
\end{eqnarray*}
and the right hand side is equal to
\begin{eqnarray*}
&&{\sum_{i=1}^n}\, [y_1, \cdots,[x_1,\cdots,x_{n-1}, y_i],\cdots, y_n]_{\lam,\omega}+{\sum_{i=1}^n}\, [y_1, \cdots,[x_1,\cdots,x_{n-1}, y_i]_{\lam,\omega},\cdots, y_n]_{\lam,\omega},\\
&=&{\sum_{i=1}^n}\, [y_1, \cdots,[x_1,\cdots,x_{n-1}, y_i],\cdots, y_n]\\
&&+\lam\{{\sum_{i=1}^n}\, \omega(y_1,\cdots,[x_1,\cdots,x_{n-1}, y_i],\cdots, y_n)+{\sum_{i=1}^n}\, [y_1,\cdots,\omega(x_1,\cdots,x_{n-1}, y_i),\cdots, y_n]\}\\
&&+\lam^2\{{\sum_{i=1}^n}\, \omega(y_1,\cdots,\omega(x_1,\cdots,x_{n-1}, y_i),\cdots,y_n)\}.
\end{eqnarray*}
Thus we have
\begin{eqnarray}
\nonumber &&\omega(x_1,\cdots,x_{n-1}, [y_1,\cdots,y_{n}])+[x_1,\cdots,x_{n-1},\omega(y_1,\cdots,y_{n})]\\
\nonumber &=&{\sum_{i=1}^n}\, \omega(y_1,\cdots,[x_1,\cdots,x_{n-1}, y_i],\cdots, y_n)\\
\label{eq:2-coc01}&&+{\sum_{i=1}^n}\, (-1)^{n-i}[y_1,\cdots,\hat{y_i},\cdots, y_n,\omega(x_1,\cdots,x_{n-1}, y_i)],
\end{eqnarray}
and
\begin{eqnarray}
\label{eq:2-coc02} &&\omega(x_1,\cdots,x_{n-1},\omega(y_1,\cdots,y_{n}))={\sum_{i=1}^n}\, \omega(y_1,\cdots,\omega(x_1,\cdots,x_{n-1}, y_i),\cdots,y_n).
\end{eqnarray}
Therefore $\omega$ defines an $n$-Lie algebra structure on $\g$ and $\omega$ is a 1-cocycle of $\g$ with the coefficients in the adjoint representation.
\end{proof}

\begin{remark}\label{rem:no-hidden-conditions}
In the above proof, we don't consider the higher order deformation which include items related $\lambda^2$, $\cdot$, $\lambda^{n-1}$, etc.
The interested reader can refer to existing literature to find interesting research findings, see \cite{LiuSheng}.
\end{remark}

\subsection{Abelian extensions of $n$-Lie algebras}

In this section, we study abelian extensions of $n$-Lie algebras.
We show that associated to any abelian extension, there is a representation and a 1-cocycle.
Furthermore, abelian extensions can be classified by the first cohomology group.

\begin{definition}
 Let $\g, {V}$ and $\hat{\g}$ be $n$-Lie algebras and
$i: {V}\to\hat{\g},~~p:\hat{\g}\to\g$
be homomorphisms. The following sequence of $n$-Lie algebras is a
short exact sequence if $\mathrm{Im}(i)=\mathrm{Ker}(p)$,
$\mathrm{Ker}(i)=0$ and $\mathrm{Im}(p)=\g$.
\begin{equation}\label{diagram:exact}
 \xymatrix{
   0  \ar[r]^{} &  {V} \ar[r]^{i} & \hat{\g} \ar[r]^{p} & \g  \ar[r]^{} & 0 \\
  }
\end{equation}
In this case, we call $\hat{\g}$  an extension of $\g$ by
$ {V}$, and denote it by $\E_{\hat{\g}}$.
It is called an abelian extension if $ {V}$ is an abelian ideal (see \cite{Kasymov}) of $\hat{\g}$.
\end{definition}

A section $\sigma:\g\to\hat{\g}$ of $p:\hat{\g}\to\g$
consists of linear maps
$\sigma:\g\to\hat{\g}$
 such that  $p\circ\sigma=\id_{\g}$.

\begin{definition}
 Two extensions of $n$-Lie algebras
 $\E_{\hat{\g}}:0\to {V}\stackrel{i}{\to}\hat{\g}\stackrel{p}{\to}\g\to0$
 and $\E_{\tilde{\g}}:0\to {V}\stackrel{j}{\to}\tilde{\g}\stackrel{q}{\to}\g\to0$ are equivalent,
 if there exists an $n$-Lie algebras homomorphism $F:\hat{\g}\to\tilde{\g}$  such that the following diagram commutes
\begin{equation}\label{diagram:equivalent}
\xymatrix{
   0  \ar[r]^{} &  {V} \ar[d]_{\id} \ar[r]^{i} & \hat{\g} \ar[d]_{F} \ar[r]^{p} & \g \ar[d]_{\id} \ar[r]^{} & 0 \\
   0 \ar[r]^{} &  {V} \ar[r]^{j} & \tilde{\g} \ar[r]^{q} & \g \ar[r]^{} & 0
   }
\end{equation}
The set of equivalent classes of extensions of $\g$ by $\h$ is denoted by $\Ext(\g, {V})$.
\end{definition}

Let $\hat{\g}$ be an abelian extension of $\g$ by
$ {V}$, and $\sigma:\g\to\hat{\g}$ be a section. Denote by
$$\si(X)=\si(x_1,\cdots,x_{n-1})\triangleq(\si(x_1),\cdots,\si(x_{n-1})),$$
and define $\rho:\wedge^{n-1}\g\to\End( {V})$ by
\begin{equation}\label{eq:rep}
\rho(X)(u)=\rho(x_1,\cdots,x_{n-1})(u)\triangleq[\sigma(x_1),\cdots,\sigma(x_{n-1}),u]=\ad(\si(X))u,
\end{equation}
for all $X=(x_1,\cdots,x_{n-1})\in\wedge^{n-1}\g$, $u\in {V}$.

\begin{lemma}\label{pro:2-modules}
With the above notations, $\rho$ is a representation of $\g$ on ${V}$ and does not depend on the choice of the section $\sigma$.
Moreover,  equivalent abelian extensions give the same representation of $\g$ on ${V}$.
\end{lemma}

\begin{proof}
First, we show that $\rho$ is independent of
the choice of $\sigma$. In fact, if we choose another section $\sigma':\g\to\hg$, then
$$p(\sigma(x_i)-\sigma'(x_i))=x_i-x_i=0
\Longrightarrow\sigma(x_i)-\sigma'(x_i)\in \h\Longrightarrow\sigma'(x_i)=\sigma'(x_i)+u_i$$
for some $u_i\in\h$.

Since we  have $[\cdots,u,v]_{\hat{\g}}=0$
for all $u,v\in {V}$, which implies that
\begin{eqnarray*}
&&[\sigma'(x_1),\cdots,\sigma'(x_{n-1}),w]_{\hat{\g}}\\
&=&[\sigma(x_1)+u_1,\cdots,\sigma(x_{n-1})+u_{n-1},w]_{\hat{\g}}\\
&=&[\sigma(x_1),\cdots,\sigma(x_{n-1})+u_{n-1},w]_{\hat{\g}}+[u_1,\cdots,\sigma(x_{n-1})+u_{n-1},w]_{\hat{\g}}\\
&=&[\sigma(x_1),\cdots,\sigma(x_{n-1}),w]_{\hat{\g}}+\cdots+[\sigma(x_1),\cdots,u_{n-1},w]_{\hat{\g}}\\
&=&[\sigma(x_1),\cdots,\sigma(x_{n-1}),w]_{\hat{\g}},
\end{eqnarray*}
thus $\rho$ is independent on the choice of $\sigma$.

Second, we show that $\rho$ is representation of $\g$.

By the equality
\begin{eqnarray*}
&&[\si(x_1),\cdots,\si(x_{n-1}), [\si(y_1),\cdots,\si y_{n-1},u]] \\
&=&\sum_{i=1}^{n-1} [\si(y_1),\cdots,[\si(x_1),\cdots,\si(x_{n-1}),\si(y_i)],\cdots,\si y_{n-1},u]\\
&&+ [\si(y_1),\cdots,\si y_{n-1}, [\si(x_1),\cdots,\si(x_{n-1}),u]\\
&=&\sum_{i=1}^{n-1} [\si(y_1),\cdots,\si[ x_1,\cdots, x_{n-1}, y_i],\cdots,\si y_{n-1},u]
+ [\si(y_1),\cdots,\si y_{n-1}, [\si(x_1),\cdots,\si(x_{n-1}),u]
\end{eqnarray*}
we have
\begin{eqnarray*}
\rho(x_1,\cdots,x_{n-1})\rho(y_1,\cdots,y_{n-1})u&=&\rho((x_1,\cdots,x_{n-1})\circ (y_1,\cdots,y_{n-1}))u\\
&&+ \rho(y_1,\cdots,y_{n-1})\rho(x_1,\cdots,x_{n-1})u,
\end{eqnarray*}
where we use the fact that
$$\si([x_1,\cdots,x_{n-1},y_i])-[\si(x_1),\cdots,\si(x_{n-1}),\si(y_i)]\in V\cong \mathrm{Ker}(p),$$
and that $ {V}$ is abelian ideal of $\hat{\g}$,
$$[\si(y_1),\cdots,\si[x_1,\cdots,x_{n-1},y_i]-[\si(x_1),\cdots,\si(x_{n-1}),\si(y_i)],\cdots,\si y_{n-1},u]=0,$$
thus we get the condition (R1).

Similarly, by the equality
\begin{eqnarray*}
&&[u,\si x_2,\cdots,\si(x_{n-1}), [\si(y_1),\cdots,\si(y_n)]] ={\sum_{i=1}^n}\,  [\si(y_1),\cdots,[u,\si x_2,\cdots,\si(x_{n-1}),\si(y_i)],\cdots,\si(y_n)],
\end{eqnarray*}
we have
\begin{eqnarray*}
\rho(x_2,\cdots, x_{n-1},[y_1,\cdots,y_{n}])u={\sum_{i=1}^n}\,  (-1)^{n-i}\rho (y_1,\cdots,\hat{y_i},\cdots,y_n)\rho(x_2,\cdots,x_{n-1},y_i)u.
\end{eqnarray*}
thus we get the condition (R2). Therefore we see that $\rho$ is a representation of $\g$ on $\h$.

At last, suppose that $\E_{\hat{\g}}$ and $\E_{\tilde{\g}}$ are equivalent abelian extensions, and $F:\hat{\g}\to\tilde{\g}$ is the $n$-Lie algebra homomorphism satisfying $F\circ i=j$, $q\circ F=p$.
Choose linear sections $\sigma$ and $\sigma'$ of $p$ and $q$, we get $qF\sigma(x_i)=p\sigma(x_i)=x_i=q\sigma'(x_i)$,
then $F\sigma(x_i)-\sigma'(x_i)\in \Ker (q)\cong\h$. Thus, we have
$$
[\sigma(x_1),\cdots,\sigma(x_{n-1}),u]_{\hat{\g}}=[F\sigma(x_1),\cdots,F\sigma(x_{n-1}),u]_{\tilde{\g}}
=[\sigma'(x_1),\cdots,\sigma'(x_{n-1}),u]_{\tilde{\g}}.
$$
Therefore, equivalent abelian extensions give the same $\rho$. The proof is finished.
\end{proof}
\medskip

Let $\sigma:\g\to\hat{\g}$  be a
section of the abelian extension. Define the following map:
\begin{equation}\label{eq:coc}
\omega(x_1,\cdots,x_n)\triangleq[\sigma(x_1),\cdots,,\sigma(x_n)]_{\hat{\g}}-\sigma([x_1,\cdots,x_n]_\g),
\end{equation}
for all $x_1,\cdots,x_n\in\g$.

\begin{lemma}\label{thm:2-cocylce}
Let $0\to {V}{\to}\hat{\g}{\to}\g\to 0$ be an abelian extension of $\g$ by $\h$. Then $\omega$ defined by \eqref{eq:coc} is a 1-cocycle of $\g$ with coefficients in $ {V}$,
where the representation $\rho$ is given by \eqref{eq:rep}.
\end{lemma}

\begin{proof}
By the equality
\begin{eqnarray*}
[\si(x_1),\cdots,\si(x_{n-1}), [\si(y_1),\cdots,\si(y_n)]]
&=&{\sum_{i=1}^n}\,  [\si(y_1),\cdots,[\si(x_1),\cdots,\si(x_{n-1}),\si(y_i)],\cdots,\si(y_n)],
\end{eqnarray*}
we get that the left hand side is equal to
\begin{eqnarray*}
&=&[\si(x_1),\cdots,\si(x_{n-1}), \omega(y_1,\cdots,y_n)+\sigma([y_1,\cdots,y_n])] \\
&=&\rho(x_1,\cdots,x_{n-1})\omega(y_1,\cdots,y_n)+[\si(x_1),\cdots,\si(x_{n-1}),\sigma([y_1,\cdots,y_n])]\\
&=&\rho(x_1,\cdots,x_{n-1})\omega(y_1,\cdots,y_n)+\omega(x_1,\cdots,x_{n-1},[y_1,\cdots,y_n])\\
&&+\sigma([x_1,\cdots,x_{n-1},[y_1,\cdots,y_n]]).
\end{eqnarray*}
Similarily, the right hand side is equal to
\begin{eqnarray*}
&=&{\sum_{i=1}^n}\, (-1)^{n-i}[\si(y_1),\cdots,\widehat{\si(y_i)},\cdots,\si(y_n), [\si(x_1),\cdots,\si(x_{n-1}),\si(y_i)]]\\
&=&{\sum_{i=1}^n}\, (-1)^{n-i}[\si(y_1),\cdots,\widehat{\si(y_i)},\cdots,\si(y_n),,\omega(x_1,\cdots,x_{n-1},y_i)+\sigma([x_1,\cdots, x_{n-1},y_i])] \\
&=&{\sum_{i=1}^n}\, (-1)^{n-i}\rho(y_1,\cdots,\hat{y_i},\cdots,y_n)\omega(x_1,\cdots,x_{n-1},y_i)\\
&&+\omega({\sum_{i=1}^n}\, [y_1,\cdots,[x_1,\cdots,x_{n-1},y_i],\cdots,y_n]])\\
&&+\sigma({\sum_{i=1}^n}\, [y_1,\cdots,[x_1,\cdots,x_{n-1},y_i],\cdots,y_n]])\\
\end{eqnarray*}
Thus we have
\begin{eqnarray*}
&& \omega(x_1,\cdots,x_{n-1},[y_1,\cdots,y_{n}])+\rho(x_1,\cdots,x_{n-1})\omega(y_1,\cdots,y_{n})\\
&=&\omega({\sum_{i=1}^n}\, [y_1,\cdots,[x_1,\cdots,x_{n-1},y_i],\cdots,y_n]])\\
&&+{\sum_{i=1}^n}\, (-1)^{n-i}\rho(y_1,\cdots,\hat{y_i},\cdots,y_n)\omega(x_1,\cdots,x_{n-1},y_i)
\end{eqnarray*}
This is exactly the 1-cocycle condition in Definition \ref{def:1coc}.
\end{proof}
\medskip

Now the $n$-Lie algebra structure on $\hat{\g}$ can be transferred to the $n$-Lie algebra structure on $\g\oplus {V}$ using the 1-cocycle given above.
See \cite{BLW} for a direct computation of the following Lemma.

\begin{lemma}[\cite{BLW}]
Let $\g$ be an $n$-Lie algebra and $(V, \rho)$ be an $\g$-module. If $\omega: \wedge^n \g\to V$ is a 1-cocycle,
then $\g\oplus V$ is an $n$-Lie algebra under the following $n$-bracket:
\begin{eqnarray}
\notag[x_1 + u_1, \cdots, x_n + u_n]_\omega &=&[x_1,\cdots,x_{n}] + \omega(x_1,\cdots,x_{n})\\
&&+{\sum_{i=1}^n}\, (-1)^{n-i} \rho (x_1,\cdots, \hat{x_i},\cdots, x_n)(u_i),
\end{eqnarray}
where $x_1,\cdots,x_{n} \in \g$ and $u_1, \cdots, u_n \in V$. This  is  also called the semidirect product of $n$-Lie algebra $\g$ and $V$.
\end{lemma}

\begin{lemma}\label{thm:2-cocylce}
Two abelian extensions of $n$-Lie algebras
$$0\to {V}{\to}\g\oplus_\omega {V}{\to}\g\to 0\quad\text{and}\quad0\to {V}{\to}\g\oplus_{\omega'} {V}{\to}\g\to 0$$
are equivalent if and only if $\omega$ and $\omega'$ are in the same cohomology class.
\end{lemma}

\begin{proof} Let $F:\g\oplus_\omega\h\to \g\oplus_{\omega'}\h$ be the corresponding homomorphism, then
\begin{eqnarray}\label{hhh}
&&F[x_1,\cdots,x_{n}]_{\omega}=[F(x_1),\cdots,F(x_n)]_{\omega'}.
\end{eqnarray}
Since $F$ is an equivalence of extensions, there exist $\nu:\g\to  {V}$ such that
$$F(x_i+u)=x_i+\nu(x_i)+u,\quad i=1,\cdots,n.$$
The left hand side of \eqref{hhh} is equal to
\begin{eqnarray*}
&=&F_1([x_1,\cdots,x_{n}]+\omega(x_1,\cdots,x_{n}))\\
&=&[x_1,\cdots,x_{n}]+\omega(x_1,\cdots,x_{n})+\nu([x_1,\cdots,x_{n}]),
\end{eqnarray*}
and the right hand side of \eqref{hhh} is equal to
\begin{eqnarray*}
&=&[x_1+\nu(x_1),\cdots,x_n+\nu(x_n)]_{\omega'}\\
&=&[x_1,\cdots,x_n]+\omega'(x_1,\cdots,x_n)+{\sum_{i=1}^n}\, (-1)^{n-i}\rho(x_1,\cdots,\hat{x_i}\cdots,x_n)\nu(x_i).
\end{eqnarray*}
Thus we have
\begin{eqnarray}\label{eq:exact4}
\nonumber(\omega-\omega')(x_1,\cdots,x_{n})&=&{\sum_{i=1}^n}\, (-1)^{n-i}\rho(x_1,\cdots,\hat{x_i}\cdots,x_n)\nu(x_i)-\nu([x_1,\cdots,x_{n}]),
\end{eqnarray}
that is $\omega-\omega'=d\nu$. Therefore $\omega$ and $\omega'$ are in the same cohomology class.
\end{proof}

Using the above lemmas, we obtain the main result of this subsection.
\begin{theorem}
Given a representation $\rho:\g\to \End({V})$,  there is a one-to-one correspondence between equivalence classes of abelian extensions of $n$-Lie algebra $\g$ by ${V}$ and the first cohomology group $\mathbf{H}^1(\g,{V})$.
\end{theorem}

\section{The action-form unified product}
\label{sec:unified}

Let \((\g,[-,\ldots,-]_{\g})\) be an \(n\)-Lie algebra and let \(V\)
be a vector space.  In \(\rho_{\g}\) and \(\rho_V\), the subscript names
the factor supplying the \(n-1\) acting variables.  In \(\vartheta_V\)
and \(\vartheta_{\g}\), it names the factor supplying the unique
distinguished variable.  This convention keeps the two representation
maps visibly distinct from the remaining extreme mixed components.

\begin{definition}\label{def:action-datum}
An \emph{extreme action datum} of \(\g\) through \(V\) is a sextuple
\[
 \Omega=(\rho_{\g},\rho_V,\vartheta_V,\vartheta_{\g},
          \omega,\{- ,\ldots,-\}_V)
\]
consisting of alternating linear maps
\begin{align*}
 \rho_{\g}&:\bigwedge^{n-1}\g\longrightarrow\operatorname{End}(V),
 &\rho_V&:\bigwedge^{n-1}V\longrightarrow
          \operatorname{End}(\g),\\
 \vartheta_V&:V\longrightarrow
          \Hom(\bigwedge^{n-1}\g,\g),
 &\vartheta_{\g}&:\g\longrightarrow
          \Hom(\bigwedge^{n-1}V,V),\\
 \omega&:\bigwedge^nV\longrightarrow\g,
 &\{- ,\ldots,-\}_V&:\bigwedge^nV\longrightarrow V.
\end{align*}
No Filippov identity is initially imposed on \(\{- ,\ldots,-\}_V\),
and none of the four action maps is initially assumed to be a
representation.
\end{definition}

The defining homogeneous brackets are
\begin{align}
 [u,x_2,\ldots,x_n]_{\Omega}
 &=\vartheta_V(u)(x_2,\ldots,x_n)
   +(-1)^{n-1}\rho_{\g}(x_2,\ldots,x_n)u,
\label{eq:def-rho-actions}\\
[u_1,\ldots,u_{n-1},x]_{\Omega}
 &=\rho_V(u_1,\ldots,u_{n-1})x
   +\vartheta_{\g}(x)(u_1,\ldots,u_{n-1}),
\label{eq:def-vartheta-actions}\\
 [u_1,\ldots,u_n]_{\Omega}
 &=\omega(u_1,\ldots,u_n)+\{u_1,\ldots,u_n\}_V.
\label{eq:def-pure-V}
\end{align}
Here and below the first summand of a displayed value belongs to \(\g\)
and the second belongs to \(V\).  Brackets containing between two and
\(n-2\) entries of \(V\) are declared to be zero.  This last declaration
is void when \(n=3\).

\begin{proposition}[Direct expansion of the bracket]
\label{prop:direct-bracket}
There is a unique alternating multilinear operation on \(\g\oplus V\)
with the homogeneous values \eqref{eq:def-rho-actions}--
\eqref{eq:def-pure-V}.  For \(x_i\in\g\) and \(u_i\in V\), it is
\begin{align}
&[x_1+u_1,\ldots,x_n+u_n]_{\Omega}
   =G(x_1,u_1;\ldots;x_n,u_n)
    +H(x_1,u_1;\ldots;x_n,u_n),
\label{eq:action-direct-expansion}
\end{align}
where
\begin{align}
G(x_1,u_1;\ldots;x_n,u_n)
={}&[x_1,\ldots,x_n]_{\g}+\omega(u_1,\ldots,u_n)
\notag\\
&+\sum_{i=1}^{n}(-1)^{i-1}
  \vartheta_V(u_i)(x_1,\ldots,\widehat{x_i},\ldots,x_n)
\notag\\
&+\sum_{i=1}^{n}(-1)^{n-i}
  \rho_V(u_1,\ldots,\widehat{u_i},\ldots,u_n)x_i,
\label{eq:G-component}\\
H(x_1,u_1;\ldots;x_n,u_n)
={}&\{u_1,\ldots,u_n\}_V
\notag\\
&+\sum_{i=1}^{n}(-1)^{n-i}
  \rho_{\g}(x_1,\ldots,\widehat{x_i},\ldots,x_n)u_i
\notag\\
&+\sum_{i=1}^{n}(-1)^{n-i}
  \vartheta_{\g}(x_i)
  (u_1,\ldots,\widehat{u_i},\ldots,u_n).
\label{eq:H-component}
\end{align}
\end{proposition}

\begin{proof}
Let \(x_1,\ldots,x_n\in\g\) and \(u_1,\ldots,u_n\in V\).
Multilinearity expands the bracket into its homogeneous mixed parts:
\begin{align*}
&[x_1+u_1,\ldots,x_n+u_n]_{\Omega}\\
&=[x_1,\ldots,x_n]_{\g}
 +\sum_{i=1}^{n}
 [x_1,\ldots,x_{i-1},u_i,x_{i+1},\ldots,x_n]_{\Omega}\\
&\quad+\sum_{i=1}^{n}
 [u_1,\ldots,u_{i-1},x_i,u_{i+1},\ldots,u_n]_{\Omega}
 +[u_1,\ldots,u_n]_{\Omega},
\end{align*}
where every term containing between two and \(n-2\) entries of \(V\)
vanishes by definition:
\[
[u_1,\ldots,u_k,x_{k+1},\ldots,x_n]_{\Omega}=0
\qquad(2\leq k\leq n-2),
\]
with no such term when \(n=3\).  For a unique \(V\)-entry, moving it
to the first or last distinguished position gives, respectively, the
signs \((-1)^{i-1}\) and \((-1)^{n-i}\).  For a unique \(\g\)-entry,
the move to the last position gives \((-1)^{n-i}\).  Hence
\begin{align*}
&[x_1,\ldots,x_{i-1},u_i,x_{i+1},\ldots,x_n]_{\Omega}\\
&=(-1)^{i-1}\vartheta_V(u_i)
 (x_1,\ldots,\widehat{x_i},\ldots,x_n)
 +(-1)^{n-i}\rho_{\g}
 (x_1,\ldots,\widehat{x_i},\ldots,x_n)u_i,\\
&[u_1,\ldots,u_{i-1},x_i,u_{i+1},\ldots,u_n]_{\Omega}\\
&=(-1)^{n-i}\rho_V
 (u_1,\ldots,\widehat{u_i},\ldots,u_n)x_i
 +(-1)^{n-i}\vartheta_{\g}(x_i)
 (u_1,\ldots,\widehat{u_i},\ldots,u_n),\\
&[u_1,\ldots,u_n]_{\Omega}
 =\omega(u_1,\ldots,u_n)+\{u_1,\ldots,u_n\}_V,
\end{align*}
and substitution gives
\[
[x_1+u_1,\ldots,x_n+u_n]_{\Omega}
 =G(x_1,u_1;\ldots;x_n,u_n)
  +H(x_1,u_1;\ldots;x_n,u_n).
\]
\end{proof}

When the operation in \eqref{eq:action-direct-expansion} is an
\(n\)-Lie bracket, the resulting algebra is denoted by
\(\g\natprod_{\Omega}V\) and is called the \emph{action-form unified
product}.  Its compatibility criterion will now be obtained without a
decomposition by homogeneous degree.

\subsection{Direct calculation with the defining Filippov identity}

Take arbitrary elements
\[
 \begin{gathered}
 x_1,\ldots,x_{n-1},y_1,\ldots,y_n\in\g,\qquad
 u_1,\ldots,u_{n-1},v_1,\ldots,v_n\in V,\\
 z_a=x_a+u_a\quad(1\leq a\leq n-1),\qquad
 w_i=y_i+v_i\quad(1\leq i\leq n).
 \end{gathered}
\]
The two components of the inner bracket on the left-hand side of
\eqref{eq:FI} are
\begin{align}
G_0={}&[y_1,\ldots,y_n]_{\g}+\omega(v_1,\ldots,v_n)
\notag\\
&+\sum_{j=1}^{n}(-1)^{j-1}
 \vartheta_V(v_j)(y_1,\ldots,\widehat{y_j},\ldots,y_n)
\notag\\
&+\sum_{j=1}^{n}(-1)^{n-j}
 \rho_V(v_1,\ldots,\widehat{v_j},\ldots,v_n)y_j,
\label{eq:G0}\\
H_0={}&\{v_1,\ldots,v_n\}_V
\notag\\
&+\sum_{j=1}^{n}(-1)^{n-j}
 \rho_{\g}(y_1,\ldots,\widehat{y_j},\ldots,y_n)v_j
\notag\\
&+\sum_{j=1}^{n}(-1)^{n-j}
 \vartheta_{\g}(y_j)
 (v_1,\ldots,\widehat{v_j},\ldots,v_n).
\label{eq:H0}
\end{align}
Thus \([w_1,\ldots,w_n]_{\Omega}=G_0+H_0\).

For \(1\leq i\leq n\), direct substitution into the inner bracket in
the \(i\)-th summand on the right-hand side of \eqref{eq:FI} gives
\begin{align}
G_i={}&[x_1,\ldots,x_{n-1},y_i]_{\g}
 +\omega(u_1,\ldots,u_{n-1},v_i)
\notag\\
&+\sum_{a=1}^{n-1}(-1)^{a-1}
 \vartheta_V(u_a)
 (x_1,\ldots,\widehat{x_a},\ldots,x_{n-1},y_i)
\notag\\
&+(-1)^{n-1}\vartheta_V(v_i)(x_1,\ldots,x_{n-1})
\notag\\
&+\sum_{a=1}^{n-1}(-1)^{n-a}
 \rho_V(u_1,\ldots,\widehat{u_a},\ldots,u_{n-1},v_i)x_a
\notag\\
&+\rho_V(u_1,\ldots,u_{n-1})y_i,
\label{eq:Gi}\\
H_i={}&\{u_1,\ldots,u_{n-1},v_i\}_V
\notag\\
&+\sum_{a=1}^{n-1}(-1)^{n-a}
 \rho_{\g}(x_1,\ldots,\widehat{x_a},\ldots,x_{n-1},y_i)u_a
 +\rho_{\g}(x_1,\ldots,x_{n-1})v_i
\notag\\
&+\sum_{a=1}^{n-1}(-1)^{n-a}
 \vartheta_{\g}(x_a)
 (u_1,\ldots,\widehat{u_a},\ldots,u_{n-1},v_i)
\notag\\
&+\vartheta_{\g}(y_i)(u_1,\ldots,u_{n-1}).
\label{eq:Hi}
\end{align}
Hence
\[
 [z_1,\ldots,z_{n-1},w_i]_{\Omega}=G_i+H_i.
\]

\begin{theorem}[Action-form unified-product criterion]
\label{thm:reduced-unified-product}
The operation \eqref{eq:action-direct-expansion} is an \(n\)-Lie
bracket on \(\g\oplus V\) if and only if the following two identities
hold for all the elements chosen above:
\begin{align}
&G(x_1,u_1;\ldots;x_{n-1},u_{n-1};G_0,H_0)
\notag\\
&\quad=\sum_{i=1}^{n}
G(y_1,v_1;\ldots;y_{i-1},v_{i-1};G_i,H_i;
  y_{i+1},v_{i+1};\ldots;y_n,v_n),
\label{eq:direct-FI-g}\\
&H(x_1,u_1;\ldots;x_{n-1},u_{n-1};G_0,H_0)
\notag\\
&\quad=\sum_{i=1}^{n}
H(y_1,v_1;\ldots;y_{i-1},v_{i-1};G_i,H_i;
  y_{i+1},v_{i+1};\ldots;y_n,v_n),
\label{eq:direct-FI-V}
\end{align}
where \(G,H,G_0,H_0,G_i,H_i\) are the explicit expressions in
\eqref{eq:G-component}, \eqref{eq:H-component}, and
\eqref{eq:G0}--\eqref{eq:Hi}.
\end{theorem}

\begin{proof}
With \(z_a=x_a+u_a\) and \(w_i=y_i+v_i\), equations
\eqref{eq:G0}--\eqref{eq:H0} first give
\begin{align*}
&[z_1,\ldots,z_{n-1},[w_1,\ldots,w_n]_{\Omega}]_{\Omega}\\
&=[x_1+u_1,\ldots,x_{n-1}+u_{n-1},G_0+H_0]_{\Omega}\\
&=G(x_1,u_1;\ldots;x_{n-1},u_{n-1};G_0,H_0)\\
&\quad+H(x_1,u_1;\ldots;x_{n-1},u_{n-1};G_0,H_0).
\end{align*}
The \(\g\)-component in the last equality is
\begin{align*}
&[x_1,\ldots,x_{n-1},G_0]_{\g}
 +\omega(u_1,\ldots,u_{n-1},H_0)\\
&\quad+\sum_{a=1}^{n-1}(-1)^{a-1}
 \vartheta_V(u_a)
 (x_1,\ldots,\widehat{x_a},\ldots,x_{n-1},G_0)\\
&\quad+(-1)^{n-1}\vartheta_V(H_0)(x_1,\ldots,x_{n-1})\\
&\quad+\sum_{a=1}^{n-1}(-1)^{n-a}
 \rho_V(u_1,\ldots,\widehat{u_a},\ldots,u_{n-1},H_0)x_a
 +\rho_V(u_1,\ldots,u_{n-1})G_0,
\end{align*}
and its \(V\)-component is
\begin{align*}
&\{u_1,\ldots,u_{n-1},H_0\}_V\\
&\quad+\sum_{a=1}^{n-1}(-1)^{n-a}
 \rho_{\g}(x_1,\ldots,\widehat{x_a},\ldots,x_{n-1},G_0)u_a
 +\rho_{\g}(x_1,\ldots,x_{n-1})H_0\\
&\quad+\sum_{a=1}^{n-1}(-1)^{n-a}
 \vartheta_{\g}(x_a)
 (u_1,\ldots,\widehat{u_a},\ldots,u_{n-1},H_0)
 +\vartheta_{\g}(G_0)(u_1,\ldots,u_{n-1}).
\end{align*}

For \(1\leq i\leq n\), equations \eqref{eq:Gi}--\eqref{eq:Hi}
give
\begin{align*}
&[w_1,\ldots,w_{i-1},
 [z_1,\ldots,z_{n-1},w_i]_{\Omega},
 w_{i+1},\ldots,w_n]_{\Omega}\\
&=[y_1+v_1,\ldots,y_{i-1}+v_{i-1},G_i+H_i,
 y_{i+1}+v_{i+1},\ldots,y_n+v_n]_{\Omega}\\
&=G(y_1,v_1;\ldots;y_{i-1},v_{i-1};G_i,H_i;
 y_{i+1},v_{i+1};\ldots;y_n,v_n)\\
&\quad+H(y_1,v_1;\ldots;y_{i-1},v_{i-1};G_i,H_i;
 y_{i+1},v_{i+1};\ldots;y_n,v_n).
\end{align*}
Therefore the \(\g\)-component of the Filippov difference in
\eqref{eq:FI} is
\begin{align*}
&G(x_1,u_1;\ldots;x_{n-1},u_{n-1};G_0,H_0)\\
&\quad-\sum_{i=1}^{n}
G(y_1,v_1;\ldots;y_{i-1},v_{i-1};G_i,H_i;
  y_{i+1},v_{i+1};\ldots;y_n,v_n),
\end{align*}
while its \(V\)-component is
\begin{align*}
&H(x_1,u_1;\ldots;x_{n-1},u_{n-1};G_0,H_0)\\
&\quad-\sum_{i=1}^{n}
H(y_1,v_1;\ldots;y_{i-1},v_{i-1};G_i,H_i;
  y_{i+1},v_{i+1};\ldots;y_n,v_n).
\end{align*}
The direct-sum decomposition shows that the difference vanishes exactly
when both displayed components vanish.  These are
\eqref{eq:direct-FI-g} and \eqref{eq:direct-FI-V}.
\end{proof}

\begin{remark}\label{rem:no-hidden-conditions}
The two equations in \Cref{thm:reduced-unified-product} are the two
components obtained after substituting arbitrary elements directly into
the defining Filippov identity \eqref{eq:FI}.  Setting selected
\(x_a,y_i,u_a,v_i\) equal to zero extracts every homogeneous
compatibility equation, including those forcing the omitted middle
mixed brackets to remain zero.
\end{remark}

\begin{example}[Six-parameter unified product]
\label{ex:examples-six-map}
Let
\[
 \g=\operatorname{span}\{x_1,\ldots,x_n,z\},
 \qquad
 V=\operatorname{span}\{u_0,u_1,\ldots,u_n\}.
\]
The only nonzero bracket of \(\g\) is
\begin{equation}
 [x_1,\ldots,x_{n-1},x_n]_{\g}=x_n.
\label{eq:examples-g-standard}
\end{equation}
Put
\[
 X_0=(x_1,\ldots,x_{n-1}),\qquad
 U_0=(u_1,\ldots,u_{n-1}).
\]
For arbitrary \(a,b,c,d,e,f\in K \), define
\begin{align}
 \rho_{\g}(X_0)u_0&=b u_0,
&
 \vartheta_V(u_0)(X_0)&=(-1)^{n-1}a x_n,
\label{eq:examples-six-one-V}\\
 \rho_V(U_0)z&=c z,
&
 \vartheta_{\g}(z)(U_0)&=d u_n,
\label{eq:examples-six-one-g}\\
 \omega(u_1,\ldots,u_n)&=e z,
&
 \{u_1,\ldots,u_n\}_V&=f u_n.
\label{eq:examples-six-pure-V}
\end{align}
All unlisted values are zero.  Then the bracket
\eqref{eq:action-direct-expansion} makes
\(E_{a,b,c,d,e,f}=\g\natprod V\) an \(n\)-Lie algebra.  Its complete
list of nonzero basis brackets is
\begin{align}
 [x_1,\ldots,x_{n-1},x_n]&=x_n,
\label{eq:examples-six-final-1}\\
 [x_1,\ldots,x_{n-1},u_0]&=a x_n+b u_0,
\label{eq:examples-six-final-2}\\
 [u_1,\ldots,u_{n-1},z]&=c z+d u_n,
\label{eq:examples-six-final-3}\\
 [u_1,\ldots,u_{n-1},u_n]&=e z+f u_n.
\label{eq:examples-six-final-4}
\end{align}
In particular, all six maps are nonzero if \(abcdef\neq0\).
\end{example}

\section{Realization and equivalence of action data}
\label{sec:classification}

Let \(E\) be a vector space containing \(\g\), and choose a complement
\(V\), so that \(E=\g\oplus V\).  Write
\(\pr_{\g}:E\to\g\) and \(\pr_V:E\to V\) for the associated
projections.

\begin{definition}\label{def:extreme-structure}
An \(n\)-Lie bracket on \(E\) containing \(\g\) as a subalgebra is
called \emph{extreme with respect to \(V\)} if
\begin{equation}
 [u_1,\ldots,u_k,x_{k+1},\ldots,x_n]_E=0
 \quad(2\leq k\leq n-2).
\label{eq:extreme-sparsity}
\end{equation}
For \(n=3\), this condition is empty.
\end{definition}

\begin{theorem}[Realization theorem]\label{thm:realization}
Every extreme extending structure on \(E=\g\oplus V\) is an
action-form unified product.  More precisely, it determines the datum
\begin{align}
\vartheta_V(u)(x_2,\ldots,x_n)
 &=\pr_{\g}[u,x_2,\ldots,x_n]_E,
\label{eq:extract-thetaV}\\
\rho_{\g}(x_1,\ldots,x_{n-1})u
 &=\pr_V[x_1,\ldots,x_{n-1},u]_E,
\label{eq:extract-rhog}\\
\rho_V(u_1,\ldots,u_{n-1})x
 &=\pr_{\g}[u_1,\ldots,u_{n-1},x]_E,
\label{eq:extract-rhoV}\\
\vartheta_{\g}(x)(u_1,\ldots,u_{n-1})
 &=\pr_V[u_1,\ldots,u_{n-1},x]_E,
\label{eq:extract-thetag}\\
\omega(u_1,\ldots,u_n)
 &=\pr_{\g}[u_1,\ldots,u_n]_E,
\label{eq:extract-omega}\\
\{u_1,\ldots,u_n\}_V
 &=\pr_V[u_1,\ldots,u_n]_E.
\label{eq:extract-Vbracket}
\end{align}
The identity map of \(\g\oplus V\) identifies the given bracket with
\eqref{eq:action-direct-expansion}.  Conversely, every datum satisfying
\eqref{eq:direct-FI-g}--\eqref{eq:direct-FI-V} defines such an extending
structure.
\end{theorem}

\begin{proof}
Let \(x_1,\ldots,x_n\in\g\) and \(u_1,\ldots,u_n\in V\).
Multilinearity in \(E=\g\oplus V\) gives
\begin{align*}
&[x_1+u_1,\ldots,x_n+u_n]_E\\
&=[x_1,\ldots,x_n]_{\g}
 +\sum_{i=1}^{n}(-1)^{i-1}[u_i,
 x_1,\ldots,\widehat{x_i},\ldots,x_n]_E\\
&\quad+\sum_{k=2}^{n-2}
 \sum_{1\leq i_1<\cdots<i_k\leq n}
 (-1)^{i_1+\cdots+i_k-k(k+1)/2}\\
&\qquad\qquad\cdot
 [u_{i_1},\ldots,u_{i_k},
 x_{j_1},\ldots,x_{j_{n-k}}]_E\\
&\quad+\sum_{i=1}^{n}(-1)^{n-i}
 [u_1,\ldots,\widehat{u_i},\ldots,u_n,x_i]_E
 +[u_1,\ldots,u_n]_E,
\end{align*}
where \(j_1<\cdots<j_{n-k}\) are the complementary indices.  By
\eqref{eq:extreme-sparsity}, every summand in the middle sum is zero.
Taking the two projections and using
\eqref{eq:extract-thetaV}--\eqref{eq:extract-Vbracket} gives
\begin{align*}
\operatorname{pr}_{\g}[x_1+u_1,\ldots,x_n+u_n]_E
&=[x_1,\ldots,x_n]_{\g}
 +\omega(u_1,\ldots,u_n)\\
&\quad+\sum_{i=1}^{n}(-1)^{i-1}\vartheta_V(u_i)
 (x_1,\ldots,\widehat{x_i},\ldots,x_n)\\
&\quad+\sum_{i=1}^{n}(-1)^{n-i}\rho_V
 (u_1,\ldots,\widehat{u_i},\ldots,u_n)x_i=G,\\
\operatorname{pr}_{V}[x_1+u_1,\ldots,x_n+u_n]_E
&=\{u_1,\ldots,u_n\}_V\\
&\quad+\sum_{i=1}^{n}(-1)^{n-i}\rho_{\g}
 (x_1,\ldots,\widehat{x_i},\ldots,x_n)u_i\\
&\quad+\sum_{i=1}^{n}(-1)^{n-i}\vartheta_{\g}(x_i)
 (u_1,\ldots,\widehat{u_i},\ldots,u_n)=H.
\end{align*}
Thus the transported bracket agrees with
\eqref{eq:action-direct-expansion} on every
\(x_1+u_1,\ldots,x_n+u_n\).  Applying \eqref{eq:FI} to arbitrary
\(z_1,\ldots,z_{n-1},w_1,\ldots,w_n\in E\) and invoking
\Cref{thm:reduced-unified-product} shows that
\([-,\ldots,-]_E=[-,\ldots,-]_{\Omega}\) satisfies the Filippov
identity precisely when both \eqref{eq:direct-FI-g} and
\eqref{eq:direct-FI-V} hold.
\end{proof}

We next give the equivalence formulas directly in action notation.  Let
\(\Omega\) and \(\Omega'\) be two action data satisfying
\Cref{thm:reduced-unified-product}; put a prime on every map belonging to
\(\Omega'\).

\begin{theorem}[Homomorphisms fixing \(\g\)]
\label{thm:action-morphisms}
A linear map
\[
 \psi:\g\natprod_{\Omega}V\longrightarrow
       \g\natprod_{\Omega'}V
\]
fixes \(\g\) pointwise if and only if there are linear maps
\(r:V\to\g\) and \(s:V\to V\) such that
\begin{equation}
 \psi(x+u)=x+r(u)+s(u)
\label{eq:psi-rs}
\end{equation}
and the identities below hold.

For \(u\in V\) and \(x_2,\ldots,x_n\in\g\),
\begin{align}
&\vartheta_V(u)(x_2,\ldots,x_n)
 +(-1)^{n-1}r\bigl(\rho_{\g}(x_2,\ldots,x_n)u\bigr)
\notag\\
&\qquad=[r(u),x_2,\ldots,x_n]_{\g}
 +\vartheta'_V(s(u))(x_2,\ldots,x_n),
\label{eq:morphism-one-gpart}\\
&s\bigl(\rho_{\g}(x_2,\ldots,x_n)u\bigr)
 =\rho'_{\g}(x_2,\ldots,x_n)s(u).
\label{eq:morphism-one-Vpart}
\end{align}

For \(2\leq k\leq n-2\), \(u_1,\ldots,u_k\in V\), and
\(x_{k+1},\ldots,x_n\in\g\),
\begin{align}
0={}&[r(u_1),\ldots,r(u_k),x_{k+1},\ldots,x_n]_{\g}
\notag\\
&+\sum_{i=1}^{k}(-1)^{i-1}\vartheta'_V(s(u_i))
 (r(u_1),\ldots,\widehat{r(u_i)},\ldots,r(u_k),
  x_{k+1},\ldots,x_n),
\label{eq:morphism-middle-g}\\
0={}&\sum_{i=1}^{k}(-1)^{n-i}\rho'_{\g}
 (r(u_1),\ldots,\widehat{r(u_i)},\ldots,r(u_k),
  x_{k+1},\ldots,x_n)s(u_i).
\label{eq:morphism-middle-V}
\end{align}

For \(u_1,\ldots,u_{n-1}\in V\) and \(x\in\g\),
\begin{align}
&\rho_V(u_1,\ldots,u_{n-1})x
 +r\bigl(\vartheta_{\g}(x)(u_1,\ldots,u_{n-1})\bigr)
\notag\\
&=[r(u_1),\ldots,r(u_{n-1}),x]_{\g}
\notag\\
&\quad+\sum_{i=1}^{n-1}(-1)^{i-1}\vartheta'_V(s(u_i))
 (r(u_1),\ldots,\widehat{r(u_i)},\ldots,r(u_{n-1}),x)
\notag\\
&\quad+\rho'_V(s(u_1),\ldots,s(u_{n-1}))x,
\label{eq:morphism-nminusone-g}\\
&s\bigl(\vartheta_{\g}(x)(u_1,\ldots,u_{n-1})\bigr)
\notag\\
&=\sum_{i=1}^{n-1}(-1)^{n-i}\rho'_{\g}
 (r(u_1),\ldots,\widehat{r(u_i)},\ldots,r(u_{n-1}),x)s(u_i)
\notag\\
&\quad+\vartheta'_{\g}(x)
 (s(u_1),\ldots,s(u_{n-1})).
\label{eq:morphism-nminusone-V}
\end{align}

Finally, for \(u_1,\ldots,u_n\in V\),
\begin{align}
&\omega(u_1,\ldots,u_n)+r(\{u_1,\ldots,u_n\}_V)
\notag\\
&=[r(u_1),\ldots,r(u_n)]_{\g}
 +\omega'(s(u_1),\ldots,s(u_n))
\notag\\
&\quad+\sum_{i=1}^{n}(-1)^{i-1}\vartheta'_V(s(u_i))
 (r(u_1),\ldots,\widehat{r(u_i)},\ldots,r(u_n))
\notag\\
&\quad+\sum_{i=1}^{n}(-1)^{n-i}\rho'_V
 (s(u_1),\ldots,\widehat{s(u_i)},\ldots,s(u_n))r(u_i),
\label{eq:morphism-pure-g}\\
&s(\{u_1,\ldots,u_n\}_V)
 =\{s(u_1),\ldots,s(u_n)\}'_V
\notag\\
&\quad+\sum_{i=1}^{n}(-1)^{n-i}\rho'_{\g}
 (r(u_1),\ldots,\widehat{r(u_i)},\ldots,r(u_n))s(u_i)
\notag\\
&\quad+\sum_{i=1}^{n}(-1)^{n-i}\vartheta'_{\g}(r(u_i))
 (s(u_1),\ldots,\widehat{s(u_i)},\ldots,s(u_n)).
\label{eq:morphism-pure-V}
\end{align}
Moreover, \(\psi\) is an isomorphism if and only if \(s\) is
invertible.
\end{theorem}

\begin{proof}
Because \(\psi\) fixes each \(x\in\g\), its restriction to \(V\)
has unique \(\g\)- and \(V\)-components.  Thus
\[
\psi|_{\g}=\id_{\g},\qquad
r=\operatorname{pr}_{\g}\psi|_V,\qquad
s=\operatorname{pr}_V\psi|_V,
\qquad
\psi(x+u)=x+r(u)+s(u).
\]
By multilinearity, bracket preservation can be checked on homogeneous
tuples.  First take \(u\in V\) and \(x_2,\ldots,x_n\in\g\).  The
source and target calculations are
\begin{align*}
\psi([u,x_2,\ldots,x_n]_{\Omega})
&=\vartheta_V(u)(x_2,\ldots,x_n)
 +(-1)^{n-1}r(\rho_{\g}(x_2,\ldots,x_n)u)\\
&\quad+(-1)^{n-1}s(\rho_{\g}(x_2,\ldots,x_n)u),\\
[r(u)+s(u),x_2,\ldots,x_n]_{\Omega'}
&=[r(u),x_2,\ldots,x_n]_{\g}
 +\vartheta'_V(s(u))(x_2,\ldots,x_n)\\
&\quad+(-1)^{n-1}\rho'_{\g}(x_2,\ldots,x_n)s(u),
\end{align*}
The equality
\[
\psi([u,x_2,\ldots,x_n]_{\Omega})
=[\psi(u),x_2,\ldots,x_n]_{\Omega'}
\]
holds exactly when its \(\g\)-components and \(V\)-components agree.
These two component comparisons give
\eqref{eq:morphism-one-gpart} and \eqref{eq:morphism-one-Vpart},
respectively.
Next let \(2\leq k\leq n-2\), take \(u_1,\ldots,u_k\in V\) and
\(x_{k+1},\ldots,x_n\in\g\).  The source bracket is zero, whereas
the target bracket is
\begin{align*}
&\psi([u_1,\ldots,u_k,x_{k+1},\ldots,x_n]_{\Omega})=0,
 \qquad 2\leq k\leq n-2,\\
&[r(u_1)+s(u_1),\ldots,r(u_k)+s(u_k),
 x_{k+1},\ldots,x_n]_{\Omega'}\\
&=[r(u_1),\ldots,r(u_k),x_{k+1},\ldots,x_n]_{\g}\\
&\quad+\sum_{i=1}^{k}(-1)^{i-1}\vartheta'_V(s(u_i))
 (r(u_1),\ldots,\widehat{r(u_i)},\ldots,r(u_k),
 x_{k+1},\ldots,x_n)\\
&\quad+\sum_{i=1}^{k}(-1)^{n-i}\rho'_{\g}
 (r(u_1),\ldots,\widehat{r(u_i)},\ldots,r(u_k),
 x_{k+1},\ldots,x_n)s(u_i),
\end{align*}
whose \(\g\)- and \(V\)-components vanish exactly when
\begin{align*}
&\psi([u_1,\ldots,u_k,x_{k+1},\ldots,x_n]_{\Omega})\\
&\quad=[\psi(u_1),\ldots,\psi(u_k),
        x_{k+1},\ldots,x_n]_{\Omega'}.
\end{align*}
Comparing its two direct-sum components gives
\eqref{eq:morphism-middle-g} and \eqref{eq:morphism-middle-V}.
For \(u_1,\ldots,u_{n-1}\in V\) and \(x\in\g\), direct expansion
gives
\begin{align*}
&\psi([u_1,\ldots,u_{n-1},x]_{\Omega})\\
&=\rho_V(u_1,\ldots,u_{n-1})x
 +r(\vartheta_{\g}(x)(u_1,\ldots,u_{n-1}))
 +s(\vartheta_{\g}(x)(u_1,\ldots,u_{n-1})),\\
&[r(u_1)+s(u_1),\ldots,r(u_{n-1})+s(u_{n-1}),x]_{\Omega'}\\
&=[r(u_1),\ldots,r(u_{n-1}),x]_{\g}\\
&\quad+\sum_{i=1}^{n-1}(-1)^{i-1}\vartheta'_V(s(u_i))
 (r(u_1),\ldots,\widehat{r(u_i)},\ldots,r(u_{n-1}),x)\\
&\quad+\rho'_V(s(u_1),\ldots,s(u_{n-1}))x\\
&\quad+\sum_{i=1}^{n-1}(-1)^{n-i}\rho'_{\g}
 (r(u_1),\ldots,\widehat{r(u_i)},\ldots,r(u_{n-1}),x)s(u_i)\\
&\quad+\vartheta'_{\g}(x)(s(u_1),\ldots,s(u_{n-1})),
\end{align*}
Thus the equality
\[
\psi([u_1,\ldots,u_{n-1},x]_{\Omega})
=[\psi(u_1),\ldots,\psi(u_{n-1}),x]_{\Omega'}
\]
gives \eqref{eq:morphism-nminusone-g} after comparison in \(\g\), and
it gives \eqref{eq:morphism-nminusone-V} after comparison in \(V\).
Finally, for \(u_1,\ldots,u_n\in V\), the two pure \(V\)-input
expansions are
\begin{align*}
&\psi([u_1,\ldots,u_n]_{\Omega})\\
&=\omega(u_1,\ldots,u_n)+r(\{u_1,\ldots,u_n\}_V)
 +s(\{u_1,\ldots,u_n\}_V),\\
&[r(u_1)+s(u_1),\ldots,r(u_n)+s(u_n)]_{\Omega'}\\
&=[r(u_1),\ldots,r(u_n)]_{\g}
 +\omega'(s(u_1),\ldots,s(u_n))\\
&\quad+\sum_{i=1}^{n}(-1)^{i-1}\vartheta'_V(s(u_i))
 (r(u_1),\ldots,\widehat{r(u_i)},\ldots,r(u_n))\\
&\quad+\sum_{i=1}^{n}(-1)^{n-i}\rho'_V
 (s(u_1),\ldots,\widehat{s(u_i)},\ldots,s(u_n))r(u_i)\\
&\quad+\{s(u_1),\ldots,s(u_n)\}'_V\\
&\quad+\sum_{i=1}^{n}(-1)^{n-i}\rho'_{\g}
 (r(u_1),\ldots,\widehat{r(u_i)},\ldots,r(u_n))s(u_i)\\
&\quad+\sum_{i=1}^{n}(-1)^{n-i}\vartheta'_{\g}(r(u_i))
 (s(u_1),\ldots,\widehat{s(u_i)},\ldots,s(u_n)),
\end{align*}
The equality
\[
\psi([u_1,\ldots,u_n]_{\Omega})
=[\psi(u_1),\ldots,\psi(u_n)]_{\Omega'}
\]
therefore gives \eqref{eq:morphism-pure-g} in the \(\g\)-component
and \eqref{eq:morphism-pure-V} in the \(V\)-component.
The all-\(\g\) tuple is preserved because \(\psi|_{\g}=\id_{\g}\);
the four computations above exhaust all remaining mixed degrees.
Therefore \(\psi\) preserves the bracket if and only if
\eqref{eq:morphism-one-gpart}--\eqref{eq:morphism-pure-V} hold.

With respect to \(\g\oplus V\), the matrix of \(\psi\) is
\[
[\psi]_{\g\oplus V}
=\begin{pmatrix}\id_{\g}&r\\0&s\end{pmatrix}.
\]
If \(s\in\GL(V)\), direct multiplication gives
\[
[\psi]^{-1}
=\begin{pmatrix}\id_{\g}&-rs^{-1}\\0&s^{-1}\end{pmatrix}.
\]
Conversely, if \(\psi\) is invertible, the induced linear map on the
quotient \((\g\oplus V)/\g\cong V\) is \(s\), so \(s\) is invertible.
\end{proof}

\begin{definition}\label{def:equivalent-action-data}
Two compatible action data are called \emph{equivalent} if
\eqref{eq:morphism-one-gpart}--\eqref{eq:morphism-pure-V} hold with
\(s=\id_V\).  They are called \emph{isomorphic} if those identities
hold for some \(r\) and some \(s\in\GL(V)\).
\end{definition}

\begin{corollary}[Classification]\label{cor:classification}
Equivalence classes of compatible action data classify extreme
extending structures on \(E\) up to \(n\)-Lie algebra isomorphisms which
fix \(\g\) and induce the identity on \(E/\g\).  Isomorphism classes of
compatible action data classify them up to arbitrary isomorphisms which
fix \(\g\) pointwise.
\end{corollary}

\begin{proof}
Fix \(E=\g\oplus V\).  For arbitrary \(x_i\in\g\) and \(u_i\in V\),
\Cref{thm:realization} assigns to each extreme bracket the datum
\[
\Omega=(\rho_{\g},\rho_V,\vartheta_V,\vartheta_{\g},
\omega,\{- ,\ldots,-\}_V)
\]
through \eqref{eq:extract-thetaV}--\eqref{eq:extract-Vbracket}.
The direct calculation in \Cref{thm:reduced-unified-product} shows
that this datum satisfies both \eqref{eq:direct-FI-g} and
\eqref{eq:direct-FI-V}.  Conversely, these two identities reconstruct
an extreme bracket by \eqref{eq:action-direct-expansion}.  If
\(\psi(x+u)=x+r(u)+s(u)\), then its induced map on \(E/\g\cong V\)
is \(s\):
\[
\psi_{\Omega,\Omega'}(x+u)=x+r(u)+s(u),\qquad
\overline{\psi}_{\Omega,\Omega'}=s:E/\g\longrightarrow E/\g.
\]
Applying the elementwise identities of
\Cref{thm:action-morphisms} to
\(u,x_2,\ldots,x_n\), to
\(u_1,\ldots,u_k,x_{k+1},\ldots,x_n\), and to the two extreme
input types gives the following conclusions.  Taking \(s=\id_V\),
the identities \eqref{eq:morphism-one-gpart}--\eqref{eq:morphism-pure-V}
are exactly the bracket-preservation equations for an isomorphism
\(\psi\) which fixes \(\g\) and induces \(\id_{E/\g}\).  Allowing any
\(s\in\GL(V)\), the same identities are exactly the equations for an
isomorphism which fixes \(\g\) pointwise.  These are the two
classification relations in the statement.
\end{proof}

\section{Crossed products and sparse non-abelian extensions}
\label{sec:crossed}

Let both \(\g\) and \(V\) be \(n\)-Lie algebras.  In this section the
bracket of \(V\) is denoted by \([- ,\ldots,-]_V\).

\begin{definition}\label{def:crossed-system}
A \emph{sparse crossed system} consists of
\begin{align*}
 \rho_V&:\bigwedge^{n-1}V\longrightarrow
                  \operatorname{End}(\g),\\
 \vartheta_V&:V\longrightarrow\Hom(\bigwedge^{n-1}\g,\g),\\
 \omega&:\bigwedge^nV\longrightarrow\g.
\end{align*}
Its crossed-product operation on \(\g\oplus V\) is
\begin{align}
&[x_1+u_1,\ldots,x_n+u_n]_{\rho_V,\vartheta_V,\omega}
\notag\\
&=[x_1,\ldots,x_n]_{\g}+[u_1,\ldots,u_n]_V
 +\omega(u_1,\ldots,u_n)
\notag\\
&\quad+\sum_{i=1}^{n}(-1)^{i-1}\vartheta_V(u_i)
 (x_1,\ldots,\widehat{x_i},\ldots,x_n)
\notag\\
&\quad+\sum_{i=1}^{n}(-1)^{n-i}\rho_V
 (u_1,\ldots,\widehat{u_i},\ldots,u_n)x_i.
\label{eq:crossed-bracket}
\end{align}
If this is an \(n\)-Lie bracket, the resulting algebra is denoted by
\(\g\#_{\rho_V,\vartheta_V,\omega}V\).
\end{definition}

This is the specialization
\[
 \rho_{\g}=0,\qquad \vartheta_{\g}=0,
 \qquad \{u_1,\ldots,u_n\}_V=[u_1,\ldots,u_n]_V
\]
of the unified product.  The complete criterion can be written as one
direct Filippov calculation.  For this purpose, denote only the
\(\g\)-component of \eqref{eq:crossed-bracket} by
\begin{align}
\mathcal C(x_1,u_1;\ldots;x_n,u_n)
={}&[x_1,\ldots,x_n]_{\g}+\omega(u_1,\ldots,u_n)
\notag\\
&+\sum_{i=1}^{n}(-1)^{i-1}\vartheta_V(u_i)
 (x_1,\ldots,\widehat{x_i},\ldots,x_n)
\notag\\
&+\sum_{i=1}^{n}(-1)^{n-i}\rho_V
 (u_1,\ldots,\widehat{u_i},\ldots,u_n)x_i.
\label{eq:crossed-C}
\end{align}
For the defining Filippov identity, take
\[
 \begin{gathered}
 x_1,\ldots,x_{n-1},y_1,\ldots,y_n\in\g,\qquad
 u_1,\ldots,u_{n-1},v_1,\ldots,v_n\in V,\\
 z_a=x_a+u_a,\qquad w_i=y_i+v_i.
 \end{gathered}
\]
Set
\begin{align}
 \mathcal C_0&=\mathcal C(y_1,v_1;\ldots;y_n,v_n),
&
 \mathcal V_0&=[v_1,\ldots,v_n]_V,
\label{eq:crossed-C0}\\
 \mathcal C_i&=\mathcal C(x_1,u_1;\ldots;
 x_{n-1},u_{n-1};y_i,v_i),
&
 \mathcal V_i&=[u_1,\ldots,u_{n-1},v_i]_V.
\label{eq:crossed-Ci}
\end{align}
Thus
\begin{align}
 [w_1,\ldots,w_n]&=\mathcal C_0+\mathcal V_0,
&
 [z_1,\ldots,z_{n-1},w_i]&=\mathcal C_i+\mathcal V_i.
\label{eq:crossed-V-components}
\end{align}

\begin{theorem}[Direct crossed-product criterion]
\label{thm:crossed-direct}
The operation \eqref{eq:crossed-bracket} is an \(n\)-Lie bracket if and
only if, for all the elements chosen above,
\begin{align}
&\mathcal C(x_1,u_1;\ldots;x_{n-1},u_{n-1};
              \mathcal C_0,\mathcal V_0)
\notag\\
&\quad=\sum_{i=1}^{n}
\mathcal C(y_1,v_1;\ldots;y_{i-1},v_{i-1};
 \mathcal C_i,\mathcal V_i;
 y_{i+1},v_{i+1};\ldots;y_n,v_n).
\label{eq:crossed-direct-condition}
\end{align}
\end{theorem}

\begin{proof}
The left-hand side of \eqref{eq:FI} expands as
\begin{align*}
&[z_1,\ldots,z_{n-1},[w_1,\ldots,w_n]]\\
&=[x_1+u_1,\ldots,x_{n-1}+u_{n-1},
  \mathcal C_0+\mathcal V_0]\\
&=\mathcal C(x_1,u_1;\ldots;x_{n-1},u_{n-1};
             \mathcal C_0,\mathcal V_0)
  +[u_1,\ldots,u_{n-1},\mathcal V_0]_V.
\end{align*}
For \(1\leq i\leq n\), the \(i\)-th summand on the right-hand side is
\begin{align*}
&[w_1,\ldots,w_{i-1},
 [z_1,\ldots,z_{n-1},w_i],
 w_{i+1},\ldots,w_n]\\
&=[y_1+v_1,\ldots,y_{i-1}+v_{i-1},
 \mathcal C_i+\mathcal V_i,
 y_{i+1}+v_{i+1},\ldots,y_n+v_n]\\
&=\mathcal C(y_1,v_1;\ldots;y_{i-1},v_{i-1};
 \mathcal C_i,\mathcal V_i;
 y_{i+1},v_{i+1};\ldots;y_n,v_n)\\
&\quad+[v_1,\ldots,v_{i-1},\mathcal V_i,
 v_{i+1},\ldots,v_n]_V.
\end{align*}
The \(V\)-components agree because
\begin{align*}
&[u_1,\ldots,u_{n-1},[v_1,\ldots,v_n]_V]_V\\
&\quad=\sum_{i=1}^{n}
[v_1,\ldots,v_{i-1},
 [u_1,\ldots,u_{n-1},v_i]_V,
 v_{i+1},\ldots,v_n]_V
\end{align*}
is \eqref{eq:FI} in \(V\).  Equality of the remaining
\(\g\)-components is exactly
\eqref{eq:crossed-direct-condition}.
\end{proof}

Several familiar conditions are visible by setting variables equal to
zero in \eqref{eq:crossed-direct-condition}.

\begin{proposition}[Explicit consequences]
\label{prop:crossed-consequences}
Let \(n>3\) and let
\((\vartheta_V,\rho_V,\omega)\) be a crossed system satisfying
\Cref{thm:crossed-direct}.  Then:

\begin{enumerate}[label=\textup{(\roman*)}]
\item \(\rho_V\) is a representation of \(V\) on \(\g\);

\item \(\omega\) is a one-cocycle for this representation, namely it
satisfies \eqref{eq:one-cocycle}, with
\(\{- ,\ldots,-\}_V=[- ,\ldots,-]_V\);

\item for fixed \(u\in V\) and \(y_2,\ldots,y_{n-1}\in\g\), the map
\[
 D_{u;y_2,\ldots,y_{n-1}}(x)
 =\vartheta_V(u)(x,y_2,\ldots,y_{n-1})
\]
is a derivation of \(\g\):
\begin{align}
&\vartheta_V(u)([x_1,\ldots,x_n]_{\g},y_2,\ldots,y_{n-1})
\notag\\
&\quad=\sum_{i=1}^{n}
 [x_1,\ldots,x_{i-1},
  \vartheta_V(u)(x_i,y_2,\ldots,y_{n-1}),
  x_{i+1},\ldots,x_n]_{\g};
\label{eq:thetaV-derivation}
\end{align}

\item for \(u_1,u_2\in V\) and \(x_3,\ldots,x_n,y_1,\ldots,y_{n-1}
\in\g\),
\begin{align}
0={}&-\vartheta_V(u_2)
 (\vartheta_V(u_1)(y_1,\ldots,y_{n-1}),x_3,\ldots,x_n)
\notag\\
&+\vartheta_V(u_1)
 (\vartheta_V(u_2)(y_1,\ldots,y_{n-1}),x_3,\ldots,x_n);
\label{eq:thetaV-two-V-zero}
\end{align}

\item for \(u_1,\ldots,u_n\in V\) and
\(y_1,\ldots,y_{n-1}\in\g\),
\begin{align}
&[\omega(u_1,\ldots,u_n),y_1,\ldots,y_{n-1}]_{\g}
 +\vartheta_V([u_1,\ldots,u_n]_V)(y_1,\ldots,y_{n-1})
\notag\\
&\quad=\sum_{i=1}^{n}(-1)^{n-i}\rho_V
 (u_1,\ldots,\widehat{u_i},\ldots,u_n)
 \vartheta_V(u_i)(y_1,\ldots,y_{n-1});
\label{eq:omega-rho-theta}\\
&\vartheta_V(v)
 (\omega(u_1,\ldots,u_n),y_2,\ldots,y_{n-1})=0
\label{eq:rho-omega-zero}
\end{align}
for every \(v\in V\) and \(y_2,\ldots,y_{n-1}\in\g\).
\end{enumerate}
\end{proposition}

\begin{proof}
Let \(u_1,\ldots,u_n,v_1,\ldots,v_n\in V\) and
\(x,x_1,\ldots,x_n,y_1,\ldots,y_{n-1}\in\g\).  Put
\(U=(u_1,\ldots,u_{n-1})\) and
\(W=(v_1,\ldots,v_{n-1})\).

First use \eqref{eq:FI} on
\(u_1,\ldots,u_{n-1},v_1,\ldots,v_{n-1},x\).  Its
\(\g\)-components are
\begin{align*}
&\operatorname{pr}_{\g}
 [U,[W,x]]
 =\rho_V(U)\rho_V(W)x,\\
&\operatorname{pr}_{\g}
 \sum_{i=1}^{n-1}
 [v_1,\ldots,[U,v_i],\ldots,v_{n-1},x]
 =\sum_{i=1}^{n-1}
 \rho_V(v_1,\ldots,[U,v_i]_V,\ldots,v_{n-1})x,\\
&\operatorname{pr}_{\g}[W,[U,x]]
 =\rho_V(W)\rho_V(U)x,
\end{align*}
Thus
\begin{align*}
[\rho_V(U),\rho_V(W)]=\sum_{i=1}^{n-1}
 \rho_V(v_1,\ldots,[U,v_i]_V,\ldots,v_{n-1}),
\end{align*}
the first representation identity holds.

For the second representation identity, apply \eqref{eq:FI} to the
outer entries \(u_1,\ldots,u_{n-2},x\) and the inner entries
\(v_1,\ldots,v_n\).  The \(\g\)-component of the left-hand side is
\begin{align*}
&\operatorname{pr}_{\g}
[u_1,\ldots,u_{n-2},x,[v_1,\ldots,v_n]_V]\\
&=-\rho_V(u_1,\ldots,u_{n-2},[v_1,\ldots,v_n]_V)x.
\end{align*}
For the \(i\)-th term on the right, direct substitution gives
\begin{align*}
&\operatorname{pr}_{\g}[u_1,\ldots,u_{n-2},x,v_i]
=-\rho_V(u_1,\ldots,u_{n-2},v_i)x,\\
&\operatorname{pr}_{\g}
[v_1,\ldots,v_{i-1},[u_1,\ldots,u_{n-2},x,v_i],
v_{i+1},\ldots,v_n]\\
&=-(-1)^{n-i}
\rho_V(v_1,\ldots,\widehat{v_i},\ldots,v_n)
\rho_V(u_1,\ldots,u_{n-2},v_i)x.
\end{align*}
After the common minus sign is cancelled, \eqref{eq:FI} gives
\eqref{eq:rep2} for \(\rho_V\).  Together with the
preceding commutator identity, this proves (i).  For the pure \(V\)-entries
\(u_1,\ldots,u_{n-1},v_1,\ldots,v_n\), the two
\(\g\)-components of \eqref{eq:FI} are
\begin{align*}
&\operatorname{pr}_{\g}
 [u_1,\ldots,u_{n-1},[v_1,\ldots,v_n]]\\
&=\omega(u_1,\ldots,u_{n-1},[v_1,\ldots,v_n]_V)
 +\rho_V(u_1,\ldots,u_{n-1})\omega(v_1,\ldots,v_n),\\
&\operatorname{pr}_{\g}
 \sum_{i=1}^{n}
 [v_1,\ldots,[u_1,\ldots,u_{n-1},v_i],\ldots,v_n]\\
&=\sum_{i=1}^{n}
 \omega(v_1,\ldots,[u_1,\ldots,u_{n-1},v_i]_V,\ldots,v_n)\\
&\quad+\sum_{i=1}^{n}(-1)^{n-i}
 \rho_V(v_1,\ldots,\widehat{v_i},\ldots,v_n)
 \omega(u_1,\ldots,u_{n-1},v_i),
\end{align*}
Equality of these two displayed components is precisely
\eqref{eq:one-cocycle}; hence (ii) follows.

For (iii), use the outer entries
\(u,y_2,\ldots,y_{n-1}\) and the inner entries
\(x_1,\ldots,x_n\) in \eqref{eq:FI}.  The \(\g\)-components are
\begin{align*}
&\operatorname{pr}_{\g}
[u,y_2,\ldots,y_{n-1},[x_1,\ldots,x_n]_{\g}]\\
&=\vartheta_V(u)(y_2,\ldots,y_{n-1},
 [x_1,\ldots,x_n]_{\g}),
\end{align*}
and
\begin{align*}
&\sum_{i=1}^{n}\operatorname{pr}_{\g}
[x_1,\ldots,x_{i-1},
 [u,y_2,\ldots,y_{n-1},x_i],
 x_{i+1},\ldots,x_n]\\
&=\sum_{i=1}^{n}
[x_1,\ldots,x_{i-1},
 \vartheta_V(u)(y_2,\ldots,y_{n-1},x_i),
 x_{i+1},\ldots,x_n]_{\g}.
\end{align*}
Alternation gives
\begin{align*}
&\vartheta_V(u)([x_1,\ldots,x_n]_{\g},
 y_2,\ldots,y_{n-1})\\
&\quad=(-1)^{n-2}
\vartheta_V(u)(y_2,\ldots,y_{n-1},
 [x_1,\ldots,x_n]_{\g}),\\
&\vartheta_V(u)(x_i,y_2,\ldots,y_{n-1})
=(-1)^{n-2}\vartheta_V(u)(y_2,\ldots,y_{n-1},x_i).
\end{align*}
Multiplying the equality obtained from \eqref{eq:FI} by
\((-1)^{n-2}\) gives \eqref{eq:thetaV-derivation}.

For (iv), take \(y_1,\ldots,y_{n-1}\) as the outer entries and
\(u_1,u_2,x_3,\ldots,x_n\) as the inner entries in \eqref{eq:FI}.
The inner bracket on the left contains two \(V\)-entries and is zero.
On the right,
\begin{align*}
[y_1,\ldots,y_{n-1},u_j]
=(-1)^{n-1}\vartheta_V(u_j)(y_1,\ldots,y_{n-1})
\qquad(j=1,2).
\end{align*}
The summands corresponding to \(x_3,\ldots,x_n\) still contain two
\(V\)-entries and vanish.  Therefore the \(\g\)-component of
\eqref{eq:FI} is
\begin{align*}
0={}&(-1)^n\vartheta_V(u_2)
 (\vartheta_V(u_1)(y_1,\ldots,y_{n-1}),x_3,\ldots,x_n)\\
&+(-1)^{n-1}\vartheta_V(u_1)
 (\vartheta_V(u_2)(y_1,\ldots,y_{n-1}),x_3,\ldots,x_n).
\end{align*}
Cancelling \((-1)^{n-1}\) gives
\eqref{eq:thetaV-two-V-zero}.

For the first identity in (v), use \(y_1,\ldots,y_{n-1}\) as the
outer entries and \(u_1,\ldots,u_n\) as the inner entries.  The
\(\g\)-component on the left of \eqref{eq:FI} is
\begin{align*}
&\operatorname{pr}_{\g}
[y_1,\ldots,y_{n-1},[u_1,\ldots,u_n]]\\
&=(-1)^{n-1}\Bigl(
 [\omega(u_1,\ldots,u_n),y_1,\ldots,y_{n-1}]_{\g}
\notag\\
&\hspace{24mm}
 +\vartheta_V([u_1,\ldots,u_n]_V)
   (y_1,\ldots,y_{n-1})\Bigr).
\end{align*}
For \(1\leq i\leq n\),
\begin{align*}
[y_1,\ldots,y_{n-1},u_i]
=(-1)^{n-1}\vartheta_V(u_i)(y_1,\ldots,y_{n-1}),
\end{align*}
so the \(\g\)-component on the right is
\begin{align*}
\sum_{i=1}^{n}(-1)^{n-1+n-i}
\rho_V(u_1,\ldots,\widehat{u_i},\ldots,u_n)
\vartheta_V(u_i)(y_1,\ldots,y_{n-1}).
\end{align*}
Cancelling \((-1)^{n-1}\) gives
\eqref{eq:omega-rho-theta}.

Finally, put \(v,y_2,\ldots,y_{n-1}\) in the outer block and
\(u_1,\ldots,u_n\) in the inner block.  The right-hand side of
\eqref{eq:FI} is zero because each inner bracket contains two
\(V\)-entries.  Its left-hand side has \(\g\)-component
\begin{align*}
\vartheta_V(v)(y_2,\ldots,y_{n-1},
 \omega(u_1,\ldots,u_n))=0.
\end{align*}
Moving \(\omega(u_1,\ldots,u_n)\) across the \(n-2\) entries
\(y_2,\ldots,y_{n-1}\) gives \eqref{eq:rho-omega-zero}.
\end{proof}

\subsection{The non-abelian extension problem}

Following the non-abelian extension viewpoint of
\cite{SongMakhloufTang,AfiBasdouri}, consider a short exact sequence of
\(n\)-Lie algebra morphisms
\begin{equation}
0\longrightarrow\g\xrightarrow{\,i\,}E
 \xrightarrow{\,p\,}V\longrightarrow0.
\label{eq:nonabelian-extension}
\end{equation}
A linear section is a map \(\sigma:V\to E\) satisfying
\(p\sigma=\id_V\).  We identify \(\g\) with \(i(\g)\).

\begin{definition}\label{def:sparse-extension}
The extension \eqref{eq:nonabelian-extension} is \emph{sparse relative
to \(\sigma\)} if
\[
 [\sigma(u_1),\ldots,\sigma(u_k),x_{k+1},\ldots,x_n]_E=0
 \quad(2\leq k\leq n-2).
\]
\end{definition}

For such a section define
\begin{align}
\vartheta_V(u)(x_2,\ldots,x_n)
 &=[\sigma(u),x_2,\ldots,x_n]_E,
\label{eq:extension-thetaV}\\
\rho_V(u_1,\ldots,u_{n-1})x
 &=[\sigma(u_1),\ldots,\sigma(u_{n-1}),x]_E,
\label{eq:extension-rhoV}\\
\omega(u_1,\ldots,u_n)
 &=[\sigma(u_1),\ldots,\sigma(u_n)]_E
   -\sigma([u_1,\ldots,u_n]_V).
\label{eq:extension-omega}
\end{align}
All three values lie in \(\g\), since \(\g=\Ker p\) is an ideal.

\begin{theorem}\label{thm:nonabelian-crossed}
The linear isomorphism
\[
 \Phi:\g\oplus V\longrightarrow E,
 \qquad \Phi(x+u)=x+\sigma(u),
\]
is an \(n\)-Lie algebra isomorphism
\[
 \g\#_{\rho_V,\vartheta_V,\omega}V\cong E,
\]
where the three maps are given by
\eqref{eq:extension-thetaV}--\eqref{eq:extension-omega}.  Conversely,
every crossed product gives a sparse non-abelian extension
\[
0\longrightarrow\g\longrightarrow
\g\#_{\rho_V,\vartheta_V,\omega}V
\longrightarrow V\longrightarrow0.
\]
\end{theorem}

\begin{proof}
For \(x\in\g\), \(u\in V\), and \(e\in E\), define
\[
\Phi(x+u)=x+\sigma(u),\qquad
\Phi^{-1}(e)=(e-\sigma p(e))+p(e).
\]
These maps are inverse because \(p\sigma=\id_V\) and
\(\ker p=\g\).  Take \(x_1,\ldots,x_n\in\g\) and
\(u_1,\ldots,u_n\in V\).  Expanding the bracket gives
\begin{align*}
&[\Phi(x_1+u_1),\ldots,\Phi(x_n+u_n)]_E\\
&=[x_1+\sigma(u_1),\ldots,x_n+\sigma(u_n)]_E\\
&=[x_1,\ldots,x_n]_{\g}
 +\sum_{i=1}^{n}(-1)^{i-1}
 [\sigma(u_i),x_1,\ldots,\widehat{x_i},\ldots,x_n]_E\\
&\quad+\sum_{k=2}^{n-2}
 \sum_{1\leq i_1<\cdots<i_k\leq n}
 (-1)^{i_1+\cdots+i_k-k(k+1)/2}\\
&\qquad\qquad\cdot
 [\sigma(u_{i_1}),\ldots,\sigma(u_{i_k}),
 x_{j_1},\ldots,x_{j_{n-k}}]_E\\
&\quad+\sum_{i=1}^{n}(-1)^{n-i}
 [\sigma(u_1),\ldots,\widehat{\sigma(u_i)},\ldots,
 \sigma(u_n),x_i]_E\\
&\quad+[\sigma(u_1),\ldots,\sigma(u_n)]_E\\
\intertext{The middle mixed sum is zero by
\Cref{def:sparse-extension}; therefore}
&=[x_1,\ldots,x_n]_{\g}
 +\sum_{i=1}^{n}(-1)^{i-1}\vartheta_V(u_i)
 (x_1,\ldots,\widehat{x_i},\ldots,x_n)\\
&\quad+\sum_{i=1}^{n}(-1)^{n-i}\rho_V
 (u_1,\ldots,\widehat{u_i},\ldots,u_n)x_i\\
&\quad+\omega(u_1,\ldots,u_n)
 +\sigma([u_1,\ldots,u_n]_V)\\
&=\Phi([x_1+u_1,\ldots,x_n+u_n]_{\rho_V,\vartheta_V,\omega}).
\end{align*}
Thus \(\Phi\) preserves the bracket on arbitrary elements
\(x_i+u_i\), and hence is an \(n\)-Lie algebra isomorphism.

Conversely, in a crossed product the inclusion and projection satisfy
\[
\begin{gathered}
i_{\g}(x)=x+0,\qquad p_V(x+u)=u,\\
\ker p_V=\g,\qquad
p_V([x_1+u_1,\ldots,x_n+u_n])=[u_1,\ldots,u_n]_V.
\end{gathered}
\]
Thus we get a sparse non-abelian extension
\[
0\longrightarrow\g\xrightarrow{i_{\g}} \g\#_{\rho_V,\vartheta_V,\omega}V\xrightarrow{p_V}V\longrightarrow 0.
\]
\end{proof}

The next theorem records the change of section without suppressing the
middle-degree equations.

\begin{theorem}[Direct change-of-section formulas]
\label{thm:crossed-equivalence}
Let \(n>3\), and let
\((\vartheta_V^{(1)},\rho_V^{(1)},\omega_1)\) and
\((\vartheta_V^{(2)},\rho_V^{(2)},\omega_2)\) be crossed systems for
the fixed algebras \(\g\) and \(V\).  The map
\begin{equation}
 \psi_{\xi}(x+u)=x-\xi(u)+u,
 \qquad \xi:V\longrightarrow\g.
\label{eq:crossed-psi-xi}
\end{equation}
is an isomorphism from the second crossed product to the first, inducing
the identity on \(\g\) and \(V\), if and only if the following equations
hold:
\begin{align}
\vartheta_V^{(2)}(u)(x_2,\ldots,x_n)
={}&\vartheta_V^{(1)}(u)(x_2,\ldots,x_n)
 -[\xi(u),x_2,\ldots,x_n]_{\g};
\label{eq:change-thetaV}
\end{align}
for \(2\leq k\leq n-2\),
\begin{align}
0={}&[\xi(u_1),\ldots,\xi(u_k),x_{k+1},\ldots,x_n]_{\g}
\notag\\
&+\sum_{i=1}^{k}(-1)^i\vartheta_V^{(1)}(u_i)
 (\xi(u_1),\ldots,\widehat{\xi(u_i)},\ldots,\xi(u_k),
  x_{k+1},\ldots,x_n);
\label{eq:change-middle}
\end{align}
for \(u_1,\ldots,u_{n-1}\in V\) and \(x\in\g\),
\begin{multline}
\rho_V^{(2)}(u_1,\ldots,u_{n-1})x
=\rho_V^{(1)}(u_1,\ldots,u_{n-1})x
\\
+\sum_{i=1}^{n-1}(-1)^{n+i-3}\vartheta_V^{(1)}(u_i)
 (\xi(u_1),\ldots,\widehat{\xi(u_i)},\ldots,
  \xi(u_{n-1}),x)
\\
+(-1)^{n-1}[\xi(u_1),\ldots,\xi(u_{n-1}),x]_{\g};
\label{eq:change-rhoV}
\end{multline}
and
\begin{align}
&\omega_2(u_1,\ldots,u_n)-\xi([u_1,\ldots,u_n]_V)
\notag\\
&=\omega_1(u_1,\ldots,u_n)
\notag\\
&\quad+\sum_{i=1}^{n}(-1)^{n+i-2}\vartheta_V^{(1)}(u_i)
 (\xi(u_1),\ldots,\widehat{\xi(u_i)},\ldots,\xi(u_n))
\notag\\
&\quad+\sum_{i=1}^{n}(-1)^{n-i+1}\rho_V^{(1)}
 (u_1,\ldots,\widehat{u_i},\ldots,u_n)\xi(u_i)
\notag\\
&\quad+(-1)^n[\xi(u_1),\ldots,\xi(u_n)]_{\g}.
\label{eq:change-omega}
\end{align}
\end{theorem}

\begin{proof}
We check bracket preservation on explicit homogeneous tuples.  Let
\(u\in V\) and \(x_2,\ldots,x_n\in\g\).  Then
\begin{align*}
\psi_{\xi}([u,x_2,\ldots,x_n]_{(2)})
&=\vartheta_V^{(2)}(u)(x_2,\ldots,x_n),\\
[\psi_{\xi}(u),x_2,\ldots,x_n]_{(1)}
&=[-\xi(u)+u,x_2,\ldots,x_n]_{(1)}\\
&=-[\xi(u),x_2,\ldots,x_n]_{\g}
 +\vartheta_V^{(1)}(u)(x_2,\ldots,x_n),
\end{align*}
Hence the bracket-preservation equality
\[
\psi_{\xi}([u,x_2,\ldots,x_n]_{(2)})
=[\psi_{\xi}(u),x_2,\ldots,x_n]_{(1)}
\]
is exactly \eqref{eq:change-thetaV}.
Next take \(u_1,\ldots,u_k\in V\),
\(x_{k+1},\ldots,x_n\in\g\), where \(2\leq k\leq n-2\).
The source bracket is zero and the target bracket expands as
\begin{align*}
&\psi_{\xi}([u_1,\ldots,u_k,x_{k+1},\ldots,x_n]_{(2)})=0,\\
&[\psi_{\xi}(u_1),\ldots,\psi_{\xi}(u_k),
 x_{k+1},\ldots,x_n]_{(1)}\\
&=(-1)^k[\xi(u_1),\ldots,\xi(u_k),x_{k+1},\ldots,x_n]_{\g}\\
&\quad+\sum_{i=1}^{k}
 (-1)^{k-1}(-1)^{i-1}\vartheta_V^{(1)}(u_i)
 (\xi(u_1),\ldots,\widehat{\xi(u_i)},\ldots,\xi(u_k),
 x_{k+1},\ldots,x_n),
\end{align*}
After multiplying the target expression by \((-1)^k\), equality of
the two brackets gives
\begin{align*}
0={}&[\xi(u_1),\ldots,\xi(u_k),
 x_{k+1},\ldots,x_n]_{\g}\\
&+\sum_{i=1}^{k}(-1)^i\vartheta_V^{(1)}(u_i)
 (\xi(u_1),\ldots,\widehat{\xi(u_i)},\ldots,\xi(u_k),
 x_{k+1},\ldots,x_n).
\end{align*}
This is \eqref{eq:change-middle} for the chosen \(k\).
For \(u_1,\ldots,u_{n-1}\in V\) and \(x\in\g\), direct expansion
gives
\begin{align*}
&\psi_{\xi}([u_1,\ldots,u_{n-1},x]_{(2)})\\
&\quad=\rho_V^{(2)}(u_1,\ldots,u_{n-1})x,\\
&[\psi_{\xi}(u_1),\ldots,\psi_{\xi}(u_{n-1}),x]_{(1)}\\
&=\rho_V^{(1)}(u_1,\ldots,u_{n-1})x\\
&\quad+\sum_{i=1}^{n-1}
 (-1)^{n-2}(-1)^{i-1}
 \vartheta_V^{(1)}(u_i)
 (\xi(u_1),\ldots,\widehat{\xi(u_i)},\ldots,
 \xi(u_{n-1}),x)\\
&\quad+(-1)^{n-1}
 [\xi(u_1),\ldots,\xi(u_{n-1}),x]_{\g}\\
&=\rho_V^{(1)}(u_1,\ldots,u_{n-1})x\\
&\quad+\sum_{i=1}^{n-1}(-1)^{n+i-3}
 \vartheta_V^{(1)}(u_i)
 (\xi(u_1),\ldots,\widehat{\xi(u_i)},\ldots,
 \xi(u_{n-1}),x)\\
&\quad+(-1)^{n-1}
 [\xi(u_1),\ldots,\xi(u_{n-1}),x]_{\g}.
\end{align*}
Therefore the bracket-preservation equality
\[
\psi_{\xi}([u_1,\ldots,u_{n-1},x]_{(2)})
=[\psi_{\xi}(u_1),\ldots,\psi_{\xi}(u_{n-1}),x]_{(1)}
\]
is precisely \eqref{eq:change-rhoV}.

Finally, for \(u_1,\ldots,u_n\in V\), the two pure-\(V\) brackets are
\begin{align*}
&\psi_{\xi}([u_1,\ldots,u_n]_{(2)})\\
&=\omega_2(u_1,\ldots,u_n)-\xi([u_1,\ldots,u_n]_V)
 +[u_1,\ldots,u_n]_V,\\
&[\psi_{\xi}(u_1),\ldots,\psi_{\xi}(u_n)]_{(1)}\\
&=(-1)^n[\xi(u_1),\ldots,\xi(u_n)]_{\g}
 +\omega_1(u_1,\ldots,u_n)+[u_1,\ldots,u_n]_V\\
&\quad+\sum_{i=1}^{n}
 (-1)^{n-1}(-1)^{i-1}\vartheta_V^{(1)}(u_i)
 (\xi(u_1),\ldots,\widehat{\xi(u_i)},\ldots,\xi(u_n))\\
&\quad+\sum_{i=1}^{n}
 (-1)^{n-i}(-1)\rho_V^{(1)}
 (u_1,\ldots,\widehat{u_i},\ldots,u_n)\xi(u_i)\\
&=(-1)^n[\xi(u_1),\ldots,\xi(u_n)]_{\g}
 +\omega_1(u_1,\ldots,u_n)+[u_1,\ldots,u_n]_V\\
&\quad+\sum_{i=1}^{n}(-1)^{n+i-2}\vartheta_V^{(1)}(u_i)
 (\xi(u_1),\ldots,\widehat{\xi(u_i)},\ldots,\xi(u_n))\\
&\quad+\sum_{i=1}^{n}(-1)^{n-i+1}\rho_V^{(1)}
 (u_1,\ldots,\widehat{u_i},\ldots,u_n)\xi(u_i).
\end{align*}
Thus the equality
\[
\psi_{\xi}([u_1,\ldots,u_n]_{(2)})
=[\psi_{\xi}(u_1),\ldots,\psi_{\xi}(u_n)]_{(1)}
\]
gives \eqref{eq:change-omega} after comparing the \(\g\)-components;
the \(V\)-components are both \([u_1,\ldots,u_n]_V\).
The all-\(\g\) bracket is fixed automatically, and the preceding
tuples cover every remaining mixed degree.  Thus \(\psi_{\xi}\) is an
\(n\)-Lie algebra morphism exactly when
\eqref{eq:change-thetaV}--\eqref{eq:change-omega} hold.  Its inverse
is \(\psi_{-\xi}\), so it is then an isomorphism.
\end{proof}

\begin{corollary}\label{cor:nonabelian-classification}
Sparse non-abelian extensions of \(V\) by \(\g\), up to isomorphisms
which are the identity on the kernel and quotient, are classified by
crossed systems modulo the transformations
\eqref{eq:change-thetaV}--\eqref{eq:change-omega}.
\end{corollary}

We now specialize \Cref{ex:examples-six-map} by setting
\(\rho_{\g}=\vartheta_{\g}=0\), exactly as required in
\Cref{def:crossed-system}.  This makes \(\g\) an ideal while retaining
both extreme \(\g\)-valued mixed components.

\begin{example}[Three-parameter crossed product]
\label{ex:examples-crossed}
Let \(\g\) and \(V\) have the same bases as in
\Cref{ex:examples-six-map}.  Give \(\g\) the bracket
\eqref{eq:examples-g-standard}, and give \(V\) the bracket
\begin{equation}
 [u_1,\ldots,u_n]_V=u_n,
\label{eq:examples-crossed-V}
\end{equation}
with \(u_0\) central.  For arbitrary
\(\alpha,\beta,\gamma\in K \), define
\begin{align}
 \vartheta_V(u_0)(X_0)&=(-1)^{n-1}\alpha x_n,
\label{eq:examples-crossed-theta}\\
 \rho_V(U_0)z&=\beta z,
\label{eq:examples-crossed-rho}\\
 \omega(u_1,\ldots,u_n)&=\gamma z,
\label{eq:examples-crossed-omega}
\end{align}
and set \(\rho_{\g}=\vartheta_{\g}=0\).  Then
\(E_{\alpha,\beta,\gamma}
 =\g\#_{\rho_V,\vartheta_V,\omega}V\) is the crossed product with
nonzero brackets
\begin{align}
 [x_1,\ldots,x_{n-1},x_n]&=x_n,
\label{eq:examples-crossed-final-1}\\
 [x_1,\ldots,x_{n-1},u_0]&=\alpha x_n,
\label{eq:examples-crossed-final-2}\\
 [u_1,\ldots,u_{n-1},z]&=\beta z,
\label{eq:examples-crossed-final-3}\\
 [u_1,\ldots,u_{n-1},u_n]&=\gamma z+u_n.
\label{eq:examples-crossed-final-4}
\end{align}
Moreover,
\begin{equation}
 0\longrightarrow\g\longrightarrow E_{\alpha,\beta,\gamma}
 \longrightarrow V\longrightarrow0
\label{eq:examples-crossed-exact}
\end{equation}
is a non-abelian extension.
\end{example}

\section{Matched pairs and extreme factorizations}
\label{sec:matched}

Let \(\g\) and \(V\) be \(n\)-Lie algebras.  Motivated by the
factorization and bicrossed-product constructions of
\cite{AD,MajidMatched,MajidPhysics,ZhangExtending}, a matched pair in the
present extreme-degree setting uses four actions
\begin{align*}
 \rho_{\g}&:\bigwedge^{n-1}\g\longrightarrow\operatorname{End}(V),
 &\rho_V&:\bigwedge^{n-1}V\longrightarrow\operatorname{End}(\g),\\
 \vartheta_V&:V\longrightarrow\Hom(\bigwedge^{n-1}\g,\g),
 &\vartheta_{\g}&:\g\longrightarrow\Hom(\bigwedge^{n-1}V,V).
\end{align*}
For arbitrary \(x_i\in\g\) and \(u_i\in V\), define
\begin{align}
\mathcal M_{\g}(x_1,u_1;\ldots;x_n,u_n)
={}&[x_1,\ldots,x_n]_{\g}
\notag\\
&+\sum_{i=1}^{n}(-1)^{i-1}\vartheta_V(u_i)
 (x_1,\ldots,\widehat{x_i},\ldots,x_n)
\notag\\
&+\sum_{i=1}^{n}(-1)^{n-i}\rho_V
 (u_1,\ldots,\widehat{u_i},\ldots,u_n)x_i,
\label{eq:matched-Mg}\\
\mathcal M_V(x_1,u_1;\ldots;x_n,u_n)
={}&[u_1,\ldots,u_n]_V
\notag\\
&+\sum_{i=1}^{n}(-1)^{n-i}\rho_{\g}
 (x_1,\ldots,\widehat{x_i},\ldots,x_n)u_i
\notag\\
&+\sum_{i=1}^{n}(-1)^{n-i}\vartheta_{\g}(x_i)
 (u_1,\ldots,\widehat{u_i},\ldots,u_n).
\label{eq:matched-MV}
\end{align}

\begin{definition}\label{def:matched-pair}
The quadruple
\((\rho_{\g},\rho_V,\vartheta_V,\vartheta_{\g})\) is an
\emph{extreme matched pair} if
\begin{equation}
 [x_1+u_1,\ldots,x_n+u_n]_{\bowprod}
 =\mathcal M_{\g}(x_1,u_1;\ldots;x_n,u_n)
  +\mathcal M_V(x_1,u_1;\ldots;x_n,u_n)
\label{eq:matched-bracket}
\end{equation}
is an \(n\)-Lie bracket.  The resulting \(n\)-Lie algebra is denoted by
\(\g\bowprod V\).
\end{definition}

The next statement gives the complete conditions without naming
homogeneous input types.

\begin{theorem}[Direct matched-pair criterion]
\label{thm:matched-direct}
The four actions form an extreme matched pair if and only if, for all
\[
x_1,\ldots,x_{n-1},y_1,\ldots,y_n\in\g,\qquad
u_1,\ldots,u_{n-1},v_1,\ldots,v_n\in V,
\]
the following two identities hold:
\begin{align}
&\mathcal M_{\g}\bigl(
 x_1,u_1;\ldots;x_{n-1},u_{n-1};
 \mathcal M_{\g}(y_1,v_1;\ldots;y_n,v_n),
 \mathcal M_V(y_1,v_1;\ldots;y_n,v_n)\bigr)
\notag\\
&\quad=\sum_{i=1}^{n}\mathcal M_{\g}\bigl(
 y_1,v_1;\ldots;y_{i-1},v_{i-1};
 \mathcal M_{\g}(x_1,u_1;\ldots;x_{n-1},u_{n-1};y_i,v_i),
\notag\\
&\hspace{20mm}
 \mathcal M_V(x_1,u_1;\ldots;x_{n-1},u_{n-1};y_i,v_i);
 y_{i+1},v_{i+1};\ldots;y_n,v_n\bigr),
\label{eq:matched-direct-g}\\
&\mathcal M_V\bigl(
 x_1,u_1;\ldots;x_{n-1},u_{n-1};
 \mathcal M_{\g}(y_1,v_1;\ldots;y_n,v_n),
 \mathcal M_V(y_1,v_1;\ldots;y_n,v_n)\bigr)
\notag\\
&\quad=\sum_{i=1}^{n}\mathcal M_V\bigl(
 y_1,v_1;\ldots;y_{i-1},v_{i-1};
 \mathcal M_{\g}(x_1,u_1;\ldots;x_{n-1},u_{n-1};y_i,v_i),
\notag\\
&\hspace{20mm}
 \mathcal M_V(x_1,u_1;\ldots;x_{n-1},u_{n-1};y_i,v_i);
 y_{i+1},v_{i+1};\ldots;y_n,v_n\bigr).
\label{eq:matched-direct-V}
\end{align}
\end{theorem}

\begin{proof}
Put
\[
 z_a=x_a+u_a\quad(1\leq a\leq n-1),\qquad
 w_i=y_i+v_i\quad(1\leq i\leq n).
\]
The inner bracket on the left-hand side of \eqref{eq:FI} is
\begin{align*}
[w_1,\ldots,w_n]_{\bowprod}
&=\mathcal M_{\g}(y_1,v_1;\ldots;y_n,v_n)
 +\mathcal M_V(y_1,v_1;\ldots;y_n,v_n).
\end{align*}
Substitution into the outer bracket gives
\begin{align*}
&[z_1,\ldots,z_{n-1},[w_1,\ldots,w_n]_{\bowprod}]_{\bowprod}\\
&=\mathcal M_{\g}\bigl(
 x_1,u_1;\ldots;x_{n-1},u_{n-1};
\notag\\
&\hspace{16mm}\mathcal M_{\g}(y_1,v_1;\ldots;y_n,v_n),
 \mathcal M_V(y_1,v_1;\ldots;y_n,v_n)\bigr)\\
&\quad+\mathcal M_V\bigl(
 x_1,u_1;\ldots;x_{n-1},u_{n-1};
\notag\\
&\hspace{16mm}\mathcal M_{\g}(y_1,v_1;\ldots;y_n,v_n),
 \mathcal M_V(y_1,v_1;\ldots;y_n,v_n)\bigr).
\end{align*}
For \(1\leq i\leq n\), the inner bracket in the \(i\)-th summand is
\begin{align*}
[z_1,\ldots,z_{n-1},w_i]_{\bowprod}
&=\mathcal M_{\g}(x_1,u_1;\ldots;x_{n-1},u_{n-1};y_i,v_i)\\
&\quad+\mathcal M_V(x_1,u_1;\ldots;x_{n-1},u_{n-1};y_i,v_i).
\end{align*}
Consequently,
\begin{align*}
&[w_1,\ldots,w_{i-1},
 [z_1,\ldots,z_{n-1},w_i]_{\bowprod},
 w_{i+1},\ldots,w_n]_{\bowprod}\\
&=\mathcal M_{\g}\bigl(
 y_1,v_1;\ldots;y_{i-1},v_{i-1};
 \mathcal M_{\g}(x_1,u_1;\ldots;x_{n-1},u_{n-1};y_i,v_i),\\
&\hspace{32mm}
 \mathcal M_V(x_1,u_1;\ldots;x_{n-1},u_{n-1};y_i,v_i);
 y_{i+1},v_{i+1};\ldots;y_n,v_n\bigr)\\
&\quad+\mathcal M_V\bigl(
 y_1,v_1;\ldots;y_{i-1},v_{i-1};
 \mathcal M_{\g}(x_1,u_1;\ldots;x_{n-1},u_{n-1};y_i,v_i),\\
&\hspace{32mm}
 \mathcal M_V(x_1,u_1;\ldots;x_{n-1},u_{n-1};y_i,v_i);
 y_{i+1},v_{i+1};\ldots;y_n,v_n\bigr).
\end{align*}
Comparing the \(\g\)-components in \eqref{eq:FI} gives
\eqref{eq:matched-direct-g}; comparing the \(V\)-components gives
\eqref{eq:matched-direct-V}.  Since the chosen elements are arbitrary,
the two component identities are also sufficient for \eqref{eq:FI}.
\end{proof}

\begin{proposition}\label{prop:matched-representations}
For every extreme matched pair,
\(\rho_{\g}\) is a representation of \(\g\) on \(V\), and
\(\rho_V\) is a representation of \(V\) on \(\g\).
\end{proposition}

\begin{proof}
Let \(x_1,\ldots,x_n,y_1,\ldots,y_{n-1}\in\g\),
\(u\in V\), and put
\(X=(x_1,\ldots,x_{n-1})\),
\(Y=(y_1,\ldots,y_{n-1})\).  The \(V\)-components of
\eqref{eq:FI} on \(X,Y,u\) are
\begin{align*}
\operatorname{pr}_V[X,[Y,u]_{\bowprod}]_{\bowprod}
 &=\rho_{\g}(X)\rho_{\g}(Y)u,\\
\operatorname{pr}_V[Y,[X,u]_{\bowprod}]_{\bowprod}
 &=\rho_{\g}(Y)\rho_{\g}(X)u,\\
\operatorname{pr}_V\sum_{i=1}^{n-1}
[y_1,\ldots,[X,y_i]_{\g},\ldots,y_{n-1},u]_{\bowprod}
 &=\sum_{i=1}^{n-1}
\rho_{\g}(y_1,\ldots,[X,y_i]_{\g},\ldots,y_{n-1})u.
\end{align*}
Thus their equality is \eqref{eq:rep1} for \(\rho_{\g}\).  To obtain
\eqref{eq:rep2}, put \(x_1,\ldots,x_{n-2},u\) in the outer block and
\(y_1,\ldots,y_n\) in the inner block of \eqref{eq:FI}.  The two
\(V\)-components are
\begin{align*}
&-\rho_{\g}(x_1,\ldots,x_{n-2},[y_1,\ldots,y_n]_{\g})u,\\
&-\sum_{i=1}^{n}(-1)^{n-i}
\rho_{\g}(y_1,\ldots,\widehat{y_i},\ldots,y_n)
\rho_{\g}(x_1,\ldots,x_{n-2},y_i)u.
\end{align*}
Their equality, after cancelling the minus sign, is
\eqref{eq:rep2} for \(\rho_{\g}\).

Now take \(u_1,\ldots,u_n,v_1,\ldots,v_{n-1}\in V\) and \(x\in\g\).
Put \(U=(u_1,\ldots,u_{n-1})\) and
\(W=(v_1,\ldots,v_{n-1})\).  The \(\g\)-components of
\eqref{eq:FI} on \(U,W,x\) are
\begin{align*}
\operatorname{pr}_{\g}[U,[W,x]_{\bowprod}]_{\bowprod}
 &=\rho_V(U)\rho_V(W)x,\\
\operatorname{pr}_{\g}[W,[U,x]_{\bowprod}]_{\bowprod}
 &=\rho_V(W)\rho_V(U)x,\\
\operatorname{pr}_{\g}\sum_{i=1}^{n-1}
[v_1,\ldots,[U,v_i]_V,\ldots,v_{n-1},x]_{\bowprod}
 &=\sum_{i=1}^{n-1}
\rho_V(v_1,\ldots,[U,v_i]_V,\ldots,v_{n-1})x.
\end{align*}
This is \eqref{eq:rep1} for \(\rho_V\).  Finally, put
\(u_1,\ldots,u_{n-2},x\) in the outer block and
\(v_1,\ldots,v_n\) in the inner block of \eqref{eq:FI}.  The
\(\g\)-components give
\begin{align*}
&-\rho_V(u_1,\ldots,u_{n-2},[v_1,\ldots,v_n]_V)x\\
&\quad=-\sum_{i=1}^{n}(-1)^{n-i}
\rho_V(v_1,\ldots,\widehat{v_i},\ldots,v_n)
\rho_V(u_1,\ldots,u_{n-2},v_i)x.
\end{align*}
After cancellation, this is \eqref{eq:rep2} for \(\rho_V\).  Hence both maps are
representations.
\end{proof}

\begin{definition}\label{def:extreme-factorization}
An \(n\)-Lie algebra \(E\) \emph{factorizes extremely through}
\(\g\) and \(V\) if \(\g,V\subseteq E\) are \(n\)-Lie subalgebras,
\(E=\g\oplus V\) as vector spaces, and
\[
 [u_1,\ldots,u_k,x_{k+1},\ldots,x_n]_E=0
 \quad(2\leq k\leq n-2).
\]
\end{definition}

\begin{theorem}[Factorization theorem]\label{thm:factorization}
An \(n\)-Lie algebra \(E\) factorizes extremely through \(\g\) and
\(V\) if and only if there is an extreme matched pair such that
\(E\cong\g\bowprod V\), by an isomorphism which is the identity on both
factors.
\end{theorem}

\begin{proof}
Assume first that \(E=\g\oplus V\) is an extreme factorization.  For
\(x,x_i\in\g\) and \(u,u_i\in V\), use the two projections to define
\[
E=\g\oplus V,\qquad
\operatorname{pr}_{\g}+\operatorname{pr}_V=\id_E.
\]
\begin{align*}
\vartheta_V(u)(x_2,\ldots,x_n)
 &=\operatorname{pr}_{\g}[u,x_2,\ldots,x_n]_E,\\
\rho_{\g}(x_2,\ldots,x_n)u
 &=(-1)^{n-1}\operatorname{pr}_V[u,x_2,\ldots,x_n]_E,\\
\rho_V(u_1,\ldots,u_{n-1})x
 &=\operatorname{pr}_{\g}[u_1,\ldots,u_{n-1},x]_E,\\
\vartheta_{\g}(x)(u_1,\ldots,u_{n-1})
 &=\operatorname{pr}_V[u_1,\ldots,u_{n-1},x]_E.
\end{align*}
Now take arbitrary \(x_1+u_1,\ldots,x_n+u_n\in E\).
Multilinearity, together with the vanishing of every middle mixed
degree, gives
\begin{align*}
&[x_1+u_1,\ldots,x_n+u_n]_E\\
&=[x_1,\ldots,x_n]_{\g}+[u_1,\ldots,u_n]_V\\
&\quad+\sum_{i=1}^{n}(-1)^{i-1}\vartheta_V(u_i)
 (x_1,\ldots,\widehat{x_i},\ldots,x_n)\\
&\quad+\sum_{i=1}^{n}(-1)^{n-i}\rho_V
 (u_1,\ldots,\widehat{u_i},\ldots,u_n)x_i\\
&\quad+\sum_{i=1}^{n}(-1)^{n-i}\rho_{\g}
 (x_1,\ldots,\widehat{x_i},\ldots,x_n)u_i\\
&\quad+\sum_{i=1}^{n}(-1)^{n-i}\vartheta_{\g}(x_i)
 (u_1,\ldots,\widehat{u_i},\ldots,u_n)\\
&=\mathcal M_{\g}(x_1,u_1;\ldots;x_n,u_n)
 +\mathcal M_V(x_1,u_1;\ldots;x_n,u_n).
\end{align*}
Applying \eqref{eq:FI} to arbitrary outer elements
\(x_1+u_1,\ldots,x_{n-1}+u_{n-1}\) and inner elements
\(y_1+v_1,\ldots,y_n+v_n\) shows, by
\Cref{thm:matched-direct}, that both
\eqref{eq:matched-direct-g} and \eqref{eq:matched-direct-V} hold.
Therefore the four maps form a matched pair, and the identity on
\(\g\oplus V\) gives \(E\cong\g\bowprod V\).
Conversely, if the defining bracket of \(\g\bowprod V\) evaluated on both summands are subalgebras and their vector-space sum is direct, then the middle mixed brackets vanish by definition.
\end{proof}

\begin{example}[Matched pair with four nonzero actions]
\label{ex:examples-matched}
Let
\[
 \g=\operatorname{span}\{x_1,\ldots,x_n,z\},
 \qquad
 V=\operatorname{span}\{u_0,u_1,\ldots,u_n\},
\]
with
\begin{equation}
 [x_1,\ldots,x_{n-1},x_n]_{\g}=x_n,\qquad
 [u_1,\ldots,u_n]_V=u_n.
\label{eq:examples-matched-pure}
\end{equation}
Define
\begin{align}
 \rho_{\g}(X_0)u_0&=u_0,
&
 \vartheta_V(u_0)(X_0)&=(-1)^{n-1}x_n,
\label{eq:examples-matched-first}\\
 \rho_V(U_0)z&=z,
&
 \vartheta_{\g}(z)(U_0)&=u_n,
\label{eq:examples-matched-second}
\end{align}
and let every unlisted value be zero.  Then
\((\rho_{\g},\rho_V,\vartheta_V,\vartheta_{\g})\) is an extreme
matched pair, and \(E=\g\bowprod V\) has nonzero brackets
\begin{align}
 [x_1,\ldots,x_{n-1},x_n]&=x_n,
\label{eq:examples-matched-final-1}\\
 [u_1,\ldots,u_{n-1},u_n]&=u_n,
\label{eq:examples-matched-final-2}\\
 [x_1,\ldots,x_{n-1},u_0]&=x_n+u_0,
\label{eq:examples-matched-final-3}\\
 [u_1,\ldots,u_{n-1},z]&=z+u_n.
\label{eq:examples-matched-final-4}
\end{align}
\end{example}

\section{Complements and deformation maps}
\label{sec:complements}

Fix an extreme matched pair and write
\(E=\g\bowprod V\).  A subalgebra \(\overline V\subseteq E\) is a
\emph{complement of \(\g\)} if \(E=\g\oplus\overline V\) as vector
spaces.
The graph method below generalizes the deformation-map description of
complements for \(3\)-Lie algebras in~\cite{ZhangExtending}.

For a linear map \(r:V\to\g\), direct substitution of
\(x_i=r(u_i)\) into \eqref{eq:matched-Mg}--\eqref{eq:matched-MV} gives
\begin{align}
G_r(u_1,\ldots,u_n)
={}&[r(u_1),\ldots,r(u_n)]_{\g}
\notag\\
&+\sum_{i=1}^{n}(-1)^{i-1}\vartheta_V(u_i)
 (r(u_1),\ldots,\widehat{r(u_i)},\ldots,r(u_n))
\notag\\
&+\sum_{i=1}^{n}(-1)^{n-i}\rho_V
 (u_1,\ldots,\widehat{u_i},\ldots,u_n)r(u_i),
\label{eq:deformation-Gr}\\
H_r(u_1,\ldots,u_n)
={}&[u_1,\ldots,u_n]_V
\notag\\
&+\sum_{i=1}^{n}(-1)^{n-i}\rho_{\g}
 (r(u_1),\ldots,\widehat{r(u_i)},\ldots,r(u_n))u_i
\notag\\
&+\sum_{i=1}^{n}(-1)^{n-i}\vartheta_{\g}(r(u_i))
 (u_1,\ldots,\widehat{u_i},\ldots,u_n).
\label{eq:deformation-Hr}
\end{align}

\begin{definition}\label{def:deformation-map}
A linear map \(r:V\to\g\) is a \emph{deformation map} of the matched
pair if
\begin{equation}
 G_r(u_1,\ldots,u_n)=r(H_r(u_1,\ldots,u_n))
\label{eq:deformation-equation}
\end{equation}
for all \(u_1,\ldots,u_n\in V\).
\end{definition}

\begin{theorem}[Graph criterion]\label{thm:graph-criterion}
The graph
\[
 \operatorname{Gr}(r)=\{r(u)+u\mid u\in V\}
\]
is an \(n\)-Lie subalgebra of \(\g\bowprod V\) if and only if \(r\) is
a deformation map.  In that case
\begin{equation}
 [u_1,\ldots,u_n]_r:=H_r(u_1,\ldots,u_n)
\label{eq:deformed-bracket}
\end{equation}
is an \(n\)-Lie bracket on \(V\), and
\[
 V_r\longrightarrow\operatorname{Gr}(r),
 \qquad u\longmapsto r(u)+u
\]
is an \(n\)-Lie algebra isomorphism.
\end{theorem}

\begin{proof}
Take \(u_1,\ldots,u_n\in V\).  Substitution of
\(x_i=r(u_i)\) into \eqref{eq:matched-bracket} gives
\begin{align*}
&[r(u_1)+u_1,\ldots,r(u_n)+u_n]_{\bowprod}\\
&=G_r(u_1,\ldots,u_n)+H_r(u_1,\ldots,u_n).
\end{align*}
An element \(a+b\in\g\oplus V\) belongs to the graph precisely when
\(a=r(b)\), where \(a\in\g\) and \(b\in V\).  Hence the displayed
bracket belongs to \(\operatorname{Gr}(r)\) exactly when
\[
G_r(u_1,\ldots,u_n)=r(H_r(u_1,\ldots,u_n)).
\]
This proves the graph criterion.  If it holds, define
\[
\iota_r:V\longrightarrow\operatorname{Gr}(r),\qquad
\iota_r(u)=r(u)+u,\qquad
\iota_r^{-1}=\operatorname{pr}_V|_{\operatorname{Gr}(r)}.
\]
Then, on arbitrary \(u_1,\ldots,u_n\in V\),
\[
\begin{aligned}
[\iota_r(u_1),\ldots,\iota_r(u_n)]_{\bowprod}
&=r(H_r(u_1,\ldots,u_n))+H_r(u_1,\ldots,u_n)\\
&=\iota_r(H_r(u_1,\ldots,u_n))
=\iota_r([u_1,\ldots,u_n]_r),
\end{aligned}
\]
so \(\iota_r\) preserves the bracket.  Finally, apply the Filippov
identity in \(\operatorname{Gr}(r)\) to the outer entries
\(\iota_r(u_1),\ldots,\iota_r(u_{n-1})\) and the inner entries
\(\iota_r(v_1),\ldots,\iota_r(v_n)\).  It gives
\begin{align*}
\iota_r\!\Biggl(&
 [u_1,\ldots,u_{n-1},[v_1,\ldots,v_n]_r]_r\\
&-\sum_{i=1}^{n}[v_1,\ldots,v_{i-1},
 [u_1,\ldots,u_{n-1},v_i]_r,
 v_{i+1},\ldots,v_n]_r\Biggr)=0.
\end{align*}
Since \(\iota_r\) is injective, the expression inside the parentheses
vanishes.  This is the Filippov identity \eqref{eq:FI} for \(V_r\).
\end{proof}

\begin{theorem}[Classification of complements]
\label{thm:complement-classification}
The assignment
\[
 r\longmapsto\operatorname{Gr}(r)
\]
is a bijection from deformation maps to complements of \(\g\) in
\(\g\bowprod V\).  Two complements \(\operatorname{Gr}(r)\) and
\(\operatorname{Gr}(r')\) are isomorphic as \(n\)-Lie algebras if and
only if there is a linear isomorphism \(\sigma:V\to V\) satisfying
\begin{equation}
 \sigma(H_r(u_1,\ldots,u_n))
 =H_{r'}(\sigma(u_1),\ldots,\sigma(u_n))
\label{eq:deformation-equivalence}
\end{equation}
for all \(u_1,\ldots,u_n\in V\).
\end{theorem}

\begin{proof}
For a linear map \(r:V\to\g\), projection onto \(V\) gives
\[
\operatorname{pr}_V|_{\operatorname{Gr}(r)}
:\operatorname{Gr}(r)\xrightarrow{\sim}V,
\qquad
\operatorname{Gr}(r)\cap\g=0,\qquad
\g+\operatorname{Gr}(r)=\g\oplus V.
\]
Conversely, if \(\overline V\) is a complement, then
\(\operatorname{pr}_V|_{\overline V}:\overline V\to V\) is injective
because \(\overline V\cap\g=0\), and it is surjective because
\(\g+\overline V=\g\oplus V\).  Hence it is a linear isomorphism.
For each \(u\in V\), write the unique inverse image as \(r(u)+u\).
Then
\[
\left(\operatorname{pr}_V|_{\overline V}\right)^{-1}(u)
=r(u)+u,\qquad
r=\operatorname{pr}_{\g}\circ
\left(\operatorname{pr}_V|_{\overline V}\right)^{-1},
\qquad
\overline V=\operatorname{Gr}(r).
\]
For arbitrary \(u_1,\ldots,u_n\in V\), the elementwise graph
calculation of \Cref{thm:graph-criterion} shows that
\(\overline V\) is a subalgebra precisely when
\[
[r(u_1)+u_1,\ldots,r(u_n)+u_n]\in\operatorname{Gr}(r),
\]
and the latter condition is exactly \eqref{eq:deformation-equation}.
Thus
\[
\{r\in\Hom(V,\g)\mid G_r=rH_r\}
\xrightarrow[\ \sim\ ]{\,r\mapsto\operatorname{Gr}(r)\,}
\{\overline V\leq\g\bowprod V\mid
 \g\oplus\overline V=\g\bowprod V\}.
\]
For two deformation maps \(r,r'\), the graph isomorphisms are
\[
\iota_r:(V,H_r)\xrightarrow{\sim}\operatorname{Gr}(r),\qquad
\iota_{r'}:(V,H_{r'})\xrightarrow{\sim}\operatorname{Gr}(r').
\]
Conjugating an isomorphism between the two complements by
\(\iota_r\) and \(\iota_{r'}\) gives a linear isomorphism
\(\sigma:V\to V\).  It preserves the brackets on arbitrary
\(u_1,\ldots,u_n\in V\) exactly when
\[
\sigma(H_r(u_1,\ldots,u_n))
=H_{r'}(\sigma(u_1),\ldots,\sigma(u_n)).
\]
This is \eqref{eq:deformation-equivalence}.  Conversely, any
\(\sigma\in\GL(V)\) satisfying this equation yields the required
isomorphism \(\iota_{r'}\sigma\iota_r^{-1}\).
\end{proof}

The graph criterion of \Cref{thm:graph-criterion} can change the
isomorphism type of a complement even in a semidirect product.  The
following family generalizes the ternary complement examples in
\cite{ZhangExtending}.

\begin{example}[Two non-isomorphic complements]
\label{ex:examples-complements}
Let
\[
 \g=\operatorname{span}\{x_1,\ldots,x_n\},\qquad
 [x_1,\ldots,x_n]_{\g}=x_n,
\]
and let \(V=\operatorname{span}\{u_1,\ldots,u_n\}\) be abelian.
Define
\(\rho_{\g}:\bigwedge^{n-1}\g\to\operatorname{End}(V)\) by
\begin{equation}
 \rho_{\g}(x_1,\ldots,x_{n-1})u_n=u_n
\label{eq:examples-complement-representation}
\end{equation}
and set every other basis value equal to zero.  The semidirect product
\(E=\g\ltimes_{\rho_{\g}}V\) has only the two nonzero brackets
\begin{equation}
 [x_1,\ldots,x_{n-1},x_n]=x_n,\qquad
 [x_1,\ldots,x_{n-1},u_n]=u_n.
\label{eq:examples-complement-E}
\end{equation}
Define \(r:V\to\g\) by
\begin{equation}
 r(u_i)=x_i\quad(1\leq i\leq n-1),\qquad r(u_n)=0.
\label{eq:examples-deformation-map}
\end{equation}
Then \(r\) is a deformation map.  Its graph is a non-abelian
complement of \(\g\), and the transported bracket on \(V\) is
\begin{equation}
 [u_1,\ldots,u_n]_r=u_n.
\label{eq:examples-deformed-V}
\end{equation}
Thus the original complement \(V\) and
\(\operatorname{Gr}(r)\) are not isomorphic.
\end{example}

\begin{proof}
For the deformation map, only \(\rho_{\g}\) is nonzero, so for
arbitrary \(v_1,\ldots,v_n\in V\),
\[
\vartheta_V=\rho_V=\vartheta_{\g}=0,\qquad [-,\ldots,-]_V=0,
\]
\begin{equation}
\begin{aligned}
H_r(v_1,\ldots,v_n)
&=\sum_{i=1}^{n}(-1)^{n-i}
 \rho_{\g}(r(v_1),\ldots,\widehat{r(v_i)},\ldots,r(v_n))v_i.
\end{aligned}
\label{eq:examples-semidir-deformation}
\end{equation}
Evaluate these components on the basis generator of
\(\bigwedge^nV\).  Using \(r(u_n)=0\), we get
\[
\begin{aligned}
H_r(u_1,\ldots,u_n)
&=\sum_{i=1}^{n-1}(-1)^{n-i}
 \rho_{\g}(x_1,\ldots,\widehat{x_i},\ldots,x_{n-1},0)u_i\\
&\quad+\rho_{\g}(x_1,\ldots,x_{n-1})u_n=u_n,
\end{aligned}
\]
\[
G_r(u_1,\ldots,u_n)
=[x_1,\ldots,x_{n-1},0]_{\g}=0
=r(u_n)=r(H_r(u_1,\ldots,u_n)).
\]
Since \(\bigwedge^nV\) is one-dimensional and
\[
\bigwedge\nolimits^nV^*
= K \,(u_1^*\wedge\cdots\wedge u_n^*),
\]
this single calculation proves \(G_r=rH_r\) for all
\(v_1,\ldots,v_n\in V\).
\[
G_r=rH_r,\qquad
[u_1,\ldots,u_n]_r=H_r(u_1,\ldots,u_n)=u_n.
\]
The corresponding graph bracket is visible directly in \(E\):
\[
\begin{aligned}
&[x_1+u_1,\ldots,x_{n-1}+u_{n-1},u_n]_E\\
&=[x_1,\ldots,x_{n-1},u_n]_E
=u_n=r(u_n)+u_n\in\operatorname{Gr}(r).
\end{aligned}
\]
Thus \(V\) is abelian whereas \(\operatorname{Gr}(r)\) is not.
Consequently \(V\ncong\operatorname{Gr}(r)\).
\end{proof}

\section*{Acknowledgments}
I thank Professor G. Militaru for remarks on this paper and telling me the history that the extending problem was first introduced and studied in \cite{AgoreMilitaru1,AgoreMilitaru2} and subsequently developed for Lie algebras and Leibniz algebras in \cite{AgoreMilitaru,AgoreMilitaru5}.

Most of the first section \ref{sec:prelim} was done while the author was visiting Courant Research Centre, Georg-August Universit\"{a}t G\"{o}ttingen from June to September, 2013. After that I investigated the extending structures for 3-Lie algebras in 2020 and the results were published in \cite{ZhangExtending}.
Recently, I ask chatgpt to help solving the extending problems for $n$-Lie algebras as I did in \cite{ZhangExtending}.
The answer was given in Appendix A of this paper using the full extending datum without any significance in representation theory or cohomology theory,
which is too complicated for non-experts.
Thus I proposed a route to solve the problem by using the action-form unified product which is the ultimate solution to the problem tackled in this paper.
Section \ref{sec:unified} -- section \ref{sec:complements} were done by chatgpt 5.6  under the direction of the author step by step, but the Appendix A and Examples were all most completely done by chatgpt 5.6.
The author has reviewed and corrected the full manuscript in its entirety and takes full responsibility for the work.

\section*{Data availability}

No data was used for the research described in the article.

\appendix

\section{The unrestricted extending problem}
\label{app:full}

The main text imposes the vanishing of all mixed brackets containing
between two and \(n-2\) entries of \(V\).  We now remove that
hypothesis.  The construction below contains every alternating bracket
on \(\g\oplus V\) whose restriction to \(\g\) is the prescribed
\(n\)-Lie bracket.  It is the full \(n\)-ary counterpart of the unified
product philosophy in \cite{AgoreMilitaru,ZhangExtending}.

\subsection{The full extending datum and its direct expansion}

For \(0\leq k\leq n\), consider maps
\begin{align}
 \Gamma_k^{\g}&:\bigwedge^kV\otimes
 \bigwedge^{n-k}\g\longrightarrow\g,
\notag\\
 \Gamma_k^V&:\bigwedge^kV\otimes
 \bigwedge^{n-k}\g\longrightarrow V,
\label{eq:full-component-maps}
\end{align}
alternating separately in their \(V\)-variables and their
\(\g\)-variables.  The degree-zero components are fixed by
\begin{equation}
 \Gamma_0^{\g}(x_1,\ldots,x_n)=[x_1,\ldots,x_n]_{\g},
 \qquad
 \Gamma_0^V(x_1,\ldots,x_n)=0.
\label{eq:full-degree-zero}
\end{equation}

\begin{definition}\label{def:full-datum}
A \emph{full extending datum} of \(\g\) through \(V\) is the family
\[
 \boldsymbol\Gamma=
 \bigl((\Gamma_k^{\g},\Gamma_k^V)\bigr)_{1\leq k\leq n}
\]
together with the boundary convention
\eqref{eq:full-degree-zero}.  No component with
\(2\leq k\leq n-2\) is required to vanish.
\end{definition}

Fix \(x_i\in\g\) and \(u_i\in V\).  For every ordered set of indices
\[
 1\leq i_1<\cdots<i_k\leq n,
\]
let
\[
 1\leq j_1<\cdots<j_{n-k}\leq n
\]
be its ordered complement.  Moving
\(u_{i_1},\ldots,u_{i_k}\) to the first \(k\) positions requires
\begin{equation}
 (i_1-1)+\cdots+(i_k-k)
 =i_1+\cdots+i_k-\frac{k(k+1)}2
\label{eq:full-transposition-count}
\end{equation}
transpositions.  Define
\begin{align}
&\mathcal G_{\boldsymbol\Gamma}
 (x_1,u_1;\ldots;x_n,u_n)
\notag\\
&=\sum_{k=0}^{n}
 \ \sum_{1\leq i_1<\cdots<i_k\leq n}
 (-1)^{i_1+\cdots+i_k-k(k+1)/2}
\notag\\[-1mm]
&\qquad\cdot
 \Gamma_k^{\g}
 (u_{i_1},\ldots,u_{i_k};
  x_{j_1},\ldots,x_{j_{n-k}}),
\label{eq:full-G}\\
&\mathcal H_{\boldsymbol\Gamma}
 (x_1,u_1;\ldots;x_n,u_n)
\notag\\
&=\sum_{k=0}^{n}
 \ \sum_{1\leq i_1<\cdots<i_k\leq n}
 (-1)^{i_1+\cdots+i_k-k(k+1)/2}
\notag\\[-1mm]
&\qquad\cdot
 \Gamma_k^V
 (u_{i_1},\ldots,u_{i_k};
  x_{j_1},\ldots,x_{j_{n-k}}).
\label{eq:full-H}
\end{align}
For \(k=0\), each inner sum contains the single term fixed by
\eqref{eq:full-degree-zero}; for \(k=n\), the list of \(x\)-arguments
is empty.

\begin{proposition}[Full unified-product bracket]
\label{prop:full-bracket}
The formula
\begin{align}
&[x_1+u_1,\ldots,x_n+u_n]_{\boldsymbol\Gamma}
\notag\\
&\qquad=
 \mathcal G_{\boldsymbol\Gamma}(x_1,u_1;\ldots;x_n,u_n)
 +\mathcal H_{\boldsymbol\Gamma}(x_1,u_1;\ldots;x_n,u_n)
\label{eq:full-bracket}
\end{align}
defines an alternating multilinear operation on \(\g\oplus V\), and
every alternating operation on \(\g\oplus V\) restricting to
\([- ,\ldots,-]_{\g}\) is obtained uniquely in this way.
\end{proposition}

\begin{proof}
Let \(B:\bigwedge^n(\g\oplus V)\to\g\oplus V\) be any alternating
operation extending the bracket of \(\g\), and take
\(x_1,\ldots,x_n\in\g\), \(u_1,\ldots,u_n\in V\).  For
\[
I=\{i_1<\cdots<i_k\}\subseteq\{1,\ldots,n\},\qquad
I^c=\{j_1<\cdots<j_{n-k}\},
\]
\[
\eta(I):=\sum_{a=1}^{k}(i_a-a)
=i_1+\cdots+i_k-\frac{k(k+1)}2.
\]
Multilinearity and the \(\eta(I)\) transpositions that move the
\(V\)-entries to the first \(k\) positions give
\begin{align*}
&B(x_1+u_1,\ldots,x_n+u_n)\\
&=\sum_{k=0}^{n}\ \sum_{\substack{I\subseteq\{1,\ldots,n\}\\|I|=k}}
 (-1)^{\eta(I)}
 B(u_{i_1},\ldots,u_{i_k};
 x_{j_1},\ldots,x_{j_{n-k}}),
\end{align*}
where \(I^c=\{j_1<\cdots<j_{n-k}\}\).  Taking the two direct-sum
components defines
\[
\Gamma_k^{\g}=\operatorname{pr}_{\g}
 B|_{\wedge^kV\otimes\wedge^{n-k}\g},
\qquad
\Gamma_k^V=\operatorname{pr}_{V}
 B|_{\wedge^kV\otimes\wedge^{n-k}\g},
\]
and substitution gives
\[
\operatorname{pr}_{\g}B=\mathcal G_{\boldsymbol\Gamma},
\qquad
\operatorname{pr}_{V}B=\mathcal H_{\boldsymbol\Gamma},
\qquad
B=\mathcal G_{\boldsymbol\Gamma}+\mathcal H_{\boldsymbol\Gamma}.
\]
Restriction to each homogeneous input type recovers the two maps
uniquely:
\[
\Gamma_k^{\g},\Gamma_k^V
=
\left.
(\operatorname{pr}_{\g}B,\operatorname{pr}_{V}B)
\right|_{\wedge^kV\otimes\wedge^{n-k}\g}.
\]
Thus the datum \(\boldsymbol\Gamma\) is unique.
It remains to verify alternation of the reconstructed operation.
Interchange the adjacent arguments
\(x_r+u_r\) and \(x_{r+1}+u_{r+1}\).  If exactly one of
\(r,r+1\) belongs to \(I\), the transposition count changes parity;
if both or neither belongs to \(I\), alternation of
\(\Gamma_k^\bullet\) supplies the sign.  Explicitly,
\[
\tau=(r\ r+1),\qquad
\eta(\tau I)\equiv
\eta(I)+
\begin{cases}
1\pmod2,&|\{r,r+1\}\cap I|=1,\\
0\pmod2,&|\{r,r+1\}\cap I|\in\{0,2\},
\end{cases}
\]
\[
\Gamma_k^\bullet(\ldots,u_r,u_{r+1},\ldots;\mathbf x)
=-\Gamma_k^\bullet(\ldots,u_{r+1},u_r,\ldots;\mathbf x),
\]
\[
\Gamma_k^\bullet(\mathbf u;\ldots,x_r,x_{r+1},\ldots)
=-\Gamma_k^\bullet(\mathbf u;\ldots,x_{r+1},x_r,\ldots),
\qquad \bullet\in\{\g,V\},
\]
so every summand changes sign and hence
\[
[x_1+u_1,\ldots,z_r,z_{r+1},\ldots,x_n+u_n]_{\boldsymbol\Gamma}
=-[x_1+u_1,\ldots,z_{r+1},z_r,\ldots,x_n+u_n]_{\boldsymbol\Gamma}.
\]
\end{proof}

If \eqref{eq:full-bracket} satisfies the Filippov identity, the datum is
called a \emph{full extending structure}, and the resulting algebra is
denoted by
\[
 \g\natprod_{\boldsymbol\Gamma}^{\mathrm{full}}V.
\]

\subsection{A direct complete Filippov criterion}

Take arbitrary elements
\[
 \begin{gathered}
 x_1,\ldots,x_{n-1},y_1,\ldots,y_n\in\g,\qquad
 u_1,\ldots,u_{n-1},v_1,\ldots,v_n\in V,\\
 z_a=x_a+u_a\quad(1\leq a\leq n-1),\qquad
 w_i=y_i+v_i\quad(1\leq i\leq n).
 \end{gathered}
\]
Denote the actual components of the inner brackets in \eqref{eq:FI} by
\begin{align}
 \mathcal G_0&=
 \mathcal G_{\boldsymbol\Gamma}(y_1,v_1;\ldots;y_n,v_n),
&
 \mathcal H_0&=
 \mathcal H_{\boldsymbol\Gamma}(y_1,v_1;\ldots;y_n,v_n),
\label{eq:full-inner-zero}\\
 \mathcal G_i&=
 \mathcal G_{\boldsymbol\Gamma}
 (x_1,u_1;\ldots;x_{n-1},u_{n-1};y_i,v_i),
\notag\\[-1mm]
 \mathcal H_i&=
 \mathcal H_{\boldsymbol\Gamma}
 (x_1,u_1;\ldots;x_{n-1},u_{n-1};y_i,v_i).
\label{eq:full-inner-i}
\end{align}
Thus
\[
 [w_1,\ldots,w_n]_{\boldsymbol\Gamma}
 =\mathcal G_0+\mathcal H_0,\qquad
 [z_1,\ldots,z_{n-1},w_i]_{\boldsymbol\Gamma}
 =\mathcal G_i+\mathcal H_i.
\]

\begin{theorem}[Complete full unified-product criterion]
\label{thm:full-criterion}
The datum \(\boldsymbol\Gamma\) is a full extending structure if and
only if, for all the elements chosen above,
\begin{align}
&\mathcal G_{\boldsymbol\Gamma}
 (x_1,u_1;\ldots;x_{n-1},u_{n-1};
  \mathcal G_0,\mathcal H_0)
\notag\\
&\quad=\sum_{i=1}^{n}
 \mathcal G_{\boldsymbol\Gamma}
 (y_1,v_1;\ldots;y_{i-1},v_{i-1};
  \mathcal G_i,\mathcal H_i;
  y_{i+1},v_{i+1};\ldots;y_n,v_n),
\label{eq:full-FI-g}\\
&\mathcal H_{\boldsymbol\Gamma}
 (x_1,u_1;\ldots;x_{n-1},u_{n-1};
  \mathcal G_0,\mathcal H_0)
\notag\\
&\quad=\sum_{i=1}^{n}
 \mathcal H_{\boldsymbol\Gamma}
 (y_1,v_1;\ldots;y_{i-1},v_{i-1};
  \mathcal G_i,\mathcal H_i;
  y_{i+1},v_{i+1};\ldots;y_n,v_n).
\label{eq:full-FI-V}
\end{align}
Equations \eqref{eq:full-G}--\eqref{eq:full-H} expand every term and
every sign in these two identities.
\end{theorem}

\begin{proof}
The left-hand side of \eqref{eq:FI} is
\begin{align*}
&[z_1,\ldots,z_{n-1},
 [w_1,\ldots,w_n]_{\boldsymbol\Gamma}]_{\boldsymbol\Gamma}\\
&=[x_1+u_1,\ldots,x_{n-1}+u_{n-1},
 \mathcal G_0+\mathcal H_0]_{\boldsymbol\Gamma}\\
&=\mathcal G_{\boldsymbol\Gamma}
 (x_1,u_1;\ldots;x_{n-1},u_{n-1};
  \mathcal G_0,\mathcal H_0)\\
&\quad+\mathcal H_{\boldsymbol\Gamma}
 (x_1,u_1;\ldots;x_{n-1},u_{n-1};
  \mathcal G_0,\mathcal H_0).
\end{align*}
For \(1\leq i\leq n\), the \(i\)-th summand on the right-hand side is
\begin{align*}
&[w_1,\ldots,w_{i-1},
 [z_1,\ldots,z_{n-1},w_i]_{\boldsymbol\Gamma},
 w_{i+1},\ldots,w_n]_{\boldsymbol\Gamma}\\
&=[y_1+v_1,\ldots,y_{i-1}+v_{i-1},
 \mathcal G_i+\mathcal H_i,
 y_{i+1}+v_{i+1},\ldots,y_n+v_n]_{\boldsymbol\Gamma}\\
&=\mathcal G_{\boldsymbol\Gamma}
 (y_1,v_1;\ldots;y_{i-1},v_{i-1};
  \mathcal G_i,\mathcal H_i;
  y_{i+1},v_{i+1};\ldots;y_n,v_n)\\
&\quad+\mathcal H_{\boldsymbol\Gamma}
 (y_1,v_1;\ldots;y_{i-1},v_{i-1};
  \mathcal G_i,\mathcal H_i;
  y_{i+1},v_{i+1};\ldots;y_n,v_n).
\end{align*}
After summing over \(i\), the \(\g\)-component of the difference
between the two sides of \eqref{eq:FI} is
\begin{align*}
&\mathcal G_{\boldsymbol\Gamma}
 (x_1,u_1;\ldots;x_{n-1},u_{n-1};
  \mathcal G_0,\mathcal H_0)\\
&\quad-\sum_{i=1}^{n}
 \mathcal G_{\boldsymbol\Gamma}
 (y_1,v_1;\ldots;y_{i-1},v_{i-1};
  \mathcal G_i,\mathcal H_i;
  y_{i+1},v_{i+1};\ldots;y_n,v_n),
\end{align*}
and its \(V\)-component is
\begin{align*}
&\mathcal H_{\boldsymbol\Gamma}
 (x_1,u_1;\ldots;x_{n-1},u_{n-1};
  \mathcal G_0,\mathcal H_0)\\
&\quad-\sum_{i=1}^{n}
 \mathcal H_{\boldsymbol\Gamma}
 (y_1,v_1;\ldots;y_{i-1},v_{i-1};
  \mathcal G_i,\mathcal H_i;
  y_{i+1},v_{i+1};\ldots;y_n,v_n).
\end{align*}
Since \(\g\oplus V\) is a direct sum, the Filippov difference vanishes
exactly when these two components vanish, which gives
\eqref{eq:full-FI-g} and \eqref{eq:full-FI-V}.

Every homogeneous compatibility identity is recovered by assigning
each of the \(2n-1\) entries independently to \(\g\) or \(V\), then
setting the unused component of that entry equal to zero in
\eqref{eq:full-FI-g}--\eqref{eq:full-FI-V}.  Thus the two displayed
component equations contain the complete, non-reduced criterion.
\end{proof}

\subsection{Realization of arbitrary extending structures}

Let \(E\) be a vector space containing \(\g\), and let
\([- ,\ldots,-]_E\) be an \(n\)-Lie bracket whose restriction to
\(\g\) is the fixed bracket.  Choose a projection
\[
 p:E\longrightarrow\g,
 \qquad p(x)=x\quad(x\in\g),
\]
and put
\[
 V=\Ker p,
 \qquad q=\id_E-p.
\]
Then \(E=\g\oplus V\).  For \(1\leq k\leq n\), define
\begin{align}
&\Gamma_k^{\g}(u_1,\ldots,u_k;
 x_{k+1},\ldots,x_n)
\notag\\
&\qquad=
 p\bigl([u_1,\ldots,u_k,x_{k+1},\ldots,x_n]_E\bigr),
\label{eq:full-extract-g}\\
&\Gamma_k^V(u_1,\ldots,u_k;
 x_{k+1},\ldots,x_n)
\notag\\
&\qquad=
 q\bigl([u_1,\ldots,u_k,x_{k+1},\ldots,x_n]_E\bigr).
\label{eq:full-extract-V}
\end{align}

\begin{theorem}[Full realization theorem]
\label{thm:full-realization}
The maps \eqref{eq:full-extract-g}--\eqref{eq:full-extract-V} form a
full extending structure.  The linear map
\[
 \Phi:\g\natprod_{\boldsymbol\Gamma}^{\mathrm{full}}V
 \longrightarrow E,
 \qquad \Phi(x+u)=x+u,
\]
is an \(n\)-Lie algebra isomorphism fixing \(\g\) pointwise and
inducing the identity on \(E/\g\cong V\).  Conversely, every full
unified product is an extending structure of \(\g\) on
\(\g\oplus V\).
\end{theorem}

\begin{proof}
Let \(x_1,\ldots,x_n\in\g\) and \(u_1,\ldots,u_n\in V\).  Since
\(p|_{\g}=\id_{\g}\) and \(V=\ker p\),
\[
\begin{gathered}
E=\g\oplus V,\qquad p^2=p,\qquad q=\id_E-p,\\
\Phi(x+u)=x+u,\qquad \Phi^{-1}(e)=p(e)+q(e).
\end{gathered}
\]
Moving the selected \(u\)-entries to the front gives
\begin{align*}
&[x_1+u_1,\ldots,x_n+u_n]_E\\
&=\sum_{k=0}^{n}
 \ \sum_{1\leq i_1<\cdots<i_k\leq n}
 (-1)^{i_1+\cdots+i_k-k(k+1)/2}\\
&\qquad\cdot
 [u_{i_1},\ldots,u_{i_k},
 x_{j_1},\ldots,x_{j_{n-k}}]_E.
\end{align*}
Applying \(p\) and \(q\) term by term and using
\eqref{eq:full-extract-g}--\eqref{eq:full-extract-V}, we obtain
\begin{align*}
p[x_1+u_1,\ldots,x_n+u_n]_E
&=\mathcal G_{\boldsymbol\Gamma}
 (x_1,u_1;\ldots;x_n,u_n),\\
q[x_1+u_1,\ldots,x_n+u_n]_E
&=\mathcal H_{\boldsymbol\Gamma}
 (x_1,u_1;\ldots;x_n,u_n),
\end{align*}
and hence, on arbitrary \(x_i+u_i\),
\[
[\Phi(x_1+u_1),\ldots,\Phi(x_n+u_n)]_E
=\Phi([x_1+u_1,\ldots,x_n+u_n]_{\boldsymbol\Gamma}).
\]
Apply \eqref{eq:FI} in \(E\) to the outer elements
\(x_1+u_1,\ldots,x_{n-1}+u_{n-1}\) and the inner elements
\(y_1+v_1,\ldots,y_n+v_n\), and take its two components.  By
\Cref{thm:full-criterion}, these components are exactly
\eqref{eq:full-FI-g} and \eqref{eq:full-FI-V}.  Moreover,
\(\Phi|_{\g}=\id_{\g}\) and
\(\overline{\Phi}=\id_{E/\g}\).

Conversely,
\(\Gamma_0^{\g}=[- ,\ldots,-]_{\g}\) and \(\Gamma_0^V=0\) show
directly that \([\g^n]_{\boldsymbol\Gamma}\subseteq\g\).  If
\eqref{eq:full-FI-g} and \eqref{eq:full-FI-V} hold, their sum is the
Filippov identity on
\(\g\natprod_{\boldsymbol\Gamma}^{\mathrm{full}}V\).
\end{proof}

The theorem shows why the middle components are essential.  For a
general complement \(V\), the values in
\eqref{eq:full-extract-g}--\eqref{eq:full-extract-V} need not vanish
when \(2\leq k\leq n-2\); hence no datum retaining only the two extreme
mixed degrees can represent every extending structure when \(n>3\).

\subsection{Morphisms and classification}

Let \(\boldsymbol\Gamma\) and \(\boldsymbol\Lambda\) be two full
extending structures of \(\g\) through the same vector space \(V\).
Every linear map \(\psi:\g\oplus V\to\g\oplus V\) fixing \(\g\)
pointwise has a unique form
\begin{equation}
 \psi(x+u)=x+r(u)+s(u),
 \qquad r:V\to\g,\qquad s:V\to V.
\label{eq:full-psi}
\end{equation}

\begin{theorem}[Full homomorphism criterion]
\label{thm:full-homomorphism}
The map \eqref{eq:full-psi} is an \(n\)-Lie algebra homomorphism
\[
 \g\natprod_{\boldsymbol\Gamma}^{\mathrm{full}}V
 \longrightarrow
 \g\natprod_{\boldsymbol\Lambda}^{\mathrm{full}}V
\]
if and only if, for all \(x_i\in\g\) and \(u_i\in V\),
\begin{align}
&\mathcal G_{\boldsymbol\Gamma}
 (x_1,u_1;\ldots;x_n,u_n)
 +r\bigl(\mathcal H_{\boldsymbol\Gamma}
 (x_1,u_1;\ldots;x_n,u_n)\bigr)
\notag\\
&\quad=
 \mathcal G_{\boldsymbol\Lambda}
 (x_1+r(u_1),s(u_1);\ldots;
  x_n+r(u_n),s(u_n)),
\label{eq:full-morphism-g}\\
&s\bigl(\mathcal H_{\boldsymbol\Gamma}
 (x_1,u_1;\ldots;x_n,u_n)\bigr)
\notag\\
&\quad=
 \mathcal H_{\boldsymbol\Lambda}
 (x_1+r(u_1),s(u_1);\ldots;
  x_n+r(u_n),s(u_n)).
\label{eq:full-morphism-V}
\end{align}
Moreover, \(\psi\) is an isomorphism if and only if
\(s\in\GL(V)\), and it induces the identity on \(E/\g\) if and only if
\(s=\id_V\).
\end{theorem}

\begin{proof}
Take arbitrary \(x_1,\ldots,x_n\in\g\) and
\(u_1,\ldots,u_n\in V\).  Applying \(\psi\) to their source bracket
gives
\begin{align*}
&\psi([x_1+u_1,\ldots,x_n+u_n]_{\boldsymbol\Gamma})\\
&=\mathcal G_{\boldsymbol\Gamma}
 (x_1,u_1;\ldots;x_n,u_n)\\
&\quad+r\!\left(\mathcal H_{\boldsymbol\Gamma}
 (x_1,u_1;\ldots;x_n,u_n)\right)
 +s\!\left(\mathcal H_{\boldsymbol\Gamma}
 (x_1,u_1;\ldots;x_n,u_n)\right).
\end{align*}
On the other hand, applying \(\psi\) to each element first gives
\begin{align*}
&[\psi(x_1+u_1),\ldots,\psi(x_n+u_n)]_{\boldsymbol\Lambda}\\
&=[x_1+r(u_1)+s(u_1),\ldots,
 x_n+r(u_n)+s(u_n)]_{\boldsymbol\Lambda}\\
&=\mathcal G_{\boldsymbol\Lambda}
 (x_1+r(u_1),s(u_1);\ldots;x_n+r(u_n),s(u_n))\\
&\quad+\mathcal H_{\boldsymbol\Lambda}
 (x_1+r(u_1),s(u_1);\ldots;x_n+r(u_n),s(u_n)).
\end{align*}
The \(\g\)-components of these two expressions agree exactly under
\eqref{eq:full-morphism-g}, and the \(V\)-components agree exactly
under \eqref{eq:full-morphism-V}.  Thus
\begin{align*}
&\psi([x_1+u_1,\ldots,x_n+u_n]_{\boldsymbol\Gamma})\\
&\quad=[\psi(x_1+u_1),\ldots,
        \psi(x_n+u_n)]_{\boldsymbol\Lambda}.
\end{align*}
Conversely, bracket preservation on these arbitrary elements gives
\eqref{eq:full-morphism-g} after comparison in \(\g\), and it gives
\eqref{eq:full-morphism-V} after comparison in \(V\).
Finally, with respect to \(\g\oplus V\),
\[
[\psi]_{\g\oplus V}
=\begin{pmatrix}\id_{\g}&r\\0&s\end{pmatrix},
\qquad
\det[\psi]=\det s,
\qquad
[\psi]^{-1}
=\begin{pmatrix}\id_{\g}&-rs^{-1}\\0&s^{-1}\end{pmatrix}.
\]
If \(s\in\GL(V)\), the displayed inverse proves that \(\psi\) is an
isomorphism.  Conversely, invertibility of \(\psi\) implies
invertibility of its quotient map \(s\).  Finally,
\[
\overline{\psi}:E/\g\to E/\g,\qquad
\overline{\psi}(u+\g)=s(u)+\g.
\]
Thus \(\overline{\psi}=\id\) precisely when \(s=\id_V\).
\end{proof}

The preceding criterion can be written degree by degree without any
implicit permutation sign.  Fix \(1\leq k\leq n\), take
\(u_1,\ldots,u_k\in V\) and
\(x_{k+1},\ldots,x_n\in\g\), and for
\(1\leq i_1<\cdots<i_{\ell}\leq k\), let
\(j_1<\cdots<j_{k-\ell}\) be the complementary indices in
\(\{1,\ldots,k\}\).

\begin{corollary}[Explicit degreewise form]
\label{cor:full-degreewise-morphism}
Equations \eqref{eq:full-morphism-g}--\eqref{eq:full-morphism-V} are
equivalent to the following identities, for every \(1\leq k\leq n\):
\begin{align}
&\Gamma_k^{\g}(u_1,\ldots,u_k;x_{k+1},\ldots,x_n)
 +r\bigl(\Gamma_k^V
 (u_1,\ldots,u_k;x_{k+1},\ldots,x_n)\bigr)
\notag\\
&=\sum_{\ell=0}^{k}
 \ \sum_{1\leq i_1<\cdots<i_{\ell}\leq k}
 (-1)^{i_1+\cdots+i_{\ell}-\ell(\ell+1)/2}
\notag\\[-1mm]
&\quad\cdot
 \Lambda_{\ell}^{\g}
 \bigl(s(u_{i_1}),\ldots,s(u_{i_{\ell}});
 r(u_{j_1}),\ldots,r(u_{j_{k-\ell}}),
 x_{k+1},\ldots,x_n\bigr),
\label{eq:full-degree-morphism-g}\\
&s\bigl(\Gamma_k^V
 (u_1,\ldots,u_k;x_{k+1},\ldots,x_n)\bigr)
\notag\\
&=\sum_{\ell=0}^{k}
 \ \sum_{1\leq i_1<\cdots<i_{\ell}\leq k}
 (-1)^{i_1+\cdots+i_{\ell}-\ell(\ell+1)/2}
\notag\\[-1mm]
&\quad\cdot
 \Lambda_{\ell}^V
 \bigl(s(u_{i_1}),\ldots,s(u_{i_{\ell}});
 r(u_{j_1}),\ldots,r(u_{j_{k-\ell}}),
 x_{k+1},\ldots,x_n\bigr).
\label{eq:full-degree-morphism-V}
\end{align}
For \(\ell=0\), the first target component is
\[
 [r(u_1),\ldots,r(u_k),x_{k+1},\ldots,x_n]_{\g},
\]
and the second target component is zero.
\end{corollary}

\begin{proof}
Fix \(1\leq k\leq n\), take \(u_1,\ldots,u_k\in V\) and
\(x_{k+1},\ldots,x_n\in\g\).  Since
\[
\psi(u_i)=r(u_i)+s(u_i),\qquad \psi(x_j)=x_j,
\]
multilinearity expands the target bracket by the subset of the first
\(k\) positions at which the \(V\)-component \(s(u_i)\) is chosen.
For
\[
I=\{i_1<\cdots<i_{\ell}\}\subseteq\{1,\ldots,k\},\qquad
I^c=\{j_1<\cdots<j_{k-\ell}\},
\]
\[
\eta(I)=\sum_{a=1}^{\ell}(i_a-a)
=i_1+\cdots+i_{\ell}-\frac{\ell(\ell+1)}2.
\]
the number \(\eta(I)\) is the number of transpositions that move those
\(\ell\) entries to the first positions.  Therefore
\begin{align*}
&[\psi(u_1),\ldots,\psi(u_k),x_{k+1},\ldots,x_n]_{\boldsymbol\Lambda}\\
&=\sum_{\ell=0}^{k}
 \ \sum_{1\leq i_1<\cdots<i_{\ell}\leq k}
 (-1)^{\eta(I)}
 \Bigl[
 s(u_{i_1}),\ldots,s(u_{i_{\ell}});\\
&\hspace{50mm}
 r(u_{j_1}),\ldots,r(u_{j_{k-\ell}}),
 x_{k+1},\ldots,x_n
 \Bigr]_{\boldsymbol\Lambda},
\end{align*}
and its two components are
\begin{align*}
\operatorname{pr}_{\g}(\cdots)
&=\sum_{\ell=0}^{k}\sum_{|I|=\ell}(-1)^{\eta(I)}
 \Lambda_{\ell}^{\g}
 (s(u_I);r(u_{I^c}),x_{k+1},\ldots,x_n),\\
\operatorname{pr}_{V}(\cdots)
&=\sum_{\ell=0}^{k}\sum_{|I|=\ell}(-1)^{\eta(I)}
 \Lambda_{\ell}^{V}
 (s(u_I);r(u_{I^c}),x_{k+1},\ldots,x_n).
\end{align*}
The source bracket on the same explicit tuple is
\begin{align*}
&[u_1,\ldots,u_k,x_{k+1},\ldots,x_n]_{\boldsymbol\Gamma}\\
&\quad=\Gamma_k^{\g}
 (u_1,\ldots,u_k;x_{k+1},\ldots,x_n)\\
&\qquad+\Gamma_k^V
 (u_1,\ldots,u_k;x_{k+1},\ldots,x_n).
\end{align*}
After applying \(\psi\), its components are
\begin{align*}
&\operatorname{pr}_{\g}\psi(
[u_1,\ldots,u_k,x_{k+1},\ldots,x_n]_{\boldsymbol\Gamma})\\
&\quad=\Gamma_k^{\g}
 (u_1,\ldots,u_k;x_{k+1},\ldots,x_n)\\
&\qquad+r\!\left(\Gamma_k^V
 (u_1,\ldots,u_k;x_{k+1},\ldots,x_n)\right),\\
&\operatorname{pr}_{V}\psi(
[u_1,\ldots,u_k,x_{k+1},\ldots,x_n]_{\boldsymbol\Gamma})\\
&\quad=s\!\left(\Gamma_k^V
 (u_1,\ldots,u_k;x_{k+1},\ldots,x_n)\right).
\end{align*}
Consequently, bracket preservation on this mixed degree is the
equality
\begin{align*}
&\psi([u_1,\ldots,u_k,x_{k+1},\ldots,x_n]_{\boldsymbol\Gamma})\\
&\quad=[\psi(u_1),\ldots,\psi(u_k),
        x_{k+1},\ldots,x_n]_{\boldsymbol\Lambda},
\end{align*}
Comparison of its \(\g\)-components gives
\eqref{eq:full-degree-morphism-g}; comparison of its \(V\)-components
gives \eqref{eq:full-degree-morphism-V}.
As \(k=0,\ldots,n\) runs through all homogeneous input types,
multilinearity proves equivalence with
\eqref{eq:full-morphism-g}--\eqref{eq:full-morphism-V}.
\end{proof}

Define \(\boldsymbol\Gamma\equiv\boldsymbol\Lambda\) when
\eqref{eq:full-morphism-g}--\eqref{eq:full-morphism-V} hold for some
\(r:V\to\g\) and \(s\in\GL(V)\).  Define
\(\boldsymbol\Gamma\approx\boldsymbol\Lambda\) when they hold with
\(s=\id_V\).  Let \(\mathfrak T_n^{\mathrm{full}}(\g,V)\) denote the
set of all full extending structures.

\begin{theorem}[Classification of arbitrary extending structures]
\label{thm:full-classification}
There are canonical bijections
\begin{align}
 \mathfrak T_n^{\mathrm{full}}(\g,V)/{\equiv}
 &\longrightarrow \Extd_n(E,\g),
\notag\\
 \mathfrak T_n^{\mathrm{full}}(\g,V)/{\approx}
 &\longrightarrow \Extd'_n(E,\g),
\label{eq:full-classification}
\end{align}
where \(\Extd_n(E,\g)\) denotes extending brackets modulo
isomorphisms fixing \(\g\), and \(\Extd'_n(E,\g)\) denotes the finer
quotient in which the induced map on \(E/\g\) is also the identity.
Both bijections send the class of \(\boldsymbol\Gamma\) to the class
of \(\g\natprod_{\boldsymbol\Gamma}^{\mathrm{full}}V\).
\end{theorem}

\begin{proof}
Fix a complement \(V\) of \(\g\) in \(E\).  For arbitrary
\(x_i\in\g\) and \(u_i\in V\), the full realization theorem assigns
a unique full datum \(\boldsymbol\Gamma\) to the bracket on \(E\).
Conversely, each
\(\boldsymbol\Gamma\in\mathfrak T_n^{\mathrm{full}}(\g,V)\)
reconstructs the bracket by \eqref{eq:full-bracket}.

Suppose first that
\(\boldsymbol\Gamma\equiv\boldsymbol\Lambda\).  By definition there
are \(r:V\to\g\) and \(s\in\GL(V)\) satisfying both
\eqref{eq:full-morphism-g} and \eqref{eq:full-morphism-V}.
\Cref{thm:full-homomorphism} then gives an isomorphism
\[
\g\natprod_{\boldsymbol\Gamma}^{\mathrm{full}}V
\cong_{\g}
\g\natprod_{\boldsymbol\Lambda}^{\mathrm{full}}V.
\]
Conversely, every isomorphism fixing \(\g\) has the elementwise form
\(\psi(x+u)=x+r(u)+s(u)\); comparing its two bracket components
recovers the same two identities.  If
\(\boldsymbol\Gamma\approx\boldsymbol\Lambda\), the identical argument
uses \(s=\id_V\), which is precisely the requirement that the induced
map on \(E/\g\) be the identity.  Passing to the corresponding
quotient sets therefore gives
\[
\mathfrak T_n^{\mathrm{full}}(\g,V)/{\equiv}
\xrightarrow{\sim}\Extd_n(E,\g),
\qquad
\mathfrak T_n^{\mathrm{full}}(\g,V)/{\approx}
\xrightarrow{\sim}\Extd'_n(E,\g).
\]
\end{proof}

\subsection{Full crossed products, factorizations, and the action form}

The general non-abelian extension and factorization problems are direct
specializations of the full datum.

\begin{corollary}[Unrestricted non-abelian extensions]
\label{cor:full-crossed-extension}
Let \(\g\) and \(V\) be \(n\)-Lie algebras.  For
\(1\leq k\leq n-1\), let
\[
 \gamma_k:\bigwedge^kV\otimes
 \bigwedge^{n-k}\g\longrightarrow\g
\]
be alternating, and let \(\omega:\bigwedge^nV\to\g\).  Set
\begin{align*}
 \Gamma_k^{\g}&=\gamma_k,& \Gamma_k^V&=0
 &&(1\leq k\leq n-1),\\
 \Gamma_n^{\g}&=\omega,&
 \Gamma_n^V(u_1,\ldots,u_n)&=[u_1,\ldots,u_n]_V.
\end{align*}
Then \eqref{eq:full-bracket} defines a full crossed product precisely
when \eqref{eq:full-FI-g}--\eqref{eq:full-FI-V} hold.  Moreover, every
short exact sequence
\[
 0\longrightarrow\g\longrightarrow E\longrightarrow V
 \longrightarrow0
\]
and every linear section \(\sigma:V\to E\) yield such maps by
\begin{align}
\gamma_k(u_1,\ldots,u_k;x_{k+1},\ldots,x_n)
&=[\sigma(u_1),\ldots,\sigma(u_k),
  x_{k+1},\ldots,x_n]_E,
\label{eq:full-crossed-gamma}\\
\omega(u_1,\ldots,u_n)
&=[\sigma(u_1),\ldots,\sigma(u_n)]_E
 -\sigma([u_1,\ldots,u_n]_V).
\label{eq:full-crossed-omega}
\end{align}
The map \(x+u\mapsto x+\sigma(u)\) identifies the full crossed
product with \(E\).
\end{corollary}

\begin{proof}
Let \(x,x_i\in\g\) and \(u,u_i\in V\).  Since \(\g=\ker p\) is an
ideal and \(p\sigma=\id_V\),
\[
\g=\ker(E\xrightarrow{p}V),\qquad
[\g,E,\ldots,E]_E\subseteq\g,\qquad p\sigma=\id_V.
\]
For every \(1\leq k\leq n-1\), the mixed bracket contains a
\(\g\)-entry, and therefore
\[
p[\sigma(u_1),\ldots,\sigma(u_k),x_{k+1},\ldots,x_n]_E=0.
\]
\[
\Gamma_k^V=0,\qquad
\Gamma_k^{\g}=\gamma_k
\quad(1\leq k\leq n-1).
\]
For the pure \(V\)-tuple \(u_1,\ldots,u_n\), split its bracket into
kernel and section components:
\begin{align*}
[\sigma(u_1),\ldots,\sigma(u_n)]_E
&=\bigl([\sigma(u_1),\ldots,\sigma(u_n)]_E
-\sigma([u_1,\ldots,u_n]_V)\bigr)\\
&\quad+\sigma([u_1,\ldots,u_n]_V)\\
&=\omega(u_1,\ldots,u_n)
+\sigma(\Gamma_n^V(u_1,\ldots,u_n)).
\end{align*}
\[
\Gamma_n^{\g}=\omega,\qquad
\Gamma_n^V=[- ,\ldots,-]_V.
\]
Define \(\Phi(x+u)=x+\sigma(u)\).  The full multilinear expansion on
arbitrary \(x_1+u_1,\ldots,x_n+u_n\) gives
\[
\Phi(x+u)=x+\sigma(u),
\]
\begin{align*}
&[\Phi(x_1+u_1),\ldots,\Phi(x_n+u_n)]_E\\
&=\Phi\!\left(
\mathcal G_{\boldsymbol\Gamma}(x_1,u_1;\ldots;x_n,u_n)
+\mathcal H_{\boldsymbol\Gamma}(x_1,u_1;\ldots;x_n,u_n)
\right).
\end{align*}
Thus \(\Phi\) preserves the bracket.  Applying \eqref{eq:FI}
in \(E\) to the outer elements
\(x_1+\sigma(u_1),\ldots,x_{n-1}+\sigma(u_{n-1})\) and the inner
elements \(y_1+\sigma(v_1),\ldots,y_n+\sigma(v_n)\), and comparing
the two direct-sum components, gives
\eqref{eq:full-FI-g} and \eqref{eq:full-FI-V}.
Conversely, in any full crossed product the inclusion and projection
satisfy
\[
\begin{gathered}
i_{\g}(x)=x+0,\qquad p_V(x+u)=u,\qquad \ker p_V=\g.
\end{gathered}
\]
\end{proof}

\begin{corollary}[Unrestricted factorizations]
\label{cor:full-factorization}
Assume that \(\Gamma_n^{\g}=0\) and that
\(\Gamma_n^V=[- ,\ldots,-]_V\) is an \(n\)-Lie bracket.  Then a full
extending structure makes both \(\g\) and \(V\) subalgebras of
\(\g\natprod_{\boldsymbol\Gamma}^{\mathrm{full}}V\).  Conversely, every
factorization \(E=\g\oplus V\) by two \(n\)-Lie subalgebras determines
unique full mixed components \(\Gamma_k^{\g},\Gamma_k^V\) for
\(1\leq k\leq n-1\), and
\(E\cong\g\natprod_{\boldsymbol\Gamma}^{\mathrm{full}}V\).
\end{corollary}

\begin{proof}
Let \(u_1,\ldots,u_n\in V\) and \(x_1,\ldots,x_n\in\g\).  Under the
stated boundary conditions, their pure brackets are
\[
[u_1,\ldots,u_n]_{\boldsymbol\Gamma}
=\Gamma_n^{\g}(u_1,\ldots,u_n)
 +\Gamma_n^V(u_1,\ldots,u_n)
=[u_1,\ldots,u_n]_V\in V,
\]
\[
[x_1,\ldots,x_n]_{\boldsymbol\Gamma}
=\Gamma_0^{\g}(x_1,\ldots,x_n)
=[x_1,\ldots,x_n]_{\g}\in\g.
\]
Thus both factors are subalgebras.  Conversely, suppose
\(E=\g\oplus V\) is a factorization.  For
\(u_1,\ldots,u_k\in V\) and
\(x_{k+1},\ldots,x_n\in\g\), take the two components
\[
E=\g\oplus V,\qquad \g,V\leq E.
\]
\begin{align*}
\Gamma_k^{\g}
&=\operatorname{pr}_{\g}
 [-,\ldots,-]_E|_{\wedge^kV\otimes\wedge^{n-k}\g},\\
\Gamma_k^V
&=\operatorname{pr}_{V}
 [-,\ldots,-]_E|_{\wedge^kV\otimes\wedge^{n-k}\g},
\qquad 1\leq k\leq n-1.
\end{align*}
The full multilinear expansion on arbitrary \(x_i+u_i\) is
\[
[x_1+u_1,\ldots,x_n+u_n]_E
=\mathcal G_{\boldsymbol\Gamma}
 (x_1,u_1;\ldots;x_n,u_n)
+\mathcal H_{\boldsymbol\Gamma}
 (x_1,u_1;\ldots;x_n,u_n),
\]
Restriction to each mixed degree proves that the datum is unique.
The preceding expansion therefore gives
\(E\cong\g\natprod_{\boldsymbol\Gamma}^{\mathrm{full}}V\).
\end{proof}

Finally, the action form of the main text is recovered without changing
any sign convention.

\begin{proposition}[Extreme-degree truncation]
\label{prop:full-to-action}
Let all full components of degrees \(2,\ldots,n-2\) vanish and make the
identifications
\begin{align}
 \Gamma_1^{\g}(u;x_2,\ldots,x_n)
 &=\vartheta_V(u)(x_2,\ldots,x_n),
\label{eq:full-identify-one-g}\\
 \Gamma_1^V(u;x_2,\ldots,x_n)
 &=(-1)^{n-1}\rho_{\g}(x_2,\ldots,x_n)u,
\label{eq:full-identify-one-V}\\
 \Gamma_{n-1}^{\g}(u_1,\ldots,u_{n-1};x)
 &=\rho_V(u_1,\ldots,u_{n-1})x,
\label{eq:full-identify-last-g}\\
 \Gamma_{n-1}^V(u_1,\ldots,u_{n-1};x)
 &=\vartheta_{\g}(x)(u_1,\ldots,u_{n-1}),
\label{eq:full-identify-last-V}\\
 \Gamma_n^{\g}(u_1,\ldots,u_n)
 &=\omega(u_1,\ldots,u_n),
\label{eq:full-identify-n-g}\\
 \Gamma_n^V(u_1,\ldots,u_n)
 &=\{u_1,\ldots,u_n\}_V.
\label{eq:full-identify-n-V}
\end{align}
Then the full bracket \eqref{eq:full-bracket} is exactly the action-form
bracket \eqref{eq:action-direct-expansion}.  Under these
identifications, \eqref{eq:full-FI-g}--\eqref{eq:full-FI-V} become
\eqref{eq:direct-FI-g}--\eqref{eq:direct-FI-V}.  For \(n=3\), the
interval \(2\leq k\leq n-2\) is empty, so every full datum is already
of action form.
\end{proposition}

\begin{proof}
Take \(x_1,\ldots,x_n\in\g\) and \(u_1,\ldots,u_n\in V\).
All intermediate mixed contributions are zero:
\[
\Gamma_k^{\g}=\Gamma_k^V=0
\quad(2\leq k\leq n-2).
\]
For the term having the unique \(V\)-entry \(u_i\), the index set and
transposition count are
\[
k=1,\quad I=\{i\},\quad
\eta(I)=i-1.
\]
Using \eqref{eq:full-identify-one-g}--\eqref{eq:full-identify-one-V}
gives
\begin{align*}
&(-1)^{i-1}\Gamma_1^{\g}
 (u_i;x_1,\ldots,\widehat{x_i},\ldots,x_n)\\
&=(-1)^{i-1}\vartheta_V(u_i)
 (x_1,\ldots,\widehat{x_i},\ldots,x_n),\\
&(-1)^{i-1}\Gamma_1^V
 (u_i;x_1,\ldots,\widehat{x_i},\ldots,x_n)\\
&=(-1)^{i-1}(-1)^{n-1}
 \rho_{\g}(x_1,\ldots,\widehat{x_i},\ldots,x_n)u_i\\
&=(-1)^{n-i}
 \rho_{\g}(x_1,\ldots,\widehat{x_i},\ldots,x_n)u_i.
\end{align*}
For the term having the unique \(\g\)-entry \(x_i\),
\[
k=n-1,\quad I=\{1,\ldots,\widehat{i},\ldots,n\},\quad
\eta(I)=n-i.
\]
Thus \eqref{eq:full-identify-last-g}--\eqref{eq:full-identify-last-V}
give
\begin{align*}
&(-1)^{n-i}\Gamma_{n-1}^{\g}
 (u_1,\ldots,\widehat{u_i},\ldots,u_n;x_i)\\
&=(-1)^{n-i}\rho_V
 (u_1,\ldots,\widehat{u_i},\ldots,u_n)x_i,\\
&(-1)^{n-i}\Gamma_{n-1}^{V}
 (u_1,\ldots,\widehat{u_i},\ldots,u_n;x_i)\\
&=(-1)^{n-i}\vartheta_{\g}(x_i)
 (u_1,\ldots,\widehat{u_i},\ldots,u_n).
\end{align*}
The two pure degrees are
\[
\begin{gathered}
k=0:\quad
(\Gamma_0^{\g},\Gamma_0^V)=([- ,\ldots,-]_{\g},0),\\
k=n:\quad
(\Gamma_n^{\g},\Gamma_n^V)=(\omega,\{- ,\ldots,-\}_V).
\end{gathered}
\]
Adding these four surviving degrees for the arbitrary elements
\(x_i+u_i\) yields
\[
\mathcal G_{\boldsymbol\Gamma}=G,\qquad
\mathcal H_{\boldsymbol\Gamma}=H.
\]
Hence \eqref{eq:full-bracket} is exactly
\eqref{eq:action-direct-expansion}.
Substitution into the two direct Filippov component identities gives
\eqref{eq:direct-FI-g} from \eqref{eq:full-FI-g}, and
\eqref{eq:direct-FI-V} from \eqref{eq:full-FI-V}.
When \(n=3\), there is no integer \(k\) satisfying
\(2\leq k\leq n-2\), so every full datum is already of action form.
\end{proof}

\subsection{A genuinely non-reduced unified product}

The last family isolates the phenomenon that distinguishes \(n>3\)
from the ternary case.  Its only nonzero mixed bracket has an
intermediate number of complement variables.  It therefore cannot be
encoded by the six action-form maps, although it is covered by the
full unified product in the appendix.

\begin{example}[An intermediate mixed component]
\label{ex:examples-nonreduced}
Fix \(k\) with \(2\leq k\leq n-2\).  Let \(E\) have basis
\(\{p_1,\ldots,p_{n-1},w\}\) and the single nonzero bracket
\begin{equation}
 [p_1,\ldots,p_{n-1},w]=w.
\label{eq:examples-nonreduced-bracket}
\end{equation}
Set
\begin{equation}
 V_k=\operatorname{span}\{p_1,\ldots,p_k\},\qquad
 \g_k=\operatorname{span}\{p_{k+1},\ldots,p_{n-1},w\}.
\label{eq:examples-nonreduced-splitting}
\end{equation}
Then \(E=\g_k\oplus V_k\), and both factors are abelian \(n\)-Lie
algebras.  Relative to this splitting, the only nonzero full extending
component is
\begin{equation}
 \Gamma_k^{\g_k}
 (p_1,\ldots,p_k;p_{k+1},\ldots,p_{n-1},w)=w.
\label{eq:examples-nonreduced-Gamma}
\end{equation}
Consequently \(E\) is a full unified product but is not an action-form
unified product for the splitting
\eqref{eq:examples-nonreduced-splitting}.
\end{example}

\begin{proof}
Take the operator-block data
\[
 P=\operatorname{span}\{p_1,\ldots,p_{n-1}\},\qquad
W=\operatorname{span}\{w\},
\]
\[
[p_1,\ldots,p_{n-1},w]
=\nu(p_1,\ldots,p_{n-1})(w)=w.
\]
Because both summands in
\eqref{eq:examples-nonreduced-splitting} have dimension less than
\(n\),
\[
\dim V_k=k<n,\qquad
\dim\g_k=n-k<n,
\]
\[
\bigwedge\nolimits^nV_k=0,\qquad
\bigwedge\nolimits^n\g_k=0,\qquad
[V_k^n]=0=[\g_k^n].
\]
On the unique nonzero basis tuple, however, the two projections are
\[
\operatorname{pr}_{\g_k}
[p_1,\ldots,p_k,p_{k+1},\ldots,p_{n-1},w]=w,
\qquad
\operatorname{pr}_{V_k}[p_1,\ldots,p_{n-1},w]=0,
\]
and therefore the only nonzero full component is
\[
\Gamma_k^{\g_k}
(p_1,\ldots,p_k;p_{k+1},\ldots,p_{n-1},w)=w,
\qquad
\Gamma_j^{\g_k}=\Gamma_j^{V_k}=0\quad(j\ne k).
\]
Since \(2\leq k\leq n-2\), the action form would force this precise
mixed degree to vanish:
\[
[V_k^{\,k},\g_k^{\,n-k}]_{\mathrm{action}}=0
\ne[V_k^{\,k},\g_k^{\,n-k}]_E,
\]
which proves that
\(E=\g_k\natprod_{\boldsymbol\Gamma}^{\mathrm{full}}V_k\), while no
action datum \(\Omega\) can satisfy
\(E=\g_k\natprod_{\Omega}V_k\).
\end{proof}

\begin{example}[Smallest non-reduced degree]
\label{ex:examples-four-nonreduced}
For \(n=4\), take \(k=2\).  Then
\[
 E=\operatorname{span}\{p_1,p_2,p_3,w\},\qquad
 [p_1,p_2,p_3,w]=w,
\]
and
\[
 V_2=\operatorname{span}\{p_1,p_2\},\qquad
 \g_2=\operatorname{span}\{p_3,w\}.
\]
Both factors are abelian, while
\([V_2,V_2,\g_2,\g_2]\neq0\).  This is the first mixed degree absent
from the ternary theory.
\end{example}

\begin{remark}\label{rem:examples-summary}
The examples separate the roles of the structure maps.  In
\Cref{ex:examples-six-map} all six action-form maps may be nonzero.
\Cref{ex:examples-crossed} retains
\(\vartheta_V,\rho_V,\omega\) and realizes a non-abelian extension.
\Cref{ex:examples-matched} retains all four actions and realizes a
factorization.  \Cref{ex:examples-complements} uses the representation
\(\rho_{\g}\) and a deformation map to produce two complement types.
Finally, \Cref{ex:examples-nonreduced} is controlled by the
intermediate component \(\Gamma_k^{\g_k}\), which is the precise reason
that the full theory of \Cref{app:full} is needed when \(n>3\).
\end{remark}

\end{document}